%% file: art.tex
\documentclass[oneside,reqno,11pt,a4paper]{amsart}
\usepackage[T1]{fontenc}
\usepackage[margin=2.5cm]{geometry}
\usepackage{amsmath,amsfonts,amssymb,amscd,amsthm}
\usepackage[usenames,dvipsnames,svgnames]{xcolor}
\usepackage[utf8x]{inputenc}
\usepackage{graphicx}
\usepackage{grffile}
\graphicspath{{figures/}{figures/0_MC_Vs_MM/}{figures/1_N_Vs_M/}{figures/2_S_Vs_W/}{figures/3_G_Trans/}{figures/4_Extrap/}{figures/5_Spectra/}}
\usepackage[labelfont=rm,width=0.9\textwidth]{caption}
\usepackage{subcaption}
\usepackage[numbers,sort&compress]{natbib}
\usepackage{doi}
\usepackage{autonum}
\usepackage{makecell}
\usepackage{hyperref}
\usepackage{acronym}
\usepackage{listings}
\usepackage{textgreek}
\usepackage{url}
\usepackage{color}
\usepackage{framed}
\usepackage{verbatim}
\usepackage{fancyvrb}
\usepackage{newtxtext}
\usepackage{newtxmath}
\usepackage{microtype} 
\usepackage[final]{pdfpages}
\usepackage{tikz}
\usepackage{booktabs}
\usepackage{multirow}

\hypersetup{
    breaklinks=true,
    bookmarksopen=true,
    pdftitle={Robust high-order unfitted finite elements by interpolation-based discrete extension},    
    pdfauthor={Santiago Badia, Author Two},     
    colorlinks=true,       
    linkcolor=black,          
    citecolor=blue,        
    filecolor=black,      
    urlcolor=blue           
}

\definecolor{bg}{rgb}{0.93,0.93,0.93}

\newtheorem{theorem}{Theorem}[section]

\newtheorem{proposition}[theorem]{Proposition}
\newtheorem{corollary}[theorem]{Corollary}
\newtheorem{definition}[theorem]{Definition}

\newtheorem{assumption}[theorem]{Assumption}
\newtheorem{remark}[theorem]{Remark}

\acrodef{pde}[PDE]{partial differential equation}
\acrodef{fe}[FE]{finite element}
\acrodef{fem}[FEM]{finite element method}
\acrodef{xfem}[XFEM]{extended finite element method}
\acrodef{dof}[DOF]{degree of freedom}
\acrodefplural{dof}[DOFs]{degrees of freedom}
\acrodef{agfe}[AgFE]{aggregated finite element}
\acrodef{agfem}[AgFEM]{aggregated finite element method}
\acrodef{cg}[CG]{continuous Galerkin}
\acrodef{dg}[DG]{discontinuous Galerkin}
\acrodef{bc}[BC]{boundary condition}
\acrodef{aabb}[AABB]{axis-aligned bounding box}
\acrodefplural{aabb}[AABBs]{axis-aligned bounding boxes}
\acrodef{csg}[CSG]{constructive solid geometry}
\acrodef{gp}[GP]{ghost penalty}

\newcommand{\tnor}[1]{\| #1 \|}

\definecolor{shadecolor}{gray}{.92}
\definecolor{incolor}{rgb}{0,0,.7}
\definecolor{outcolor}{rgb}{.65,0,0}
\definecolor{syntaxcolor}{rgb}{.65,0,0}

\newcommand{\new}[1]{{{#1}}}

\definecolor{FigBlue}{RGB}{0,0,255}
\definecolor{FigCyan}{RGB}{0,255,255}
\definecolor{FigPurple}{RGB}{222,135,170}
\definecolor{FigYellow}{RGB}{255,221,85}
\definecolor{FigRed}{RGB}{255,0,0}

\begin{document}

\title[Robust high-order unfitted finite elements]{Robust high-order unfitted finite elements \\ by interpolation-based discrete extension}
\author[S. Badia]{Santiago Badia$^{1,2}$}

\author[E. Neiva]{Eric Neiva$^{2,3,*}$}

\author[F. Verdugo]{Francesc Verdugo$^{2}$}

\thanks{\null\\
$^{1}$ School of Mathematics, Monash University, Clayton, Victoria, 3800, Australia.\\
$^{2}$ Centre Internacional de M\`etodes Num\`erics a l'Enginyeria, Esteve Terrades 5, 08860 Castelldefels, Spain.\\
$^{3}$ Center for Interdisciplinary Research in Biology (CIRB), College de France, CNRS, INSERM, Université PSL, Paris, France.\\
$^*$ Corresponding author.\\
E-mails: {\tt santiago.badia@monash.edu} (SB) 
{\tt eric.miranda-neiva@college-de-france.fr} (EN)
{\tt fverdugo@cimne.upc.edu} (FV)
}

\date{\today}

\begin{abstract}
In this work, we propose a novel formulation for the solution of partial differential equations using finite element methods on unfitted meshes. The proposed formulation relies on the discrete extension operator proposed in the aggregated finite element method. This formulation is robust with respect to the location of the boundary/interface within the cell. One can prove enhanced stability results, not only on the physical domain, but on the whole active mesh. However, the stability constants grow exponentially with the polynomial order being used, since the underlying extension operators are defined via extrapolation. To address this issue, we introduce a new variant of aggregated finite elements, in which the extension in the physical domain is an interpolation for polynomials of order higher than two. As a result, the stability constants only grow at a polynomial rate with the order of approximation. We demonstrate that this approach enables robust high-order approximations with the aggregated finite element method. \new{The proposed method is consistent, optimally convergent, and with a condition number that scales optimally for high order approximation.}
\end{abstract}

\maketitle


\input{content}

\section*{Acknowledgments}

\newcommand{\thethanks}{This research was partially funded by the Australian Government through the Australian Research Council (project number DP210103092), the European Commission under the FET-HPC ExaQUte project (Grant agreement ID: 800898) within the Horizon 2020 Framework Programme and the project RTI2018-096898-B-I00 from the ``FEDER/Ministerio de Ciencia e Innovación (MCIN) – Agencia Estatal de Investigación (AEI)''. F. Verdugo acknowledges support from the ``Severo Ochoa Program for Centers of Excellence in R\&D (2019-2023)'' under the grant CEX2018-000797-S funded by MCIN/AEI/10.13039/501100011033. This work was also supported by computational resources provided by the Australian Government through NCI under the National Computational Merit Allocation Scheme.}

\thethanks

\setlength{\bibsep}{0.0ex plus 0.00ex}
\bibliographystyle{myabbrvnat}
\bibliography{refs} 
  
\end{document}

%% file: content.tex
\section{Introduction}\label{sec:introduction}

Numerical simulations with \emph{standard} \ac{fe} methods bind together the computational mesh and the geometry of the physical problem. However, body-fitted unstructured mesh generation often requires manual intervention and does not scale properly on distributed platforms. This computational bottleneck becomes especially severe when modeling moving boundaries or interfaces. On the other hand, \emph{unfitted} \ac{fe} methods decouple the mesh from the geometry. The main idea is to embed the physical domain into a geometrically simple background grid (usually a uniform or an adaptive Cartesian grid). In this way, the computational mesh can be generated and partitioned much more efficiently. Similarly, they can easily track embedded interfaces. As a result, they are becoming increasingly attractive in applications with moving interfaces~\cite{Waisman2013,alauzet2016nitsche,Massing2015,kirchhart2016analysis,Badia2020Dec} and in applications with varying domains, such as shape or topology optimisation~\cite{Burman2018}, additive manufacturing~\cite{neiva2020numerical,carraturo2020modeling}, and stochastic geometry problems~\cite{badia2019embedded}. In the numerical community, this family of methods is known by different names, e.g., \emph{unfitted}, \emph{embedded}, or \emph{immersed}. 

Despite circumventing the mesh generation bottleneck, naive unfitted methods are prone to numerical instabilities and severe ill-conditioning~\cite{DePrenter2017,Badia2018}. In the case of unfitted boundaries, the intersection of a background cell with the physical domain can be arbitrarily small and with unbounded aspect ratio. This leads to the so-called \emph{small cut cell problem}: basis functions of the standard finite element space, defined in the background (unfitted) mesh, can have arbitrarily small support in the physical domain. This support depends on the intersection between the background mesh and the boundary (or interface), which in general cannot be controlled. This problem is also present on unfitted interfaces with a high contrast of physical properties~\cite{Neiva2021}. \new{Some works try to circumvent the ill-conditioning of the system using tailored preconditioning strategies (see \cite{DePrenter2017,badia_robust_2017}).}

There is ample literature on how to mitigate the small cut cell problem~\cite{Kummer2017,lehrenfeld2016high,guzman2017finite,li2019shifted}. \new{One of the most popular approaches to solve this issue is the so-called finite-cell method \cite{elhaddad2018multi,xu2016tetrahedral,Jomo2018,Hubrich_2019,Schillinger2015,dauge_theoretical_2015}. The finite cell stabilisation adds a non-consistent penalisation to ensure robustness. In order to preserve optimal convergence rates, the penalty coefficient must be of the order of $h^{2p-1}$ ($h$ is a characteristic mesh size and $p$ the order of approximation). However, the condition number scales suboptimally as $h^{-(2p-1)}$. The finite cell method cannot provide both optimal convergence and condition number bounds. We refer to \cite{LARSSON2022114792} for more details. The removal of ill-conditioned basis functions \cite{Elfverson2018} suffers the same problem as the finite cell method. While the method improves the condition number with respect to cut locations, the resulting condition number is in general worse than $\mathcal{O}(h^{-2})$ (see \cite{Elfverson2018}).}

Few formulations achieve both full robustness and optimality, independent of cut location and material contrast. Among them we have the so-called \emph{\ac{gp}}~\cite{burman2010ghost,burman_cutfem_2015} methods. These schemes were originally motivated for $\mathcal{C}^{0}$ finite element spaces on simplicial meshes and later used in combination with discontinuous Galerkin formulations, \new{or with B-spline basis functions \cite{HOANG2019421}}. Typically, \ac{gp} terms act on the jumps of derivatives across facets cutting the unfitted boundary.

An alternative way to address ill-conditioning due to small cuts is via \emph{cell aggregation} or \emph{cell agglomeration} techniques. This approach blends well with \ac{dg} methods, because \new{\ac{dg} schemes can be easily formulated on agglomerated meshes~\cite{Johansson2013,helzel_high-resolution_2005}.} To ensure robustness, it suffices that each cell (now an aggregate of cells) has \emph{enough} support in the interior of the domain~\cite{muller2017high}. However, accommodating this strategy to \ac{cg} methods is more involved. The key point is how to retain $\mathcal{C}^{0}$-continuity after cell agglomeration.

This issue was addressed in~\cite{Badia2018}, leading to the so-called \ac{agfem}. The main idea of \ac{agfem} is to build a \emph{discrete extension operator} from well-posed \acp{dof} (i.e., the ones related to shape functions with enough support in the domain interior) to ill-posed \acp{dof} (i.e., the ones with small support) that preserves $\mathcal{C}^{0}$-continuity. As a result, basis functions associated with badly cut cells are removed and the ill-conditioning issues solved. Underlying the construction of the discrete extension operator, there are two ingredients: (1) an easy-to-implement and general cell aggregation scheme and (2) a map assigning every geometrical entity (e.g., vertex, edge, face) to one of the aggregates containing it. \new{We also note that basis extensions are also used in the context of spline approximations (see, e.g., \cite{chu_stabilization_2022}).}

\ac{agfem} enjoys good numerical properties, such as stability, condition number bounds, optimal convergence, and continuity with respect to data; detailed mathematical analysis of the method is included in~\cite{Badia2018} for elliptic problems and in~\cite{Badia2018a} for the Stokes equation. \ac{agfem} is also amenable to arbitrarily complex 3D geometries, distributed implementations for large scale problems~\cite{Verdugo2019}, error-driven $h$-adaptivity and parallel tree-based meshes~\cite{Badia2020Jun}, explicit time-stepping for the wave equation~\cite{burman2022explicit} and elliptic interface problems with high contrast~\cite{Neiva2021}. Furthermore, a \emph{weak} version of \ac{agfem} can be formulated as a ghost penalty method which penalises the distance between the solution and its aggregation-based discrete extension~\cite{badia2021linkAgFEM}.

High-order unfitted \acp{fem} are less common~\cite{Schillinger2015,burman2021unfitted}. There are some works that deal with the geometrical aspects related to high-order geometrical approximation on unfitted meshes, e.g., using high-order geometrical maps~\cite{lehrenfeld2016high} \new{or advanced numerical quadratures on cut cells~\cite{elhaddad2018multi,KUDELA2016406}}. \ac{dg} schemes with agglomeration can readily be applied to high-order approximations~\cite{muller2017high,saye2017implicit}. Ghost penalty is conceptually applicable to high-order polynomial approximations, even though it comes with a cost: High order derivative jumps must be penalised on facets. The application of \ac{gp} to high order has been explored, e.g., in~\cite{Hansbo2017,larson2020stabilization}. Even though the algorithm is theoretically robust at high orders, the penalty term for a $p$-th order of approximation is penalising $\mathcal{C}^{p+1}$ continuity on the boundary, which weakly enforces an extension of the values from interior to cut cells. Thus, if the ${1, \ldots, p+1}$ derivatives on cell boundaries are very high, then these penalties promote the amplification of rounding errors. \new{Other interesting approaches are the $\phi$-FEM method in \cite{doi:10.1137/19M1248947} and the high-order version of the shifted boundary method in \cite{ATALLAH2022114885}.}

\acp{agfem} rely on an extension operator from well-posed to ill-posed \acp{dof}. The stability and convergence properties of the algorithm have already been proved for high-order schemes in~\cite{Badia2018}. However, the bases being used to describe the Ag\ac{fe} space are determined by an extension that relies on extrapolation. As a result, for high-order schemes, some of the building blocks of the algorithm (e.g., the constraint computations) involve huge coefficients, possible cancellation, and promote the amplification of rounding errors. \new{We note that, in the context of splines, the weighted extended B-spline method is affected by the same issue (see, e.g., \cite[Th 4.2]{10.2307/3062005}).} This problem is related to the basis being used, not the method itself. \emph{In this work, we address poor performance of extrapolation-based high-order \ac{agfem} by designing a new \ac{fe} basis, such that the discrete extension operator relies on a first order extrapolation of first order polynomials (which are harmless) and an interpolatory extension in the physical domain for high-order terms.} The new basis is derived from standard hierarchical modal $\mathcal{C}^0$ basis functions~\cite{karniadakis2013spectral}, the same ones used in $p$-\ac{fem}~\cite{szabo1991finite}. It leads to \ac{agfem} methods with much better robustness in the high-order regime, compared to purely extrapolative ones. \new{The resulting unfitted \ac{fem} is consistent, optimal and robust with respect to cut locations and order of approximation. The condition number of the resulting system optimally scales as $\mathcal{O}(h^{-2})$. One can alternatively use the proposed extension in a CutFEM-like setting, using the ideas in \cite{badia2021linkAgFEM} to end up with a weakly consistent but still optimal method.} 

This work is structured as follows. First, we introduce the problem and some notation that include a thorough description of the geometrical sets that are required to implement \ac{agfem}. Next, we introduce the purely extrapolative \emph{strong} version of \ac{agfem} and discuss its limitations for high-order approximations. After that, we propose and analyse a novel variant of a modal $\mathcal{C}^0$ basis for the interior \ac{fe} space. The key property of the new basis is that its discrete extension operator is interpolatory in the physical domain. We provide a very detailed numerical experimentation in terms of convergence rates of the $L^2$ and $H^1$ error norms and condition number bounds. We address efficient numerical integration at the cut cells and adequate static condensation techniques for high-order \ac{agfem}. We consider \emph{strong} \ac{agfem} approximations of the Poisson and linear elasticity problems on uniform meshes, as well as different approximation orders, geometries, and intersection locations. We show that the new formulation solves the ill-conditioning issues related to purely extrapolative approaches. 
The original contributions of the article are:
\begin{enumerate}
  \item A novel discrete extension operator that relies on interpolation in the physical space for high orders of approximation to build robust and convergent Ag\ac{fe} spaces on unfitted meshes;
  \item A numerical analysis of the proposed approach, which proves the continuity of this operator uniformly with respect to the polynomial order in the physical space;
  \item A thorough numerical experimentation addressing numerical integration and static condensation of high-order Ag\ac{fe} systems and demonstrating the clear superiority of this approach with respect to purely extrapolative ones, in terms of optimality, accuracy and condition numbers.
\end{enumerate}

\section{Unfitted finite elements}\label{sec:unfitted-fem}

\subsection{Rationale}\label{sec:geom-setup}

Let $\Omega \subset \mathbb{R}^d$, $d$ being the space dimension, be an open bounded Lipschitz domain. $\Omega$ represents the physical domain of the \ac{pde} problem. Standard \ac{fe} methods are formulated in a so-called \emph{body-fitted} mesh, which is a partition of $\Omega$ (or of an approximation of it). In general, we leverage unstructured mesh generation algorithms to create the \emph{body-fitted} mesh of the domain.

Unfitted discretisation techniques, by contrast, decouple the computational mesh from the physical domain. Instead of relying on a \emph{body-fitted} mesh, they embed the physical domain into an arbitrary, but simple artificial domain $\Omega_{h}^{\mathrm{art}}$, such that $\Omega \subset \Omega_{h}^{\mathrm{art}}$. The artificial domain can be trivial, e.g., a bounding box of $\Omega$. The key step in unfitted methods is to discretise $\Omega_{h}^{\mathrm{art}}$, instead of $\Omega$. Thus, the geometrical discretisation is much simpler (and cheaper) than a body-fitted partition of $\Omega$.

Simplifying the discretisation step, in turn, complicates the functional discretisation. Standard (body-fitted) \ac{fe} methods cannot be straightforwardly used. First, strong imposition of Dirichlet boundary conditions assumes the mesh is body-fitted; in an embedded setting, Dirichlet boundary conditions are weakly imposed, instead. Second, cell-wise integration of \ac{fe} forms is more involved; local integration must be performed on the intersection between cells and $\Omega$, only. Third, naive \ac{fe} discretisations can be arbitrarily ill-posed.

\subsection{Geometrical discretisation}\label{sec:cell-aggregation}

Let $\mathcal{T}_{h}$ be a conforming, quasi-uniform and shape-regular background partition of $\Omega_{h}^{\mathrm{art}}$. We represent with $h_T$ the diameter of a cell $T \in \mathcal{T}_{h}$ and the characteristic mesh size is $h \doteq \max_{T \in \mathcal{T}_{h}} h_T$. We introduce next some geometrical definitions in order to define unfitted \ac{fe} discretisations.

First, we let $\{\mathcal{T}_{h}^{\mathrm{act}}, \mathcal{T}_{h}^{\mathrm{out}}\}$ denote a partition of $\mathcal{T}_{h}$ into $\Omega$-\emph{active} and $\Omega$-\emph{exterior} cells. Exterior cells are those with null intersection with $\Omega$. Since they do not play any role in the functional discretisation, they can be discarded. Conversely, the active mesh $\mathcal{T}_{h}^{\mathrm{act}} = \mathcal{T}_{h} \setminus \mathcal{T}_{h}^{\mathrm{out}}$ refers to the subset of cells with non-null intersection with $\Omega$, i.e., those relevant to the functional discretisation. A simple unfitted method, such as the \ac{xfem}~\cite{sukumar_modeling_2001}, formulates the discrete problem on a standard \ac{fe} space on $\mathcal{T}_{h}^{\mathrm{act}}$. Nonetheless, this approach is prone to severe ill-conditioning (see discussion below). This problem, widely known as the \emph{small cut cell problem}, is caused by cut cells with arbitrarily small support on $\Omega$. To deal with this issue, we consider a further partition $\{\mathcal{T}_{h}^{\mathrm{in}}, \mathcal{T}_{h}^{\mathrm{cut}}\}$ of $\mathcal{T}_{h}^{\mathrm{act}}$ into $\Omega$-\emph{interior} and $\Omega$-\emph{cut} cells, see Fig.~\ref{fig:geodefs_a}. Strictly, one would need to only isolate cut cells with small support on $\Omega$ from the rest of active cells. Regardless of this choice, the following discussion applies verbatim. We let the interior of the closure of $\bigcup_{T \in \mathcal{T}_{h}^{\#}} T$ be represented by $\Omega_h^{\#}$, for $\# \in \left\{ \mathrm{act}, \mathrm{in}, \mathrm{cut} \right\}$.

\begin{figure}[ht!]
\centering

\begin{tabular}{cccccc}
\tikz{\draw[fill=FigYellow] rectangle(1.5ex,1.5ex);} $\in\mathcal{T}^{\mathrm{in}}_h$ &
\tikz{\draw[fill=FigCyan] rectangle(1.5ex,1.5ex);} $\in\mathcal{T}^{\mathrm{cut}}_h$ &
\tikz{\draw[fill=FigPurple] rectangle(1.5ex,1.5ex);} $\in\mathcal{T}_h^{\partial,\mathrm{ag}}$ &
\tikz{\filldraw[draw=FigBlue,fill=FigBlue,line width=1.5] (0,0) circle (0.05);} vertices $\in\mathcal{C}_{h}^{\mathrm{ipf}}$ &
\tikz{\draw[color=FigBlue,line width=1.5] (0,0)--(1.7ex,1.7ex);} edges $\in\mathcal{C}_{h}^{\mathrm{ipf}}$ &
\tikz{\draw[color=black,line width=1.5] (0,0)--(1.7ex,1.7ex);}$\in\partial\Omega$ \\[1em]
\end{tabular}

\begin{subfigure}{0.24\textwidth}
\centering
\includegraphics[width=0.8\textwidth]{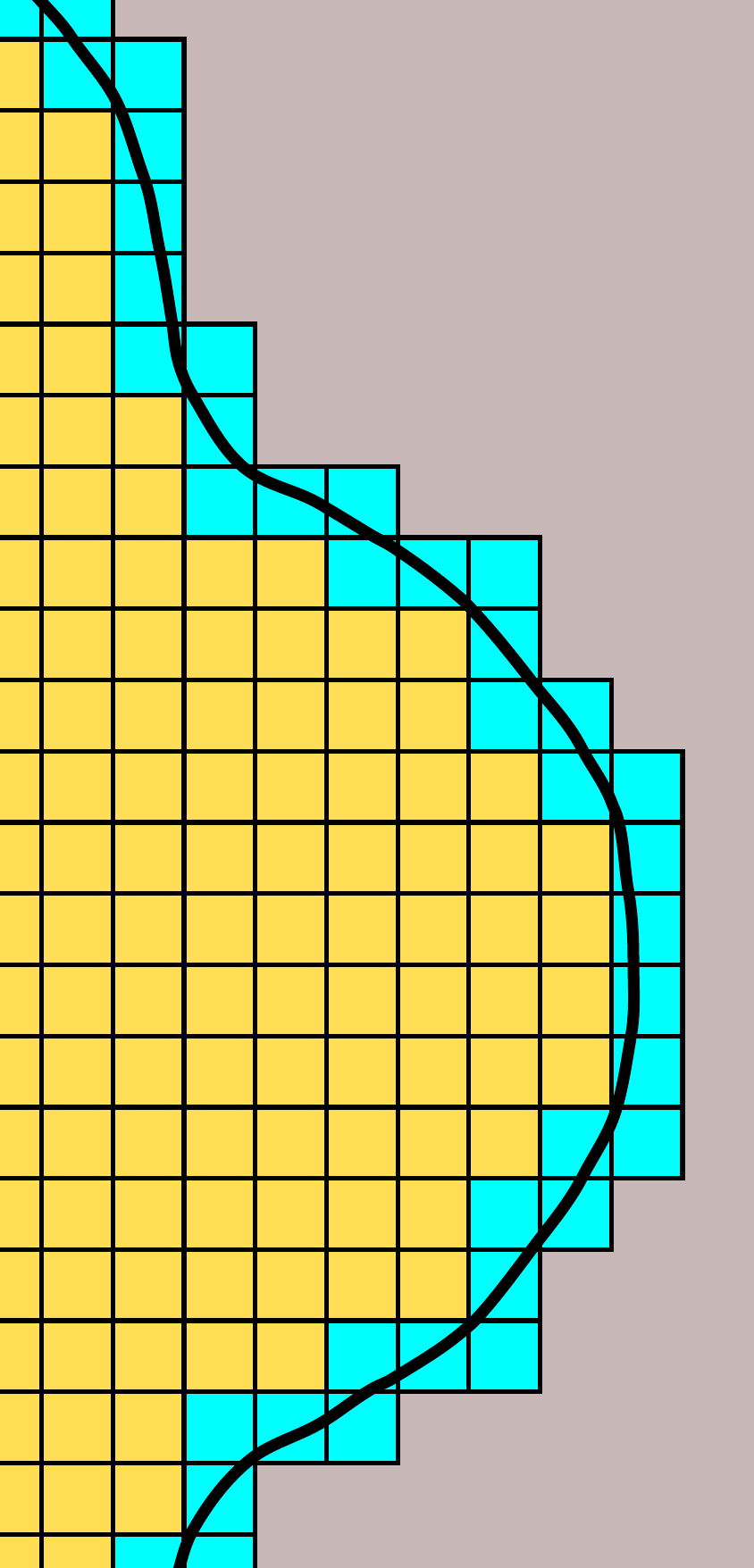}
\caption{}\label{fig:geodefs_a}
\end{subfigure}
\begin{subfigure}{0.24\textwidth}
\centering
\includegraphics[width=0.8\textwidth]{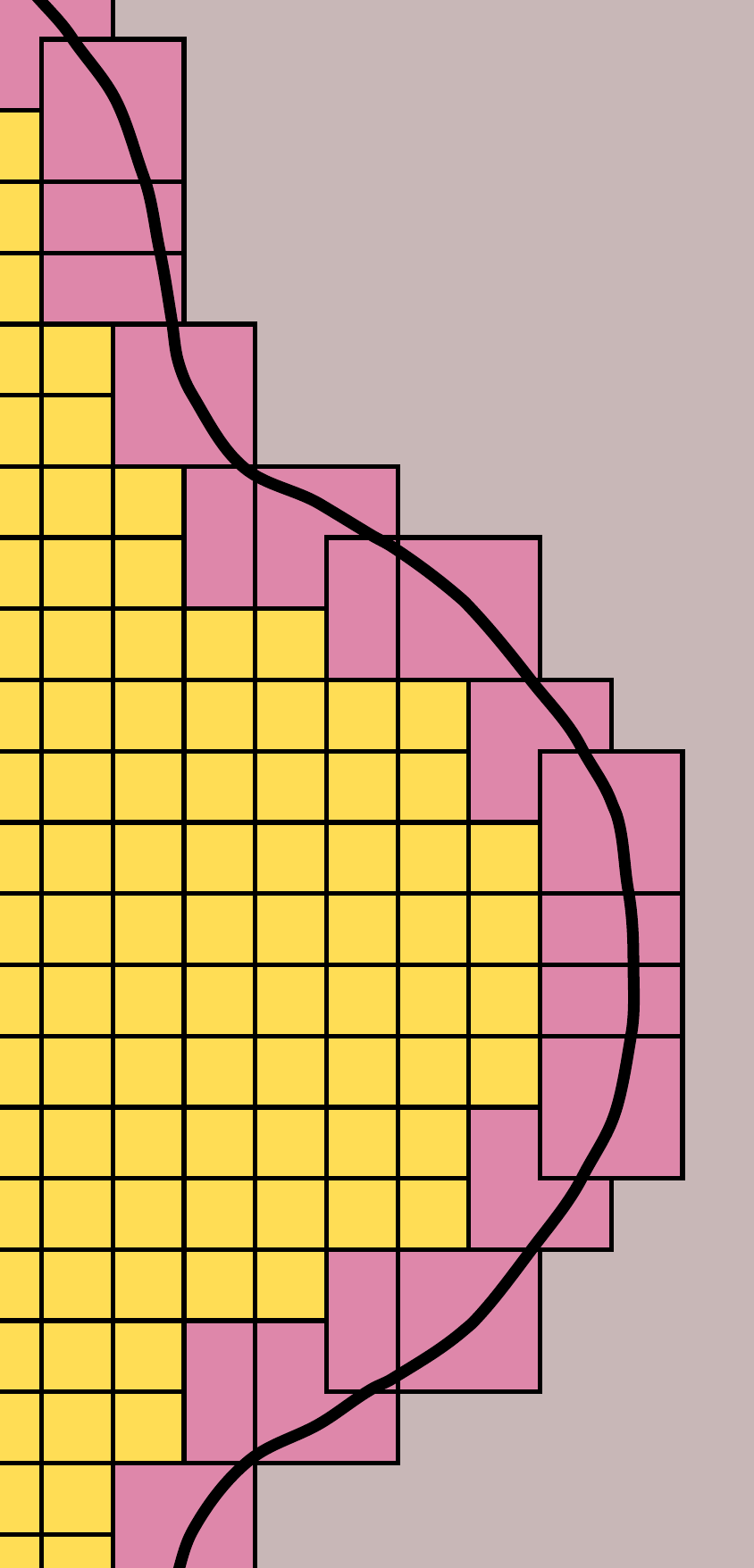}
\caption{}\label{fig:geodefs_b}
\end{subfigure}
\begin{subfigure}{0.24\textwidth}
\centering
\includegraphics[width=0.8\textwidth]{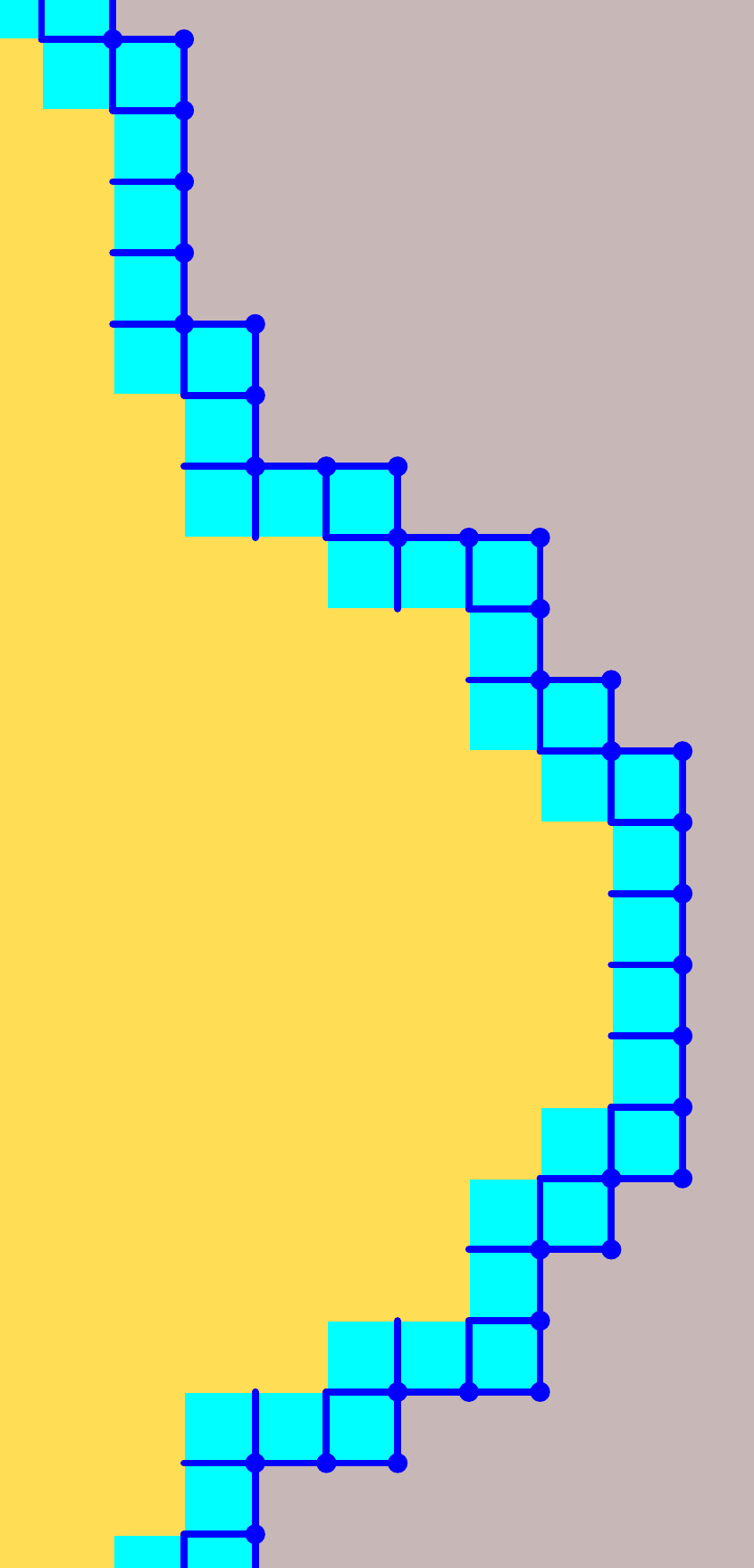}
\caption{}\label{fig:geodefs_c}
\end{subfigure}
\caption{\new{Illustration of the main geometrical sets introduced in Sect.~\ref{sec:cell-aggregation}.}}
\label{fig:geodefs}
\end{figure}

In aggregated unfitted methods, we associate (potentially problematic) cut cells to (fully $\Omega$-supported) interior cells. This leads to the notion of the so-called \emph{aggregates}: Let $\mathcal{T}_{h}^{\mathrm{ag}}$ denote an \emph{aggregated} or \emph{agglomerated} mesh. $\mathcal{T}_{h}^{\mathrm{ag}}$ is the output of a cell aggregation of $\mathcal{T}_{h}^{\mathrm{act}}$. Each aggregate is formed by exactly one interior cell in $\mathcal{T}_{h}^{\mathrm{in}}$, referred to as the \emph{root} cell, and several cut cells in $\mathcal{T}_{h}^{\mathrm{cut}}$, such that each active cell belongs to one, and only one, aggregate. It follows that cell aggregation is only meaningful on the boundary, e.g., interior cells that are not touching any cut cell become singleton aggregates. Hence, let $\mathcal{T}_h^{\partial,\mathrm{ag}} \doteq \mathcal{T}_{h}^{\mathrm{ag}} \setminus \mathcal{T}_h^{\mathrm{in}}$ be the non-trivial aggregates on the boundary, see Fig.~\ref{fig:geodefs_b}. General dimension-independent cell aggregation algorithms are described in~\cite{Badia2018} for conforming meshes, in~\cite{Badia2020Jun} for non-conforming meshes, and in~\cite{Neiva2021} for multiphase problems. Their parallel implementation is covered in~\cite{Verdugo2019}. To guarantee convergence, these algorithms should minimise the aggregate size. In particular, the characteristic size of an aggregated cell must be proportional to the one of its root cell. The goal of cell aggregation is to end up with a new partition of cells, in which all cells (aggregates) have support in $\Omega$ away from zero. The resulting mesh restores \emph{shape-regularity}. More specifically, there exists a constant $\tau > 0$ independent of the mesh size or cut location such that every cell $T \in \mathcal{T}_{h}^{\mathrm{ag}}$ contains a ball or radius $\rho_T$ inside $\Omega$, such that $\frac{h_T}{\rho_T} \leq \tau$.

To conclude with the geometrical definitions, we extend the previous classifications to the \emph{$n$-faces} of $\mathcal{T}_{h}$. Here, \emph{$n$-face} refers to entities in any dimension. For instance, in 3D, 0-faces are vertices, 1-faces are edges, 2-faces are faces and 3-faces are cells. We use \emph{facet} to denote an $n$-face of dimension $d-1$, i.e., an edge in 2D and a face in 3D. According to this, we let $\mathcal{C}_{h}^{\#}$ represent the (simplicial or hexahedral) \emph{exact complex} of $\mathcal{T}_h^\#$  for $\# \in \left\{  \mathrm{act}, \mathrm{in}, \mathrm{cut}, \mathrm{out} \right\}$, i.e., the set of all $n$-faces of cells in $\mathcal{T}^{\#}_{h}$. \new{ The $n$-faces in $\mathcal{C}_{h}^{\mathrm{ipf}}\doteq \mathcal{C}_{h}^{\mathrm{cut}} \setminus \mathcal{C}_{h}^{\mathrm{in}}$, see Fig.~\ref{fig:geodefs_c}, are referred to as \emph{ill-posed $n$-faces}, since the shape functions associated to $n$-faces in $\mathcal{C}_{h}^{\mathrm{ipf}}$ are the only ones that potentially have an arbitrarily small support on $\Omega$ and can lead to arbitrary large condition numbers.}

\section{Problem statement}\label{sec:problem_statement}

Let us consider as a model problem for our presentation the Poisson equation in $\Omega$ with Dirichlet boundary conditions on $\Gamma_{\mathrm{D}} \subset \partial \Omega$ and Neumann boundary conditions on $\Gamma_{\mathrm{N}} \doteq \partial \Omega \setminus \Gamma_{\mathrm{D}}$. After scaling with the diffusion term, the equation reads:
\emph{find} $u \in H^1(\Omega)$ \emph{such that}
\begin{equation}\label{eq:poisson-strong}
	- \boldsymbol{\nabla} \cdot \boldsymbol{\nabla} u = f \quad \text{in} \ H^{-1}(\Omega), \qquad u = g \quad \text{in} \ H^{1/2}(\Gamma_{\mathrm{D}}),  \qquad \boldsymbol{n} \cdot \boldsymbol{\nabla} u = q \quad \text{in} \ H^{-\frac{1}{2}}(\Gamma_{\mathrm{N}}),
\end{equation}
where $f$ is the source term, $g$ is the prescribed value on the Dirichlet boundary and $q$ the prescribed flux on the Neumann boundary. 

The following exposition applies to second-order elliptic equations.
For instance, in the numerical experiments, we also consider the linear elasticity problem: \emph{find} $\boldsymbol{u} \in \boldsymbol{H}^1(\Omega)$ \emph{such that}
\begin{equation}\label{eq:elasticity-strong}
	- \boldsymbol{\nabla} \cdot \boldsymbol{\sigma}(\boldsymbol{u}) = \boldsymbol{f} \quad \text{in} \ \boldsymbol{H}^{-1}(\Omega), \qquad \boldsymbol{u} = \boldsymbol{g} \quad \text{in} \ \boldsymbol{H}^{1/2}(\Gamma_{\mathrm{D}}),  \qquad \boldsymbol{n} \cdot \boldsymbol{\sigma}(\boldsymbol{u}) = \boldsymbol{q} \quad \text{in} \ \boldsymbol{H}^{-\frac{1}{2}}(\Gamma_{\mathrm{N}}),
\end{equation}
where $\boldsymbol{\sigma}, \boldsymbol{\varepsilon} : \Omega \to \mathbb{R}^{d,d}$ are the stress tensor $\boldsymbol{\sigma} (\boldsymbol{u}) = 2 \mu \boldsymbol{\varepsilon}(\boldsymbol{u}) + \lambda \mathrm{tr} (\boldsymbol{\varepsilon}(\boldsymbol{u})) \mathbf{Id}$ and the strain tensor $\boldsymbol{\varepsilon}(\boldsymbol{u}) \doteq \frac{1}{2} (\boldsymbol{\nabla} \boldsymbol{u} + { \boldsymbol{\nabla} \boldsymbol{u}}^T)$; with $\mathbf{Id}$ the identity matrix in $\mathbb{R}^d$. $(\lambda, \mu)$ are the the Lamé coefficients. We assume the Poisson ratio $\nu \doteq \lambda / ( 2 ( \lambda + \mu ) )$ is bounded away from $1/2$, i.e., the material is \emph{compressible}. Since $\lambda = 2\nu \mu / (1 - 2\nu)$, it follows that $\lambda$ is bounded above by $\mu$, i.e., $\lambda \leq C \mu$, for some positive constant $C$.

We turn now to the Galerkin approximation of the Poisson problem~(\ref{eq:poisson-strong}). Let $\mathcal{V}_{h}^{\mathrm{act}}$ be a standard Lagrangian \ac{fe} space on $\mathcal{T}_{h}^{\mathrm{act}}$. As mentioned above, boundary conditions are \emph{weakly} imposed with Nitsche's method~\cite{Badia2018,Schillinger2015,burman_cutfem_2015}. This approach yields a consistent numerical scheme with optimal convergence for arbitrary order \ac{fe} spaces. Hence, the Galerkin approximation to \eqref{eq:poisson-strong} reads: find $u_h \in \mathcal{V}_h^{\rm act}$ such that $a_h(u_h,v_h)= b_h(v_h)$ for any $v_h \in \mathcal{V}_h^{\rm act}$, with
\begin{equation}\label{eq:poisson-weak}
 	\begin{array}{l}
 		\displaystyle {a_h}(u_h,v_h) \doteq \int_{\Omega} \boldsymbol{\nabla} u_h \cdot \boldsymbol{\nabla} v_h \mathrm{\ d}\Omega \ + \int_{\Gamma_{\mathrm{D}}} \left( \tau u_h v_h  - u_h \left( \boldsymbol{n} \cdot \boldsymbol{\nabla} v_h \right ) - v_h \left( \boldsymbol{n} \cdot \boldsymbol{\nabla} u_h \right) \right) \mathrm{\ d}{\Gamma}, \quad \text{and} \\
 		\displaystyle {b_h}(v_h) \doteq \int_{\Omega} v_h f \mathrm{\ d}\Omega \ + \int_{\Gamma_{\mathrm{D}}} \left( \tau v_h g - \left( \boldsymbol{n} \cdot \boldsymbol{\nabla} v_h \right ) g \right) \mathrm{\ d}\Gamma \ + \int_{\Gamma_{\mathrm{N}}}^{} q v_h \ \mathrm{d}\Gamma,
 	\end{array}
\end{equation}
with $\boldsymbol{n}$ being the outward unit normal on $\partial \Omega$ and $\tau > 0$ a large-enough stabilisation parameter, defined shortly. Concerning the linear elasticity problem in~(\ref{eq:elasticity-strong}), the approximation takes the form:
\begin{equation}\label{eq:elasticity-weak}
  \begin{array}{l}
    \displaystyle a_h(\boldsymbol{u}_h,\boldsymbol{v}_h) \doteq \int_{\Omega} \boldsymbol{\sigma}(\boldsymbol{u}_h) : \boldsymbol{\varepsilon}(\boldsymbol{v}_h) \ \mathrm{d} \Omega \ + \int_{\Gamma_{\mathrm{D}}} \left( \tau {\boldsymbol{u}_h} \cdot {\boldsymbol{v}_h} - \boldsymbol{n} \cdot \boldsymbol{\sigma}(\boldsymbol{v}_h) \cdot {\boldsymbol{u}_h} - \boldsymbol{n} \cdot {\boldsymbol{\sigma}(\boldsymbol{u}_h)} \cdot {\boldsymbol{v}_h} \right) \ \mathrm{d} \Gamma  , \\
    \displaystyle b_h(\boldsymbol{v}_h) \doteq \int_{\Omega} \boldsymbol{f} \cdot \boldsymbol{v}_h \ \mathrm{d} \Gamma \ + \int_{\Gamma_{\mathrm{D}}} ( \tau \boldsymbol{g} \cdot {\boldsymbol{v}_h} - \boldsymbol{n} \cdot {\boldsymbol{\sigma}(\boldsymbol{v}_h)} \cdot \boldsymbol{g} ) \ \mathrm{d} \Gamma + \int_{\Gamma_{\mathrm{D}}} \boldsymbol{q} \cdot  \boldsymbol{v}_h \ \mathrm{d} \Gamma.
  \end{array}
\end{equation}

The second terms in all the forms of~(\ref{eq:poisson-weak}) and~(\ref{eq:elasticity-weak}) correspond to the Nitsche terms in charge of the weak imposition of Dirichlet boundary conditions. We observe that the penalty method or a non-symmetric version of Nitsche's method \cite{Freund1995} are common alternatives to the Nitsche method. Nonetheless, the penalty formulation is not weakly consistent for high order methods and the non-symmetric formulation sacrifices symmetry of the discrete system and adjoint consistency.

Stability of the discrete problems above depends upon the $\tau$-dependent property, e.g., for the Poisson problem,
\begin{equation}\label{eq:nistche-condition}
  \begin{array}{l}
    \int_{\Gamma_{\mathrm{D}} \cap T} \left( \tau u_h^2  - 2 u_h
    \left( \boldsymbol{n} \cdot \boldsymbol{\nabla} u_h \right ) \right) \mathrm{\ d}{\Gamma}
    \leq C \int_{\Gamma_{\mathrm{D}} \cap T}^{} \tau u_{h}^{2} \ \mathrm{d} \Gamma + \|\boldsymbol{\nabla}u_h\|^2_{L^2(T)}, \ \ \forall T\in\mathcal{T}_h^{\mathrm{cut}},
  \end{array}
\end{equation}
for some constant $C > 0$ independent of $h_T$. A cell-wise $\tau_T$ that verifies~(\ref{eq:nistche-condition}) can be computed via the solution of a generalised eigenvalue problem~\cite{DePrenter2017}. In shape-regular body-fitted meshes, it is enough to prescribe the value $\tau_T = \beta m^2 h_T^{-1}$, where $\beta$ is a large enough problem-dependent parameter and $m$ is the order of $\mathcal{V}_h^{\rm act}$. For standard unfitted \ac{fe} methods formulated in $\mathcal{V}_h^{\rm act}$ we only have stability over $\|\boldsymbol{\nabla}u_h\|^2_{L^2(T \cap \Omega)}$ in the right-hand side of~(\ref{eq:nistche-condition}). In this case, the minimum value of $\tau_T$ that ensures stability tends to infinity as $|T \cap \Omega | \to 0$. As a result, unfitted \acp{fem}, such as \ac{xfem}, are not robust to cut location (either for boundary or interface problems). 

We can also relate the lack of robustness to the scaling of the condition number in classical vs unfitted \acp{fem}. Classical \ac{fem} approximation theory has long established that the condition number of stiffness matrices associated to FEM approximations of second order elliptic differential equations on body-fitted quasi-uniform meshes scales as $h^{-2}$. The largest eigenvalue of such matrices scales as $h^{d-2}$ and the smallest eigenvalue as $h^{d}$. Their associated eigenvectors are the functions with the highest, resp., lowest frequency possible on the mesh.

Unfitted approximations in $\mathcal{V}_h^{\rm act}$ recover the classical scaling of the largest eigenvalue with $h^{d-2}$, if locally stabilised by solving the generalised eigenvalue problem mentioned above.\footnote{All results in this paragraph are proven in~\cite{DePrenter2017} under two assumptions (1) shape-regularity of cut regions $T \cap \Omega$, $T \in \mathcal{T}_h^{\rm act}$, and (2) the size of the intersection between the unfitted boundary/interface and a cell is bounded by the volume of the cut region.} However, the smallest eigenvalue is bounded above by $h^{d-2} \eta^{2m+1-2/d}$, with $\eta = \min_{T \in \mathcal{T}_h^{\rm act}} \mathrm{meas} (T \cap \Omega)$ \cite{DePrenter2017}, and the associated eigenfunction has support only in a cell with a very small volume fraction. As a result, the condition number of the discrete system scales as $\eta^{-(2m+1-2/d)}$ and arbitrarily high condition numbers occur in practice, since the position of the cuts cannot be controlled and the value of $\eta$ can be arbitrarily close to zero. In addition, by observing the exponential dependence of the scaling rate with the order of approximation $m$, we deduce that the problem becomes especially severe with high-order methods.\footnote{In this work, we consider all cut cells potentially ill-posed. Instead, we could define a parameter $\eta_0 \in (0,1]$ and consider as ill-posed only the cells $T \in \mathcal{T}_{h}^{\mathrm{cut}}$ such that $|T \cap \Omega| / |T| \geq \eta_0$. However, based on the previous bound for the condition number, this approach is much less effective as the order or approximation increases. In particular, medium or large cuts that are not problematic for linear \acp{fe}, can become significantly problematic as $m$ grows.}

In the next section, we introduce \acp{agfem}, which solve the previous stability issues. They achieve this either \emph{strongly}, by considering approximations in $\mathcal{V}_h^{\rm ag}$, instead of $\mathcal{V}_h^{\rm act}$, or \emph{weakly}, by penalising the distance of the approximation in $\mathcal{V}_h^{\rm act}$ w.r.t.~the one in $\mathcal{V}_h^{\rm ag}$. With these methods, we can use the same expression of $\tau_T$ as in body-fitted meshes, $h_T$ being the background cell size. Moreover, we recover standard $h^{-2}$ condition number bounds. We use $A \gtrsim B$ (resp.~$A \lesssim B$) to denote $A \geq C B$ (resp.~$A  \leq CB$) for some positive constant $C$ that does not depend on $h$ and the location of the cell cuts and can only depend on $m$ with a polynomic rate. Proving uniform bounds irrespectively of mesh size $h$, boundary or interface locations, which do not blow up exponentially with $m$, is the driving motivation behind all these methods.

\section{Aggregated finite elements}\label{sec:agfem}

The motivation behind \ac{agfem} is to apply the same cell aggregation ideas in \ac{dg} schemes on unfitted meshes to $\mathcal{C}^0$ Lagrangian finite element spaces. \ac{dg} methods can readily be applied to polytopal meshes, so the scheme can readily be applied to aggregation or agglomeration meshes. Shape regularity of the intersected cells is generally lost after intersection with the boundary or interface, but aggregates are defined in such a way that their region inside the domain is still shape regular. For brevity, we restrict the presentation to \emph{strong} \ac{agfem} methods. We refer to~\cite{badia2021linkAgFEM} for details on their \emph{weak} counterparts.

\subsection{Strong \ac{agfem}}\label{sec:agfem-1}

Let us formalise the \emph{strong} version of \ac{agfem} grounded on the standard (extrapolation-based) discrete extension operator. The definition of $\mathcal{C}^0$ Lagrangian finite element spaces on aggregated meshes has been proposed in \cite{Badia2018}. The underlying idea is to define a new \ac{fe} space that can be expressed in terms of an aggregate-wise \emph{discrete extension} operator $\mathcal{E}^{\mathrm{ag}}_h: \mathcal{V}_{h}^{\mathrm{in}} \longrightarrow \mathcal{V}_{h}^{\mathrm{act}}$.  $\mathcal{E}^{\mathrm{ag}}_h$ extends \ac{fe} functions from the root cells to the cut cells. This definition has two salient properties by construction: the constraints are local and the resolution of the interior cells is preserved, i.e., interior \acp{dof} are not constrained in this process. We note that these two properties are violated by spaces recovered from standard ghost-penalty methods, as the penalty coefficient goes to infinity~\cite{badia2021linkAgFEM}.

The image of this extension is the Ag\ac{fe} space $\mathcal{V}_{h}^{\mathrm{ag}} \subset \mathcal{V}_{h}^{\mathrm{act}}$; $\mathcal{V}_{h}^{\mathrm{ag}}$ can be built by adding constraints to $\mathcal{V}_{h}^{\mathrm{act}}$. The new Ag\ac{fe} space is not affected by the small cut cell problem, since the ill-posed \acp{dof} are constrained by well-posed interior \acp{dof}. It remains to see how the Ag\ac{fe} constraints are formed. Before that, we introduce some auxiliary notation.

As $\mathcal{V}_{h}^{\mathrm{act}}$ is a nodal Lagrangian \ac{fe} space, there exists a one-to-one relation between shape functions, nodes and \acp{dof}. Each node in the mesh can be associated to its \emph{owner}, which is defined as the lowest dimensional $n$-face (e.g., vertex, edge, face, cell) that contains it; we denote this map as $\mathcal{O}_h^{\mathrm{dof}\to \mathrm{nf}}$. \new{Using this notation, we define the set of \emph{ill-posed} \acp{dof} $\mathcal{D}_{h}^{\mathrm{ipd}}$  as the subset of \acp{dof} that are owned by  $n$-faces in $\mathcal{C}_{h}^{\mathrm{ipf}}$, namely, each ill-posed \ac{dof} $\alpha\in \mathcal{D}_{h}^{\mathrm{ipd}}$ is associated with the ill-posed $n$-face $\mathcal{O}_h^{\mathrm{dof}\to \mathrm{nf}}(\alpha)\in \mathcal{C}_{h}^{\mathrm{ipf}}$. }


The definition of the discrete extension operator in the \ac{agfem} requires an ownership map $\mathcal{O}_h^{\mathrm{nf \to ag}}: \mathcal{C}_{h}^{\mathrm{ipf}} \rightarrow \mathcal{T}_{h}^{\partial,\mathrm{ag}}$ from cut/external $n$-faces to aggregates. For inter-aggregate \acp{dof}, we arbitrarily choose one of the touching aggregates, such that the mapping is unique. On the other hand, each aggregate has a unique root cell in $\mathcal{T}_{h}^{\mathrm{in}}$; thus we have the bijection $\mathcal{O}_h^{\mathrm{ag \to in}}$, mapping aggregates to interior root cells, and the inverse map $\mathcal{O}_h^{\mathrm{in \to ag}}$. Composing all these maps, we end up with an \new{ill-posed-\ac{dof}-to-root-cell} map $\mathcal{O}_h^{\mathrm{dof} \to \mathrm{in}}: \mathcal{D}_{h}^{\mathrm{ipd}} \rightarrow \mathcal{T}_{h}^{\mathrm{in}}$, where $\mathcal{O}_h^{\mathrm{dof} \to \mathrm{in}} \doteq \mathcal{O}_h^{\mathrm{ag \to in}} \circ \mathcal{O}_h^{\mathrm{nf \to ag}} \circ \mathcal{O}_h^{\mathrm{dof}\to \mathrm{nf}}$.

Now, let $\mathcal{O}_{h}^{\mathrm{nf} \to \mathrm{dof}}$ be the inverse of  $\mathcal{O}_{h}^{\mathrm{dof} \to \mathrm{nf}}$, i.e., the map that returns the \acp{dof} owned by an $n$-face in $\mathcal{C}_{h}^{\mathrm{act}}$. We also need a \emph{closed} version of this ownership map $\overline{\mathcal{O}}_{h}^{\mathrm{nf} \to \mathrm{dof}}$, which given an $n$-face $C$ in $\mathcal{C}_{h}^{\mathrm{act}}$ returns the owned \acp{dof} of all $n$-faces $C' \in \mathcal{C}_{h}^{\mathrm{act}}$ in the closure of $C$, i.e., $C' \subseteq C$. $\overline{\mathcal{O}}_{h}^{\mathrm{nf} \to \mathrm{dof}}$ is the map that describes the locality of Lagrangian \ac{fe} methods: The only \acp{dof} that are \emph{active} in a cell $T \in \mathcal{T}_{h}^{\mathrm{act}}$ are the ones in $\overline{\mathcal{O}}_{h}^{\mathrm{nf} \to \mathrm{dof}}(T)$. Analogously, only the shape functions associated to these \acp{dof} have support on $T$.

Using the notation introduced above, we can readily define the standard discrete extension operator as follows. An ill-posed \ac{dof} in $\mathcal{D}_{h}^{\mathrm{ipd}}$ is computed as a linear combination of the well-posed \acp{dof} in the closure of the root cell that owns it. Combining previous definitions, the ill-posed-\ac{dof} to well-posed-\acp{dof} map is given by $\mathcal{O}_{h}^{\mathrm{ipd} \to \mathrm{wpd}} \doteq \overline{\mathcal{O}}_{h}^{\mathrm{nf} \to \mathrm{dof}} \circ \mathcal{O}_{h}^{\mathrm{dof} \to \mathrm{in}}$, see Fig.~\ref{fig:defs_maps}. According to this, the constrained value of $\sigma^\alpha \in \mathcal{D}_{h}^{\mathrm{ipd}}$ is
\begin{equation}\label{eq:agfem-constraints}
  \sigma^{\alpha}(\cdot) = 
  \sum_{\sigma^{\beta} \in \mathcal{O}_{h}^{\mathrm{ipd}\to \mathrm{wpd}}(\alpha)} 
  \sigma^{\alpha}(\phi^{\beta}) \sigma^{\beta}(\cdot) = 
  \sum_{\sigma^{\beta} \in \mathcal{O}_{h}^{\mathrm{ipd}\to \mathrm{wpd}}(\alpha)}
  \phi^{\beta}(\boldsymbol{x}^{\alpha}) \sigma^{\beta}(\cdot). 
\end{equation}
We can readily use this expression to extend well-posed \ac{dof} values on interior cells to all ill-posed \ac{dof} values, which only belong to cut cells. Thus, applying ~\eqref{eq:agfem-constraints} to a \ac{fe} function in $\mathcal{V}_{h}^{\mathrm{in}}$ provides the sought-after discrete extension operator $\mathcal{E}_{h}^{\mathrm{ag}}$ and the Ag\ac{fe} space $\mathcal{V}_{h}^{\mathrm{ag}}$.

\begin{figure}[ht!]
\centering
\includegraphics[width=0.75\textwidth]{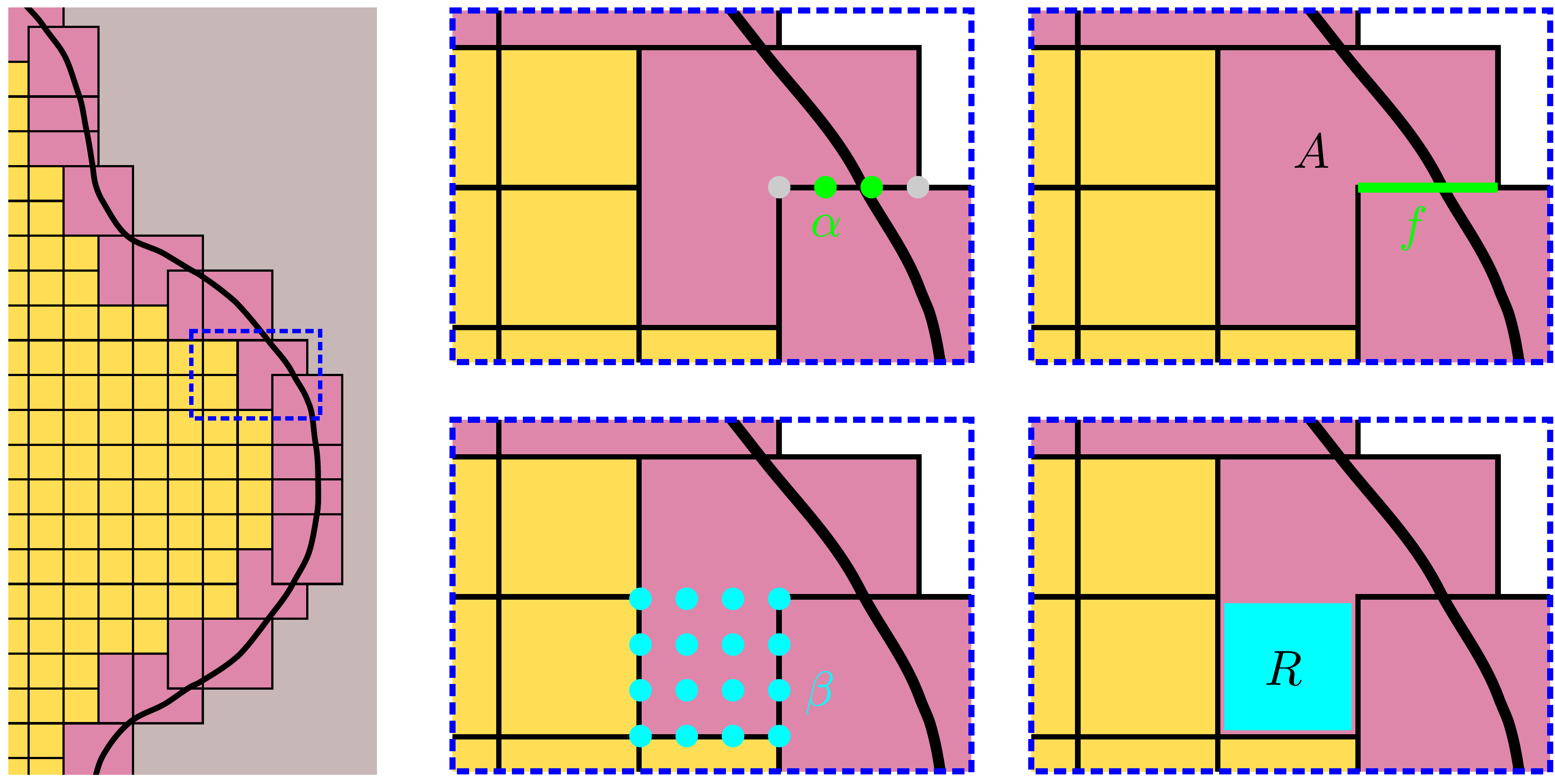}

\caption{\new{Illustration of the main notations defined in  Sect.~\ref{sec:agfem-1} for a third order Lagrangian interpolation. $\alpha\in \mathcal{O}_{h}^{\mathrm{nf} \to \mathrm{dof}}(f)$ is a \ac{dof} (i.e., a node) at the interior of face $f$ and  $\beta\in \overline{\mathcal{O}}_{h}^{\mathrm{nf} \to \mathrm{dof}}(R)$ is a \ac{dof} in the closure of $R$. $\alpha$ and $\beta$ are related by the map that transforms ill-posed \acp{dof} to well-posed ones, namely  $\beta \in \mathcal{O}_{h}^{\mathrm{ipd} \to \mathrm{wpd}}(\alpha)$. The other intermediate maps introduced in Sect.~\ref{sec:agfem-1} are as follows for this example: $f=\mathcal{O}_{h}^{\mathrm{dof} \to \mathrm{nf}}(\alpha)$ is the $n$-face that owns \ac{dof} $\alpha$, $A=\mathcal{O}_{h}^{\mathrm{nf} \to \mathrm{ag}}(f)$ is the aggregate assigned to $f$, and $R=\mathcal{O}_{h}^{\mathrm{ag} \to \mathrm{in}}(A)$ is the root cell of aggregate $A$.}}
\label{fig:defs_maps}
\end{figure}

%
%

Implementation of \ac{agfem} is straightforward, as it simply requires implementing a cell aggregation scheme (upon which to build the maps above) and the imposition of linear constraints~\eqref{eq:agfem-constraints} in the discrete system assembly. Besides, Ag\ac{fem} constraints are cell-local (much simpler than the ones in $h$-adaptive mesh refinement) and the weak form remains unchanged, it is only evaluated at a different \ac{fe} space. The method reads: find $u_h \in \mathcal{V}_{h}^{\mathrm{ag}}$ such that $a_h(u_h,v_h) = b_h(v_h)$ for any $v_h \in \mathcal{V}_{h}^{\mathrm{ag}}$. 

\subsection{Abstract stability and convergence analysis}

The \ac{agfem} relies on a discrete extension operator with the properties in the following definition. In previous works for standard Lagrangian extensions with extrapolation (see, e.g., \cite{badia2021linkAgFEM}), the requirements over the extension were stronger in the sense that the bounded norms were on $\Omega_{h}^{\mathrm{act}}$. On the other side, the constants could increase exponentially with $m$. Thus, the standard method can hardly be used for high order approximations.

In this work, we observe that stability on $\Omega$ is all what is needed in \ac{agfem}. Using this information, we propose a basis that is interpolatory in $\Omega$ and, as a result, estimates do not blow up exponentially with $m$. We will show that the \ac{agfem} with these new bases is suitable for high-order approximations.  

\begin{definition}\label{def:discrete-extension-operator}
  Let $0 \leq s \leq n \leq m$,  where $m$ is the order of $\mathcal{V}_{h}^{\mathrm{in}}$. We denote by $C(m)$ constants that can grow at a polynomial rate with the order $m$. A suitable discrete extension operator $\mathcal{E}_{h}^{\mathrm{ag}}: \mathcal{V}_{h}^{\mathrm{in}} \rightarrow  \mathcal{V}_{h}^{\mathrm{act}}$ must satisfy the following properties: 
  \begin{itemize}
         \item[(i)] Continuity: 
    \begin{equation}\label{eq:discrete-extension-continuity}
		{\| \mathcal{E}_{h}^{\mathrm{ag}}(v_h) \|_{L^2(\Omega
    )} \lesssim \|v_h\|_{L^2(\Omega_{h}^{\mathrm{in}})},} \qquad \| \boldsymbol{\nabla}\mathcal{E}_{h}^{\mathrm{ag}}(v_h) \|_{\boldsymbol{L}^2(\Omega
    )} \lesssim \|\boldsymbol{\nabla}  v_h\|_{\boldsymbol{L}^2(\Omega_{h}^{\mathrm{in}})}, \qquad \forall v_h \in \mathcal{V}_{h}^{\mathrm{in}}.
    \end{equation}
    \item[(ii)] Approximability:
\begin{equation}\label{eq:discrete-extension-approximability}
  \underset{v_h \in \mathcal{V}_{h}^{\mathrm{in}}}{\mathrm{inf}} \| u - \mathcal{E}_{h}^{\mathrm{ag}}(v_h) \|_{H^s(\Omega)} \leq C(m)h^{n-s+1} \| u \|_{H^{n+1}(\Omega)}, \qquad \forall u  \in H^{n+1}(\Omega).
\end{equation}
\end{itemize}
The image $\mathcal{V}_{h}^{\mathrm{ag}} \doteq \mathrm{Im}(\mathcal{E}_{h}^{\mathrm{ag}}) \subset \mathcal{V}_{h}^{\mathrm{act}}$ is a suitable Ag\ac{fe} space. \new{Even though functions in $\mathcal{V}_{h}^{\mathrm{ag}}$ are defined in $\Omega_{h}^{\mathrm{act}}$, we only consider their restriction to $\Omega$ in the following exposition.}    
\end{definition}
\new{We note that the stability bounds in (\ref{eq:discrete-extension-continuity}) are essential to bound the condition number of the resulting matrix (see, e.g., \cite[Corollary 5.9]{Badia2018}).} 

In the next proposition, we show that in fact these properties lead to the desired continuity and stability results. In order to treat the boundary terms due to Nitsche, we make use of a trace inequality. Given a domain $\omega$ with Lipschitz boundary, the following trace inequality holds (see, e.g., \cite[Th.~1.6.6]{Brenner1994}):
\begin{equation}\label{eq:tr-ineq}
\|u\|^{2}_{L^2(\partial \omega)} \leq C_\omega \|u\|_{L^2(\omega)} \|u\|_{H^1(\omega)}, \qquad u \in H^1(\omega).
\end{equation}
The constant $C_\omega$ depends only on the shape of $\omega$. We can safely use this expression at the aggregate level for $A \cap \Omega$, $A \in \mathcal{T}_{h}^{\mathrm{ag}}$ since the aggregate is shape regular by construction. 

We also need an inverse inequality for aggregates. Given $U \in \mathcal{T}_{h}^{\partial,\mathrm{ag}}$, we have that $\rho h_U \leq \mathrm{diam}(\Omega \cap U) \leq h_U$ , where $h_U$ is a characteristic mesh size and $\rho \ \new{\geq} \ C > 0$ . This is a result of the geometrical construction of aggregates, i.e., the size of the aggregate is uniformly bounded by the root size and the root size is shape-regular. Thus, one can use the inverse inequality in \cite[Lemma 4.5.3]{Brenner1994} to get the following inequality: for any $u_h \in \mathcal{V}_{h}^{\mathrm{ag}}$, it holds 
\begin{equation}\label{eq:inv-ineq}
\| u_h \|_{H^1(\Omega \cap U)} \le C h^{-1}_U \|u_h \|_{L^2(\Omega \cap U)}, \qquad  \forall U \in \mathcal{T}_{h}^{\partial, \mathrm{ag}}. 
\end{equation} 
The proof relies on the shape regularity properties mentioned above, the equivalence of discrete norms and scaling arguments. For interior cells, we recover the standard inverse inequality.

\begin{remark}
We note that this is not the situation for ghost penalty methods~\cite{burman2010ghost} or weak \ac{agfem} methods~\cite{badia2021linkAgFEM}. In these schemes, cut cells do not have the required shape regularity and one must use inverse inequalities that require control over the whole cut cell~\cite{hansbo2002unfitted}. 
\end{remark}

\begin{proposition}\label{prop:well-posedness-deo}
  Let $\mathcal{E}_{h}^{\mathrm{ag}}$ satisfy Def.~\ref{def:discrete-extension-operator}. Let $\mathcal{V}^{\mathrm{ag}}(h) \doteq H^2(\Omega) + \mathcal{V}_{h}^{\mathrm{ag}}$, endowed
  with the norm
  \begin{equation}\label{eq:norm-agfem}
    {\tnor{v}^{2}_{\mathcal{V}^{\mathrm{ag}}(h)} \doteq \|\boldsymbol{\nabla}v\|^{2}_{\boldsymbol{L}^2(\Omega)} +  \| \tau^{\frac{1}{2}} v \|_{L^2(\Gamma_{\mathrm{D}})}^2 
    + \sum_{T \in \mathcal{T}_{h}^{\mathrm{act}}}^{} h_T^{2} \| v \|_{H^2(T \cap \Omega)}^{2}, \qquad \forall v \in \mathcal{V}^{\mathrm{ag}}(h).}
    \end{equation}
  It holds:
  \begin{align}\label{eq:deo-coercive-continuous}
    a_h(u_h,u_h) \gtrsim \| u_h \|^2_{\mathcal{V}^{\mathrm{ag}}(h)}, \qquad
    a_h(u,v_h) \lesssim \|u\|_{\mathcal{V}^{\mathrm{ag}}(h)} \|v_h\|_{\mathcal{V}^{\mathrm{ag}}(h)},
    \end{align} 
    for any $u_h, v_h \in  \mathcal{V}_{h}^{\mathrm{ag}}$, $u \in \mathcal{V}^{\mathrm{ag}}(h)$. Thus, there is a unique 
  \begin{equation}\label{eq:deo-weak-poisson}
    u_h \in \mathcal{V}_h^{\rm ag} \ : \ a_h(u_h,v_h) = b(v_h), \quad \forall v_h \in \mathcal{V}_{h}^{\mathrm{ag}}.
  \end{equation}
\end{proposition}

\begin{proof}
  Let us consider the Poisson equation in (\ref{eq:poisson-weak}) since the proof is analogous for the elasticity problem in (\ref{eq:elasticity-weak}). In order to prove coercivity, we take $v_h = u_h$ in (\ref{eq:poisson-weak}). The only terms that require some elaboration are the Nitsche terms, which can be expressed as the sum of aggregate-wise contributions. We can now bound these terms at each aggregate $U\in\mathcal{T}_h^{\partial,\mathrm{ag}}$ as follows:
  \begin{align}\label{eq:nistche-condition-2}
      \int_{\Gamma_{\mathrm{D}} \cap U} & \left( \tau u_h^2  - 2 u_h
      \left( \boldsymbol{n} \cdot \boldsymbol{\nabla} u_h \right ) \right) \mathrm{\ d}{\Gamma} 
      \gtrsim \|\tau^{\frac{1}{2}} u_h\|^2_{L^2(\Gamma_{\mathrm{D}}\cap U)} - \|u_h\|_{L^2(\Gamma_{\mathrm{D}}\cap U)} \|\boldsymbol{\nabla}u_h\|_{\boldsymbol{L}^2(\Gamma_{\mathrm{D}}\cap U)} \\
      & \gtrsim \|\tau^{\frac{1}{2}} u_h\|^2_{L^2(\Gamma_{\mathrm{D}}\cap U)} - \|u_h\|_{L^2(\Gamma_{\mathrm{D}}\cap U)} \|\boldsymbol{\nabla}u_h\|^{\frac{1}{2}}_{\boldsymbol{L}^2(\Omega\cap U)} \|\boldsymbol{\nabla}u_h\|^{\frac{1}{2}}_{\boldsymbol{H}^1(\Omega\cap U)} \\
      & \gtrsim \|\tau^{\frac{1}{2}} u_h\|^2_{L^2(\Gamma_{\mathrm{D}}\cap U)} - \xi^{-1} h_U^{-1} \| u_h\|^2_{L^2(\Gamma_{\mathrm{D}}\cap U)}- \xi h_U \|\boldsymbol{\nabla}u_h\|_{\boldsymbol{L}^2(\Omega\cap U)} \|\boldsymbol{\nabla}u_h\|_{\boldsymbol{H}^1(\Omega\cap U)}\\
      & \gtrsim \|\tau^{\frac{1}{2}} u_h\|^2_{L^2(\Gamma_{\mathrm{D}}\cap U)} - \xi^{-1} h_U^{-1}\| u_h\|^2_{L^2(\Gamma_{\mathrm{D}}\cap U)} - \xi \|\boldsymbol{\nabla}u_h\|^2_{\boldsymbol{L}^2(\Omega \cap U)}.
  \end{align}
  In order to obtain this bound, we have used a generalised Young inequality for an arbitrary $\xi >0$, the trace inequality (\ref{eq:tr-ineq}) on $U$ and an inverse inequality at the cells $K' \in \mathcal{T}_{h}^{\mathrm{cut}}$ that belong to $U$. We can now combine this stability with the one that comes from the Galerkin terms. Choosing $\xi$ small enough, the last term can be bounded by the Galerkin control over $\|\boldsymbol{\nabla}u_h\|^2_{\boldsymbol{L}^2(\Omega \cap T)}$, $T \in \mathcal{T}_{h}^{\mathrm{act}}$. In order to absorb the second term by the first one, $\tau$ must be \emph{large enough}. In particular, at each cell of the aggregate, it must hold that $\tau_T > \xi^{-1} h_U^{-1}$ for any $T \in \mathcal{T}_{h}^{\mathrm{cut}}$, $T \subset U$. The standard expression $\tau_T = \beta m^2 h_T^{-1}$ ensures the required stability for a \emph{large enough} $\beta$ independent of the cut location, provided that the ratio between the aggregate size and root cell is bounded and the mesh is quasi-uniform.\footnote{The minimum value of $\tau_T$ can alternatively be computed using an aggregate-wise local eigenvalue problem.} The proof of continuity requires to bound the Nitsche terms using analogous arguments and make use of the inverse inequality (\ref{eq:inv-ineq}). These results lead to the well-posedness of the problem.
\end{proof}

\subsection{Abstract condition number analysis}

Let us denote with $\| \cdot \|_2$ the Euclidean norm of a vector.
\begin{definition}\label{def:mass-bounds}
Given a function $u_h \in \mathcal{V}^{\mathrm{in}}_{h}$ and its nodal vector $\mathbf{u}$, a suitable high-order basis for $\mathcal{V}^{\mathrm{in}}_{h}$ must satisfy 
\begin{equation}\label{eq:mass_matrix_bounds}
  \lambda_{\mathrm{min},M} h^d \| \mathbf{u} \|^2_{2} \leq \| \mathcal{E}_h^{\mathrm{ag}}(u_h) \|^2_{L^2(\Omega)} \leq \lambda_{\mathrm{max},M} h^d \|\mathbf{u}\|_{2}^{2},
\end{equation}
\new{for eigenvalues $(\lambda_{\mathrm{min},M},\lambda_{\mathrm{max},M})$ such that the ratio $\lambda_{\mathrm{max},M}/\lambda_{\mathrm{min},M}$ can only depend linearly with $m$.} 
\end{definition}
We note that for a standard Lagrangian (nodal) basis the lower bound is straightforward since it simply relies on the linear independence of the Lagrangian basis in interior cells~\cite{elman2014finite}. However, the bound for the maximum eigenvalue uses the norm of the discrete extension operator, and thus, grows exponentially with $m$. Thus, we cannot prove the upper bound for this basis. See~\cite{Badia2018} for more details. In the next section, we will propose a basis that satisfies this condition.

\begin{corollary}\label{cor-cond-numb}
Let us consider a basis for $\mathcal{V}^{\mathrm{in}}_{h}$ that satisfies (\ref{eq:mass_matrix_bounds}). The condition number of the system matrix $\boldsymbol{A}_h$   that arises from the bilinear form in (\ref{eq:poisson-weak}) and (\ref{eq:elasticity-weak}) satisfies $\kappa(\boldsymbol{A}_h) \lesssim h^{-2}$.   
\end{corollary}

\begin{proof}
The eigenvalues of $\boldsymbol{A}_h$ can be expressed in terms of Rayleigh quotients
\begin{align}\label{eq:max-eig}
\lambda_{\mathrm{max}}(\boldsymbol{A}_h) & = \underset{\mathbf{u} \in \mathbb{R}^n}{\mathrm{sup}} \frac{\mathbf{u}^T \boldsymbol{A}_h \mathbf{u}}{\|\mathbf{u}\|_2} = \underset{u_h \in \mathcal{V}_{h}^{\mathrm{ag}}}{\mathrm{sup}} \frac{a(u_h,u_h)}{\|u_h\|_{\mathcal{V}^{\mathrm{ag}}(h)}}\frac{\|u_h\|_{\mathcal{V}^{\mathrm{ag}}(h)}}{\|\mathbf{u}\|_2}, \\
\label{eq:min-eig}
\lambda_{\mathrm{min}}(\boldsymbol{A}_h) &= \underset{\mathbf{u} \in \mathbb{R}^n}{\mathrm{inf}} \frac{\mathbf{u}^T \boldsymbol{A}_h \mathbf{u}}{\|\mathbf{u}\|_2} = \underset{u_h \in \mathcal{V}_{h}^{\mathrm{ag}}}{\mathrm{inf}} \frac{a(u_h,u_h)}{\|u_h\|_{\mathcal{V}^{\mathrm{ag}}(h)}}\frac{\|u_h\|_{\mathcal{V}^{\mathrm{ag}}(h)}}{\|\mathbf{u}\|_2}.
\end{align}
The terms related to the bilinear form in the right-hand side of (\ref{eq:max-eig}) (resp., (\ref{eq:min-eig})) can be bounded by the coercivity and continuity results in the previous proposition. The ratio between the continuous and discrete norms is bounded above and below as follows. Using (\ref{eq:mass_matrix_bounds}), inverse inequality (\ref{eq:inv-ineq}), trace inequality (\ref{eq:tr-ineq}), Cauchy-Schwarz, and the quasi-uniformity of the mesh, we get:
\begin{equation}\label{eq:l2bound-vhag}
\tnor{u_h}^{2}_{\mathcal{V}^{\mathrm{ag}}(h)} 
\doteq \|\boldsymbol{\nabla}u_h\|^{2}_{\boldsymbol{L}^2(\Omega)} +  \| \tau^{\frac{1}{2}} u_h \|_{L^2(\Gamma_{\mathrm{D}})}^2 
+ \sum_{T \in \mathcal{T}_{h}^{\mathrm{act}}}^{} h_T^{2} \| u_h \|_{H^2(T \cap \Omega)}^{2}
\lesssim h^{-2} \|u_h\|_{L^2(\Omega)}^{2},
\end{equation}
Using the fact that $u_h \in \mathcal{V}_{h}^{\mathrm{ag}}$ can be expressed as the extension of its interior restriction, the bound in (\ref{eq:l2bound-vhag}) and the upper bound in (\ref{eq:mass_matrix_bounds}), we obtain 
\[
\|u_h\|^2_{\mathcal{V}^{\mathrm{ag}}(h)} \lesssim h^{d-2} \lambda_{\mathrm{max},M} \| \mathbf{u} \|_{2}^2.
\]    
Using a Poincaré-Friedrichs inequality, we readily get $\|u_h\|_{L^2(\Omega)}^2 \lesssim \|u_h\|^2_{\mathcal{V}^{\mathrm{ag}}(h)}$, which combined with the lower bound in (\ref{eq:mass_matrix_bounds}) yields $\|u_h\|_{\mathcal{V}^{\mathrm{ag}}(h)}^2 \gtrsim h^d \lambda_{\mathrm{min},M} \|\mathbf{u} \|^2_2$. It proves the result.
\end{proof}

\section{Interpolation-based discrete extension}\label{sec:interpolation-extension}

Although discrete extension operators relying on Lagrangian FE bases satisfy Def.~\ref{def:discrete-extension-operator}, in practise, they are not a good choice of \ac{fe} basis for high-order approximations. Indeed, the constants in the definitions of Section~\ref{sec:agfem} can depend on the polynomial order. The main problem with Lagrangian \acp{fe} is in the constant in the continuity of $\mathcal{E}_{h}^{\mathrm{ag}}$ in~(\ref{eq:discrete-extension-continuity}). The underlying issue is apparent in the linear constraint~(\ref{eq:agfem-constraints}): Given an ill-posed \ac{dof} $\sigma^\alpha \in \mathcal{D}_{h}^{\mathrm{ipd}}$, its associated aggregate $U \in \mathcal{T}_{h}^{\rm ag}$ is $U = \mathcal{O}_h^{\mathrm{in \to ag}} \circ \mathcal{O}_h^{\mathrm{dof \to in}}(\sigma^\alpha)$. Let $R_U$ be the root cell of $U$. Since we \emph{extrapolate} the root shape functions into the whole aggregate, we have for Lagrangian \ac{fe} bases that
\begin{equation}
  \phi_{\rm lag}^{\beta}(\boldsymbol{x}^{\alpha}) \propto \left(\frac{\mathrm{diam}(U)}{\mathrm{diam}(R_U)}\right)^{d m}, \quad \forall \sigma^\beta \in \mathcal{O}_h^{\mathrm{ipd \to wpd}}(\sigma^\alpha), 
  \label{eq:lag-coeffs}
\end{equation}
where $m$ is the order of the local \ac{fe} space and $d$ the space dimension. In other words, the constraint coefficients blow up with the ratio between the aggregate and root sizes raised to the power of the maximum order of the polynomials in the local \ac{fe} space. This affects the continuity constant of $\mathcal{E}_{h}^{\mathrm{ag}}$ in~(\ref{eq:discrete-extension-continuity}) and leads to severe ill-conditioning for large orders of approximation and/or extrapolation distances. It is also clear that the problem is related to the choice of basis functions, not the method itself. Moreover, it is more relevant when the cut region is large. If the cut cell has small intersection with $\Omega$, then the matrix and vector contributions of the cut cell are also small and suppress the effects of large extrapolation coefficients.

Therefore, our goal is to find a better-conditioned \ac{fe} basis for high-order Ag\ac{fem}. We have three main design criteria: 
\begin{enumerate}
  \item[(a)] \emph{Keep the discrete extension operator:} The new basis should conform to the abstract structure of Ag\ac{fem}. In particular, we want to leverage the same type of discrete extension operator, i.e., based on constraints of the form~(\ref{eq:agfem-constraints}).
  \item[(b)] \emph{Reduce ill-conditioning by extrapolation:} The constraint coefficients $\phi_{\rm new}^{\beta}(\boldsymbol{x}^{\alpha})$ should have a smoother growth with the aggregate-to-root diameter ratio $\mathrm{diam}(U) / \mathrm{diam}(R_U)$; ideally, independent of the \ac{fe} order $m$.
  \item[(c)] \emph{Easy global $\mathcal{C}^0$ continuity:} The new local basis should be suitable for conforming \ac{fe} approximation spaces, i.e., it should be easy to impose continuity across neighbouring cells.
\end{enumerate}

\subsection{Generalised modal $\mathcal{C}^0$ basis in 1D}

Let us begin with the 1D case for simplicity. Let us denote with $\mathcal{E}_h^{\mathrm{Nod}}$ the standard discrete extension operator $\mathcal{E}_h^{\mathrm{ag}}$. Recalling~(\ref{eq:lag-coeffs}), it is easy to check that
\begin{equation}
\underset{v_h \in \mathcal{V}_{h}^{\mathrm{in}}}{\mathrm{max}} 
\frac{\|\mathcal{E}_{h}^{\mathrm{Nod}}(v_h)\|_{L^2(\Omega)}}{\|v_h\|_{L^2(\Omega_{h}^{\mathrm{in}})}} \gtrsim \left( \frac{\mathrm{diam}(U)}{\mathrm{diam}(R_U)}\right)^{m}.
\label{eq:extend-lag}
\end{equation}
This result is obtained by applying the extension operator to the shape functions of a 1D aggregate. The exponent $m$ comes from the extension of the Lagrangian shape functions, which are $m$-th order polynomials. Equation~(\ref{eq:extend-lag}) readily implies that the constants in the continuity bounds in~(\ref{eq:discrete-extension-continuity}) blow up exponentially with $m$. \new{As a result, Def. \ref{def:mass-bounds} does not hold,} leading to serious ill-conditioning issues at high-order.

Let us illustrate this issue with a simple 1D example in Fig.~\ref{fig:cn1d}. \new{The aim of the following test is to show the impact of the order $m$ on the condition number of the mass matrix (and also stiffness matrix).} We consider the \ac{fe} approximation of the Poisson problem~(\ref{eq:poisson-strong}) in a rod $\Omega^{\rm art} \equiv \Omega = (0,1+\theta)$, with $\theta > 0$ a parameter. We prescribe \new{a} homogeneous strong Dirichlet \ac{bc} at $x = 0$, while a Neumann one at $x = 1 + \theta$. The problem is discretised with two \acp{fe} of order $m$; we assume the left cell $(0,1)$ is interior, the right cell $(1,1+\theta)$ is cut and the former is the root of the latter. Note that ``cut'' here is only for classification purposes, i.e., the right cell is not geometrically cut by $\Omega$. In other words, the mesh is body-fitted, $\Omega_h^{\rm act} \equiv \Omega$. In any case, this means that all \acp{dof} (only) in the right cell are constrained by the \acp{dof} of the left cell. Likewise, there is a single aggregate $U$ spanning the whole interval $(0,1+\theta)$. Hence, $\theta$ controls $\mathrm{diam}(U)$, i.e., how far we extrapolate the shape functions at the left cell to constrain the \acp{dof} of the right cell. As shown in Fig.~\ref{fig:cn1d}, if we use Lagrangian \acp{fe} (blue dashed curves), the condition number of the mass and stiffness matrix blow up with both $m$ and $\theta$ very quickly. \new{In this test, we have considered very large ratios between the aggregate and root sizes to show the behavior in extreme cases. However, such large ratios are not common in practice.}

\begin{figure}[ht!]
  \centering
  \includegraphics[width=0.67\textwidth]{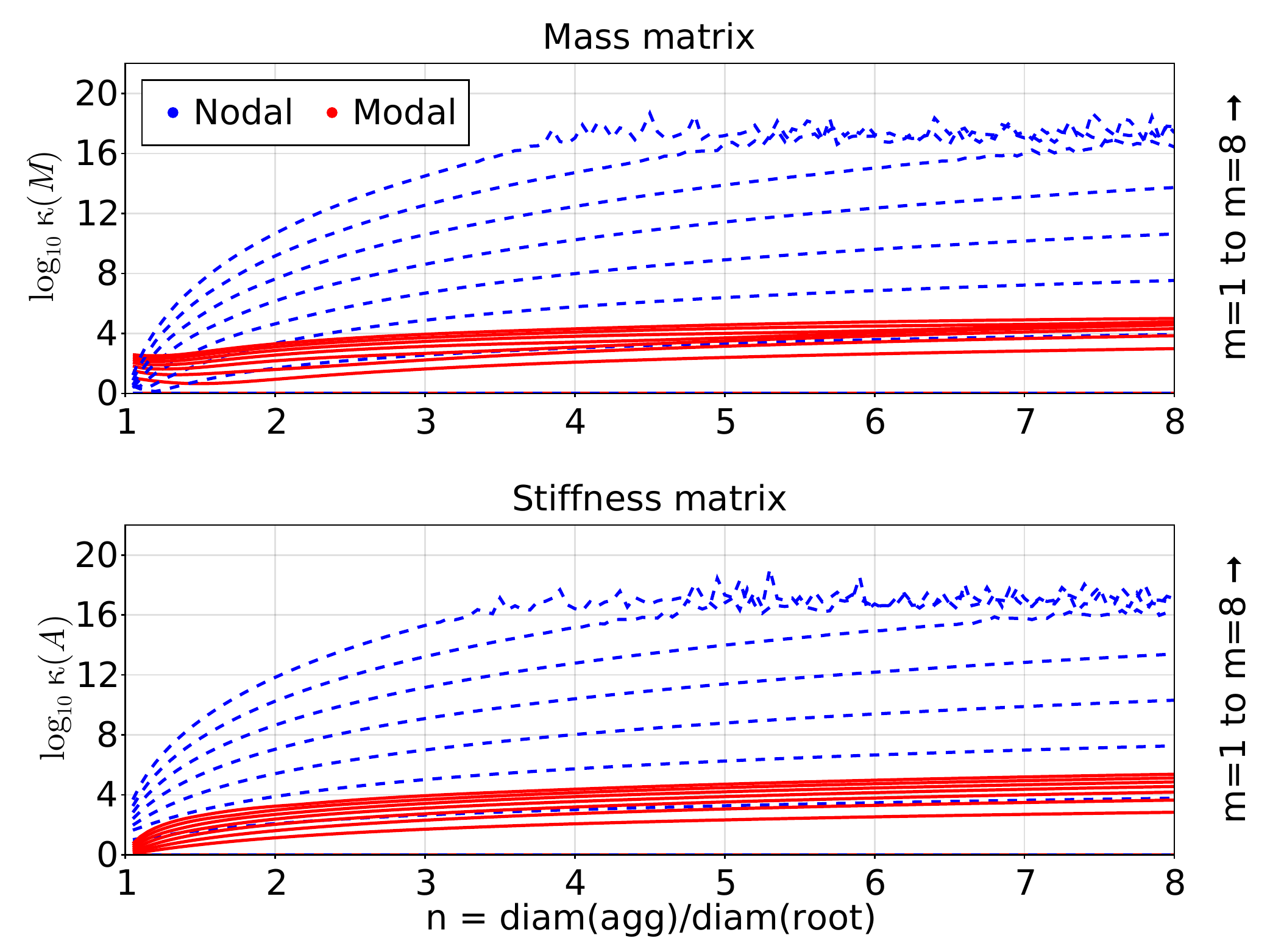}
  \caption{1D Poisson problem in $(0,1+\theta)$, $\theta > 0$. We discretise with two \acp{fe} of order $m$ and define an Ag\ac{fe} space, where the left \ac{fe} $(0,1)$ is the root cell of the right \ac{fe} $(1,1+\theta)$. Thus, $\theta$ controls the distance at which we extrapolate the shape functions of the left (root) cell to constrain the \acp{dof} of the right cell. We represent the condition number of the mass and stiffness matrices, $\kappa(M)$ and $\kappa(A)$, for Lagrangian (blue dashed) and generalised modal $\mathcal{C}^0$ (red) \ac{fe} bases against $\theta$, for different approximation orders $m$. Clearly, modal \acp{fe} are much better conditioned than Lagrangian \acp{fe} for large extrapolation $\theta$ and/or \ac{fe} order.}
  \label{fig:cn1d}
\end{figure}

A natural approach to fulfill the design criteria stated above is to try ``interpolating'' the shape functions within the aggregate, instead of ``extrapolating'' them. With this idea in mind, we realise that most of the high-order extrapolation burden can be transferred to interpolation, using a modified version of the well-established hierarchical modal $\mathcal{C}^0$ \ac{fe} bases. These expansion bases are a classical choice in $hp$-\ac{fem}~\cite{karniadakis2013spectral}. The most common modal $\mathcal{C}^0$ bases are built upon the orthogonal family of Jacobi polynomials. Here, we consider the set of \emph{integrated Legendre} polynomials, typically employed in $p$-FEM~\cite{szabo1991finite} and in the Finite Cell Method~\cite{duster2017p}. Given $\xi \in [0,1]$ and $m \geq 0$, the modal $\mathcal{C}^0$-continuous 1D basis is the set of functions $\{\varphi_0,\ldots,\varphi_m\}$ such that
\begin{equation}
  \varphi_l (\xi) = \begin{cases}
    1 - \xi & \text{if} \ l = 0, \\
    \frac{-\sqrt{2m+1}}{m} \xi(1-\xi) \mathcal{J}_{l-1}^{1,1}(\xi) & \text{if} \ 0 < l < m, \\
    \xi & \text{if} \ l = m,
  \end{cases}
  \label{eq:modalC0}
\end{equation}
where $\mathcal{J}_{n}^{1,1}(\xi)$ is the $n$-th $(1,1)$-Jacobi polynomial. Given $\alpha > -1$ and $\beta > -1$, the Jacobi polynomials $\{\mathcal{J}_{n}^{\alpha,\beta}\}_{n \geq 0}$ are defined by
\[
  \mathcal{J}_{n}^{\alpha,\beta}(t) = \frac{(-1)^n}{n!} 2^{-\alpha-\beta}(1-t)^{-\alpha}t^{-\beta} \frac{d^n}{dt^n} \left( (1-t)^{\alpha+n}t^{\beta+n} \right).
\]
We refer to $\varphi_0$ and $\varphi_m$ as the \emph{nodal modes}, as they coincide with linear Lagrangian 1D shape functions. On the other hand, $\{\varphi_l (\xi)\}_{0<l<m}$ are null at both endpoints; thus, they are referred to as \emph{internal} or \emph{bubble} modes. Besides, $\{\varphi_l (\xi)\}_{0<l<m}$ are scaled to normalise their derivatives such that, by orthogonality of Jacobi polynomials, we have
\[
  \int_0^1 \frac{\mathrm{d}\varphi_i}{\mathrm{d}\xi} \frac{\mathrm{d}\varphi_j}{\mathrm{d}\xi} \mathrm{d}\xi = \delta_{ij}, \qquad i \geq 2 \ \text{and} \ j \geq 0, \quad \text{or vice versa,}
\]

Upon observing the structure of the 1D bubbles, we discover a way to exploit these bases to meet our goals. 1D bubbles are given by the product of the two linear nodal modes against a Jacobi polynomial; in particular, $\varphi_l (\xi) = \varphi_0 (\xi) \varphi_m (\xi) \frac{-\sqrt{2m+1}}{m} \mathcal{J}_{l-1}^{1,1}(\xi)$, $0 < l < m$. It is obvious to see that the product of nodal modes $\varphi_0 \varphi_m$ cancels out $\varphi_l$ at the endpoints, whereas the Jacobi polynomial part does not play any role in that. \emph{Our core idea} is to perturb this factor, such that it is interpolated within the aggregate. This leads to a generalised form of the 1D modal $\mathcal{C}^0$ basis: Given $a, b \in \mathbb{R}$, such that $[0,1] \subset [a,b]$, we denote by $s$ the \new{affine} transformation from $[a,b]$ to $[0,1]$ and we define a generalised version of the 1D polynomial expansion in (\ref{eq:modalC0}) for $\xi \in [a,b]$ as
\begin{equation}
  \varphi_l (\xi) = \begin{cases}
    1 - \xi & \text{if} \ l = 0, \\
    \frac{-\sqrt{2m+1}}{m} \xi(1-\xi) \mathcal{J}_{l-1}^{1,1}(s(\xi)) & \text{if} \ 0 < l < m, \\
    \xi & \text{if} \ l = m.
  \end{cases}
  \label{eq:genmodalC0}
\end{equation}
The only difference with respect to (\ref{eq:modalC0}) is the change of coordinates of the Jacobi polynomial $\mathcal{J}_{l-1}^{1,1}$. Obviously, (\ref{eq:genmodalC0}) forms a polynomial basis. We observe that, in $[a,b]$, the $\varphi_0$ and $\varphi_m$ functions are extrapolated away from $[0,1]$, as in (\ref{eq:modalC0}). However, the Jacobi term is interpolated. Hence, generalised 1D modal $\mathcal{C}^0$ bases satisfy the relation
\begin{equation}
  \max_{\xi \in [a,b]} \left| \varphi_l(\xi) \right| \lesssim \mathrm{diam}(b-a)^{\min(m,2)}.
  \label{eq:mod_extrapolation}
\end{equation}

\begin{proposition}\label{prop:def41-1d}
   The Ag\ac{fe} space $\mathcal{V}_{h}^{\mathrm{ag}}$, built using the 1D generalised modal $\mathcal{C}^0$ basis defined in (\ref{eq:genmodalC0}) as local basis, satisfies Def.~\ref{def:discrete-extension-operator}.
\end{proposition}

\begin{proof}
Let us assume now that $R_U = (0,1)$ and $U = (a,b) \equiv \Omega$ and denote the discrete extension operator applied to the generalised modal $\mathcal{C}^0$ basis as $\mathcal{E}_{h}^{\mathrm{Mod}}$. We note that the definition is identical as the one for $\mathcal{E}_{h}^{\mathrm{Nod}}$. However, the operator is different because its definition depends on the choice of the basis we use in the extension. From the discussion above, we have that
\begin{equation}\label{eq:mod-op-bound}
\underset{v_h \in \mathcal{V}_{h}^{\mathrm{in}}}{\mathrm{max}} 
\frac{\|\mathcal{E}_{h}^{\mathrm{Mod}}(v_h)\|_{L^2(\Omega)}}{\|v_h\|_{L^2(\Omega_{h}^{\mathrm{in}})}} \lesssim \left( \frac{\mathrm{diam}(U)}{\mathrm{diam}(R_U)}\right)^{\min (m,2)},
\end{equation}
which can be obtained using the fact that the high-order bases are interpolated (using the ideas in \cite{Badia2018}). We can proceed analogously to show stability of the extension of gradients. The optimal convergence properties can also be proved following~\cite{Badia2018}. 
\end{proof}

Thus, the method satisfies Def.~\ref{def:discrete-extension-operator} with continuity constants that scale with the order $m$ at most quadratically. This is a clear improvement w.r.t.~Lagrangian \ac{fe} bases~(\ref{eq:extend-lag}).

\begin{assumption}
  The Ag\ac{fe} space $\mathcal{V}_{h}^{\mathrm{ag}}$, built using the 1D generalised modal $\mathcal{C}^0$ basis defined in (\ref{eq:genmodalC0}) as local basis, satisfies Def.~\ref{def:mass-bounds}.
\end{assumption}

In order to check this assumption, we observe if the mass matrix satisfies (\ref{eq:mass_matrix_bounds}). We have computed the condition number of the local mass matrix corresponding to the 1D problem of Fig.~\ref{fig:cn1d}, in which $U \equiv \Omega$ and also $\Omega \equiv \Omega_h^{\rm act}$. The condition numbers of the mass matrix obtained for different orders have a very mild dependence with $m$, and thus, satisfy the assumption for 1D bases. Additionally, we expose that the condition numbers of both the mass and stiffness matrices with modal $\mathcal{C}^0$ bases are much lower and better behaved with $m$ than with Lagrangian bases. For the general case $\Omega \subsetneq \Omega_h^{\rm act}$, we refer to the eigenspectrum convergence tests in \new{Section~\ref{sub:exp_eigen} (see Fig.~\ref{fig:eigenextrematest} and Fig.~\ref{fig:eigenbboxes}).} 

\subsection{Generalised modal $\mathcal{C}^0$ bases in multiple dimensions}

Standard tensor product extends modal $\mathcal{C}^0$ shape functions to $d$-cubes. We can also truncate the tensor product as usual to build \emph{trunk space}, a.k.a.~\emph{serendipity}, variants~\cite{duster2017p}. A well-known property of these extensions is the inherent decomposition of the basis functions into $k$-face modes, $0 \leq k \leq d$. For instance, for $d = 2$, we have \emph{vertex} ($0$-face), \emph{edge} ($1$-face) and \emph{face} ($2$-face) modes, see Fig.~\ref{fig:shapefuns_h}. Vertex modes are the 2D linear Lagrangian shape functions. Edge modes, restricted to the boundary, are null everywhere, except in the interior of a single edge. Face modes are internal modes, i.e., null at the boundary. This type of decomposition is particularly convenient to define globally $\mathcal{C}^0$-continuous basis functions. Indeed, given a $k$-face $C$ in $\mathcal{C}_h^{\mathrm{act}}$, $0 \leq k < d$, we only need to enforce continuity of all $k'$-face modes, $0 \leq k' \leq k$, in the closure of $C$. This is done by matching the shape of all individual local basis functions, in the same way as done with Lagrangian \acp{fe}.

In order to extend generalised modal $\mathcal{C}^0$ bases for $d > 1$ via tensor product, we can proceed analogously to the 1D case, i.e., by perturbing the Jacobi factor. However, the multidimensional case is more involved. Due to the tensor product, the domain in which we want to interpolate the Jacobi terms can only have the form of a cartesian product of 1D intervals. On the other hand, shape functions that belong to $k$-faces of interior cells with $k > 0$ can have support on one or multiple root cells (and thus aggregates). Hence, the high-order terms are to be interpolated in \acp{aabb} $\mathcal{B}$ of one or several aggregates. Clearly, due to the modal decomposition, the bounding box $\mathcal{B}$ is $k$-face dependent. Finally, the aggregate does not generally have the same \emph{shape} as the root cell, e.g., it is not an $n$-cube. Thus, given an aggregate $U$, the smallest \ac{aabb} $\mathcal{B}_U$ of the aggregate is larger than $U$ in general. (In the 1D case, both aggregates and root cells are 1D segments and $\mathcal{B}_U = U$.)

According to this, we extend the concept of \acp{aabb} into the context of modal $\mathcal{C}^0$ bases as follows.

\begin{definition}\label{def:modalC0aabb}
  Given a set $W \subset \mathbb{R}^d$ and $\mathcal{B}$ denoting the (geometrical) \ac{aabb} of $W$, 
   the (modal $\mathcal{C}^0$) $\mathcal{B}^{\rm mod}$ \ac{aabb} is given by
  \begin{itemize}
    \item $\mathcal{B}^{\rm mod}(T) \doteq \mathcal{B}(\mathcal{O}_h^{\rm in \to ag}(T) \cap \Omega)$, if $T \in \mathcal{T}_h^{\rm in}$, and
    \item $\mathcal{B}^{\rm mod}(C) \doteq \mathcal{B}( \cup_{T \in \mathcal{T}_C^{\rm in}} \mathcal{O}_h^{\rm in \to ag}(T) \cap \Omega)$, if $C \in \mathcal{C}_h^{\rm in}$, 
  \end{itemize}
  where $\mathcal{T}_C^{\rm in} \subset \mathcal{T}_h^{\rm in}$ is the set of interior cells that contain $C$.
\end{definition}

We refer to Fig.~\ref{fig:shapefuns_a}-\ref{fig:shapefuns_f} for some examples of $\mathcal{B}^{\rm mod}$. 

\begin{remark}
  We note that adjusting the $\mathcal{B}^{\rm mod}$ bounding box to the physical domain $\Omega$ in Def.~\ref{def:modalC0aabb}, instead of using the whole active cells $\Omega_h^{\rm act}$, has a very positive impact \new{on} the condition number when sliver cuts are present, see Fig.~\ref{fig:eigenbboxes} corresponding to the eigenextrema convergence tests of Section~\ref{sub:exp_eigen}. It also means that we are partially extrapolating the high-order terms when computing the constraints, but \emph{always} interpolating when evaluating functions inside the domain $\Omega$, e.g., when integrating the weak form.
  \label{rem:bboxes}
\end{remark}

We can now use $\mathcal{B}^{\rm mod}$ to define the $d$-dimensional change of coordinates $s$ for each $k$-face mode of the generalised modal $\mathcal{C}^0$ basis. The goal is to interpolate the Jacobi polynomial term inside the whole region in which the $k$-face mode is extended by aggregation.

\begin{definition}\label{def:shapefuns}
Let us consider the $d$-dimensional tensor product of generalised modal $\mathcal{C}^0$ basis functions (\ref{eq:genmodalC0}). Given a $k$-face mode, $0 \leq k \leq d$, and $T \in \mathcal{T}_h^{\rm in}$, the $d$-dimensional change of coordinates $s$ is defined as follows:
\begin{itemize}
  \item if $k = 0$, $0$-modes are the product of nodal modes; they remain unchanged (map $s$ does not apply).
  \item if $k = d$, $d$-modes are internal. As a result, it suffices to take $s: \mathcal{B}^{\rm mod}(T) \to T$ and evaluate the unidirectional Jacobi polynomial terms at the corresponding component of $s$.
  \item if $0 < k < d$, let $C \in \mathcal{C}_h^{\rm in}$ denote the $k$-face associated to the $k$-mode. In this case, $s: \mathcal{B}^{\rm mod}(C) \to C$ and we proceed as in the previous case.
\end{itemize}
\end{definition}

By construction, the transformations $s$ do not alter global $\mathcal{C}^0$-continuity of the basis, because we consider unique $s$ mappings for each $C \in \mathcal{C}_h^{\rm in}$. Likewise, they allow to extend the Jacobi polynomial terms by interpolation, since $\mathcal{B}^{\rm mod}$ encloses the aggregate (or aggregates) where the local $k$-mode can be evaluated.

\begin{figure}[ht!]
  \centering
  \begin{subfigure}{0.32\textwidth}
    \centering
    \includegraphics[width=0.9\textwidth]{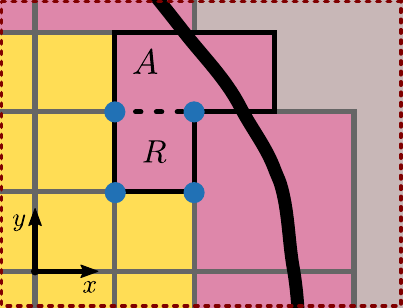}
    \caption{$\mathcal{B}^{\rm mod}$ for corner modes}
    \label{fig:shapefuns_a}
  \end{subfigure}
  \begin{subfigure}{0.32\textwidth}
    \centering
    \includegraphics[width=0.9\textwidth]{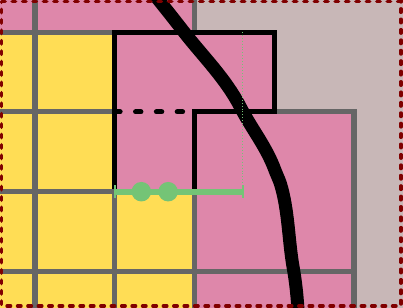}
    \caption{$\mathcal{B}^{\rm mod}$ for edge 1 modes}
    \label{fig:shapefuns_b}
  \end{subfigure}
  \begin{subfigure}{0.32\textwidth}
    \centering
    \includegraphics[width=0.9\textwidth]{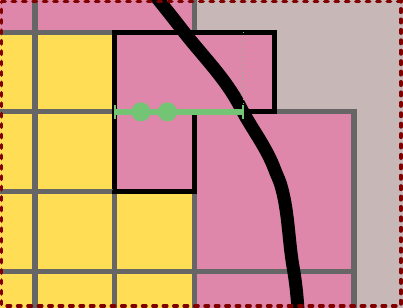}
    \caption{$\mathcal{B}^{\rm mod}$ for edge 2 modes}
    \label{fig:shapefuns_c}
  \end{subfigure} \\ \vspace{0.3cm}
  \begin{subfigure}{0.32\textwidth}
    \centering
    \includegraphics[width=0.9\textwidth]{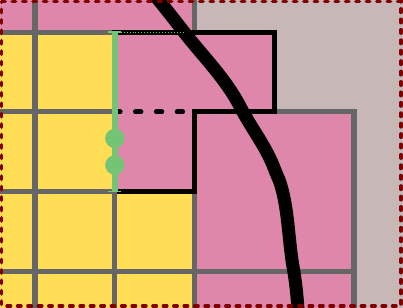}
    \caption{$\mathcal{B}^{\rm mod}$ for edge 3 modes}
    \label{fig:shapefuns_d}
  \end{subfigure}
  \begin{subfigure}{0.32\textwidth}
    \centering
    \includegraphics[width=0.9\textwidth]{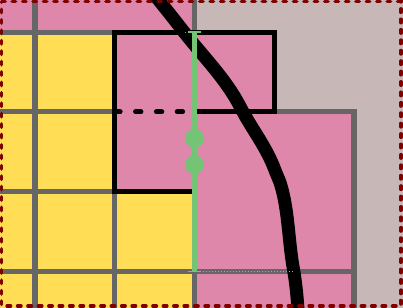}
    \caption{$\mathcal{B}^{\rm mod}$ for edge 4 modes}
    \label{fig:shapefuns_e}
  \end{subfigure}
  \begin{subfigure}{0.32\textwidth}
    \centering
    \includegraphics[width=0.9\textwidth]{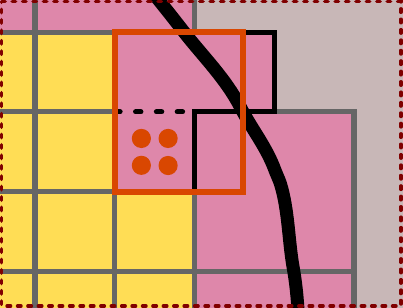}
    \caption{$\mathcal{B}^{\rm mod}$ for face modes}
    \label{fig:shapefuns_f}
  \end{subfigure} \\
  \begin{subfigure}{0.49\textwidth}
    \centering
    \includegraphics[width=0.9\textwidth]{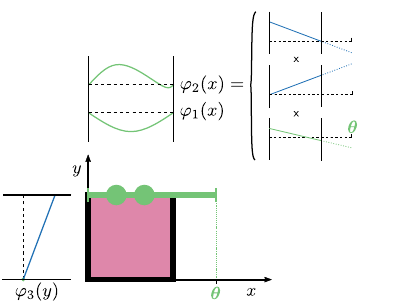}
    \caption{Tensor product for edge 1 modes}
    \label{fig:shapefuns_g}
  \end{subfigure}
  \begin{subfigure}{0.49\textwidth}
    \centering
    \includegraphics[width=0.9\textwidth]{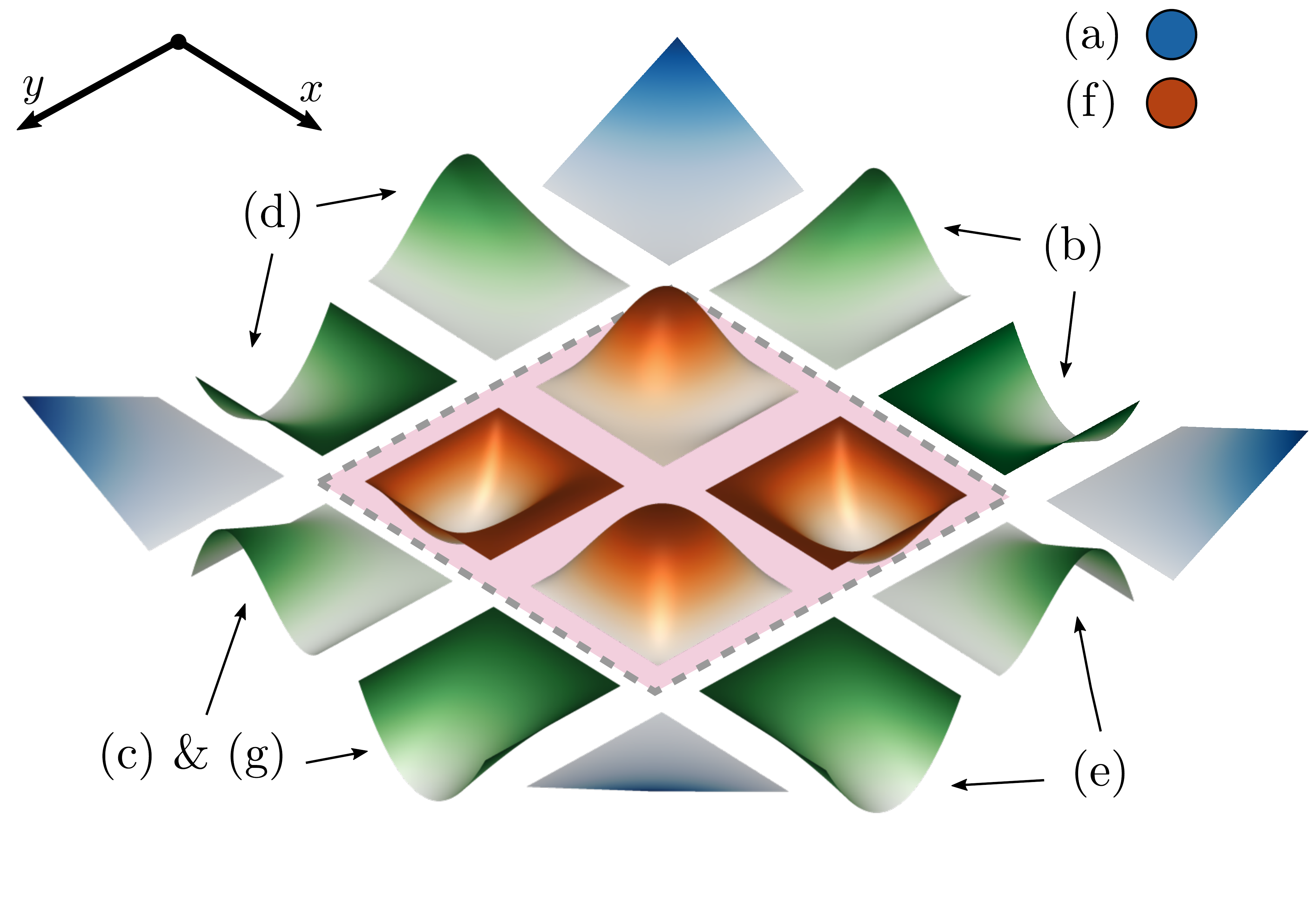}
    \caption{Modal decomposition of root basis functions}
    \label{fig:shapefuns_h}
  \end{subfigure}
  \caption{\new{ Illustration of the main building blocks to construct a generalised 2D modal $\mathcal{C}^0$ polynomial basis on the square for $m = 3$. Circles denote nodes in the interior of the corresponding vertex, edge, or face. The root cell $R$ and the aggregate $A$ correspond to the ones depicted in Fig.~\ref{fig:defs_maps}. } }
  \label{fig:shapefuns}
\end{figure}

It remains to prove that the discrete extension operator defined with generalised modal $\mathcal{C}^0$ expansions leads to a suitable Ag\ac{fe} space and, more crucially, the constants bounding the operator do not \emph{exponentially} depend on the order of approximation $m$.

\begin{proposition}\label{prop:def41-nd}
   The Ag\ac{fe} space $\mathcal{V}_{h}^{\mathrm{ag}}$, built using the multi-dimensional generalised modal $\mathcal{C}^0$ basis in Def.~\ref{def:shapefuns} as local basis, satisfies Def.~\ref{def:discrete-extension-operator}.
\end{proposition}

\begin{proof}
The proof that the multidimensional extension operator satisfies Def.~\ref{def:discrete-extension-operator} follows the same line as in 1D. By construction, the discrete extension operator is interpolatory (for higher order modes). Thus, we get the stability properties in (\ref{eq:mod-op-bound}).
\end{proof}

\begin{assumption}
  The Ag\ac{fe} space $\mathcal{V}_{h}^{\mathrm{ag}}$, built using the 1D generalised modal $\mathcal{C}^0$ basis in Def.~\ref{def:shapefuns} as local basis, satisfies Def.~\ref{def:mass-bounds}.
\end{assumption}

Checking condition number bounds for the mass matrices at the aggregates in multiple dimensions is more involved. The upper bound in (\ref{eq:mass_matrix_bounds}) can readily be obtained. However, a proof of a general lower bound is elusive since, in contrast with standard \ac{fem}, the unfitted method can have arbitrary topologies of aggregates. The main issue is the dependence of the lower bound on the vast amount of possible geometrical configurations that we can have for $\mathcal{B}^{\rm mod}$ and its intersection with $\Omega$. Therefore, for simplicity, we check experimentally this assumption for a  similar scenario to the one in the 1D example of Fig.~\ref{fig:cn1d}, where we integrate the shape functions over the whole aggregate $U$.

In this simplified setting, it is easy to find a worst-case scenario. For the sake of brevity, we limit ourselves to a representative 2D case. We consider the Poisson problem~(\ref{eq:poisson-strong}) in a square $\Omega_h^{\rm art} \equiv \Omega = [0,n]^2$, with $n > 0$ an integer parameter. We prescribe \new{a} homogeneous Dirichlet \ac{bc} at the bottom and left sides, \new{and a} homogeneous Neumann \new{\ac{bc}} at the top and right sides. The problem is discretised with a uniform $n \times n$ Cartesian grid of \acp{fe} of order $m$; we assume the bottom left cell $[0,1]^2$ is interior, the rest of cells are cut (constrained) and the former is their root (constraining). As in Fig.~\ref{fig:cn1d}, ``cut'' is only meant for classification purposes, i.e., there are no cells geometrically cut by $\Omega$. Therefore, there is a single aggregate $U$ and we also have that $\mathcal{B}^{\rm mod}(U) = U$. Hence, $n$ controls $\mathrm{diam}(U)$, i.e., how far we extrapolate the shape functions at the root cell to constrain the \acp{dof} of the rest of cells. As shown in Fig.~\ref{fig:cn2d}, \new{the condition number of the mass (and also stifness matrix) does not depend exponentially on $m$, thus satisfying the requirement in Def. \ref{def:mass-bounds}.}

Additionally, we plot the condition numbers obtained with Lagrangian bases. We reach the same conclusion as in the 1D example. In contrast to Lagrangian bases, the rate of growth of condition number with extrapolation distance for modal $\mathcal{C}^0$ bases is significantly lower and independent of $m$, $m > 1$. We also note that stretching high-order Jacobi terms at very high aggregate-to-root size ratios (e.g., >4) could lead to linear dependency issues, especially in 3D. However, high aggregate-to-root size ratios are not expected, because they imply lack of mesh resolution at the boundary of the geometry. In any case, this issue can be easily mitigated by controlling aggregate size with $h$-refinement~\cite{Badia2020Jun}.

\begin{figure}[ht!]
  \centering
  \begin{subfigure}{0.25\textwidth}
    \centering
    \includegraphics[width=0.9\textwidth]{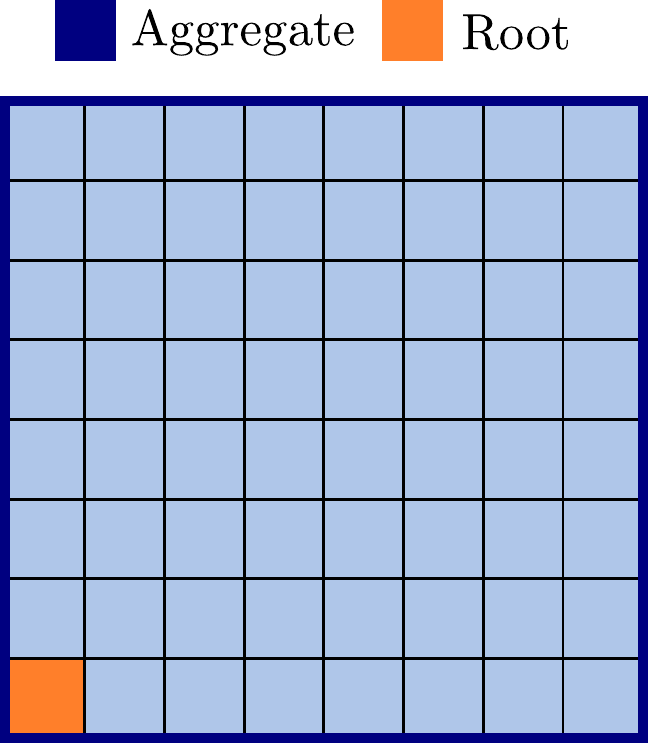}
    \caption{Setup for $n = 8$}
  \end{subfigure}
  \begin{subfigure}{0.73\textwidth}
    \centering
    \includegraphics[width=0.9\textwidth]{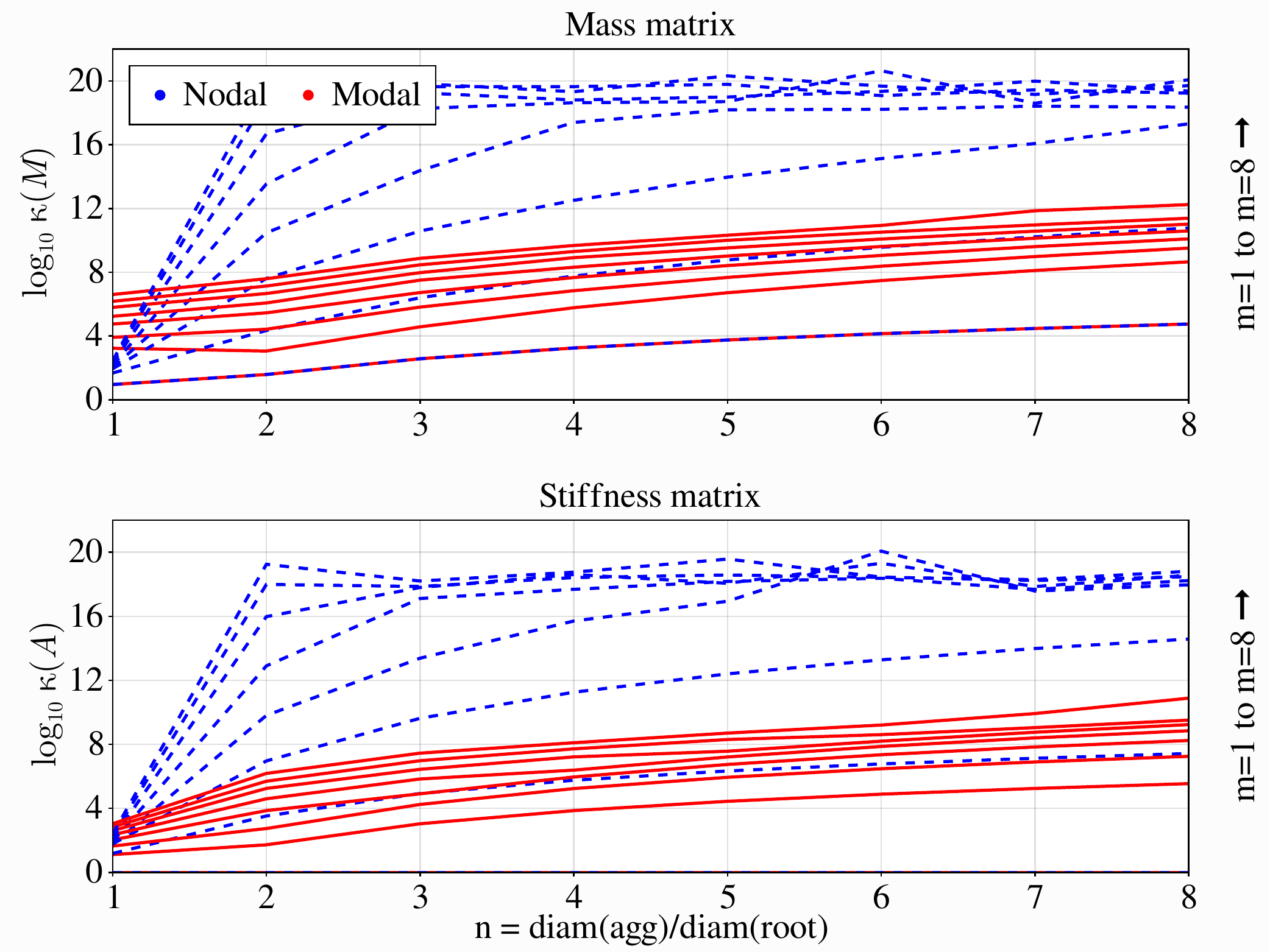}
    \caption{Condition number plots}
  \end{subfigure}
  \caption{2D Poisson problem in $[0,n]^2$, $n > 0$. We discretise with an $n \times n$ uniform Cartesian grid with \acp{fe} of order $m$ and define an Ag\ac{fe} space, where the orange cell in (\textsc{a}) is the root cell of the blue cells. Thus, $n$ controls the distance at which we extrapolate the shape functions of the root cell to constrain the \acp{dof} of the rest of cells. We represent the condition number of the mass $\kappa(M)$ and stiffness $\kappa(A)$ matrices for tensor-product Lagrangian nodal and generalised modal $\mathcal{C}^0$ \ac{fe} bases against $n$, for different approximation orders $m$. Clearly, modal \acp{fe} are much better conditioned than Lagrangian \acp{fe} for large extrapolation ($n$) and/or \ac{fe} order ($m$). We note that values of diam(agg)/diam(root) above 3 or, exceptionally, 4 are not expected in practical simulations; we just show them for illustration purposes. Values of diam(agg)/diam(root) above 4 imply the background mesh has probably not enough mesh resolution to capture the geometric features of $\Omega$ and should be refined.}
  \label{fig:cn2d}
\end{figure}

Based on these numerical results, we observe that the assumption holds for the case being considered
while the upper bound blows up exponentially with the polynomial order for standard Lagrangian bases.  More complex numerical experiments that show the good behaviour of the proposed basis compared to standard bases can be found in Section~\ref{sec:numerical_experiments}.

\section{Numerical experiments}\label{sec:numerical_experiments}

\subsection{Methods and parameter space}
\label{sub:exp_params}

We solve the Poisson problem~(\ref{eq:poisson-strong}) with weak Dirichlet boundary conditions everywhere and the elasticity problem~(\ref{eq:elasticity-strong}) with Neumann and strong Dirichlet boundary conditions. We use the discrete approximations~(\ref{eq:poisson-weak}) and~(\ref{eq:elasticity-weak}) and set $f$ and $g$, such that the solution to the problem is $u(\boldsymbol{x})=\left(\sum_{i=1}^d x_i\right)^{s}$, where $s = 1$ or $m+1$, with $m$ the order of the \ac{fe} space at hand. We refer to the solutions for $s = 1$ or $s = m+1$ as the \emph{in-\ac{fe}-space} and \emph{out-\ac{fe}-space} solutions. Table~\ref{tab:params} collects all simulation parameters. We consider six different (level-set) embedded geometries, as shown in Fig.~\ref{fig:geoms}. In 2D, (a) a disk, (b) a square and (c) a letter G. In 3D, (d) a torus, (e) a cube and (f) a spherical object. (c) and (f) are both \ac{csg} objects~\cite{requicha1977constructive}, i.e., obtained by merging, subtracting and intersecting elementary primitive (level-set) geometries, such as cubes or cylinders. In the linear elasticity case, we simulate (a), (b), (d) and (e) only in the first (positive) orthant. In this way, we can apply strong Dirichlet conditions on the body-fitted boundary and Neumann conditions on the cut boundary.

\begin{table}[ht!]
	\centering
	\begin{small}
		\begin{tabular}{ll}
			\toprule
			Description & Considered methods/values \\
			\midrule
			Model problem & Poisson equation~(\ref{eq:poisson-strong}); \\ 
                    & (compressible) Linear elasticity~(\ref{eq:elasticity-strong})
			\vspace{0.12cm} \\
			Boundary conditions & Nitsche's method (Poisson equation) \\
                          & Neumann and strong Dirichlet (Linear elasticity) \vspace{0.12cm} \\
			Analytical solution & $u(\boldsymbol{x})=\left(\sum_{i=1}^d x_i\right)^{s}$,\\
                          & $s=1$ or $m+1$, with $m$ the \ac{fe} interpolation order, \\
                          & $s=1 \to$ \emph{in-\ac{fe}-space} and $s=m+1 \to$ \emph{out-\ac{fe}-space} 
	  \vspace{0.12cm} \\
			Problem geometry & 2D: disk, square, letter G; \\ 
        & 3D: torus, cube, spherical \ac{csg} object 
			\vspace{0.12cm} \\ 
			Interpolation & $m$-th order globally $\mathcal{C}^0$ \acp{fe} on uniform Cartesian grids, \\
      (and discrete extension) & using (i) tensor-product Lagrangian and (ii) trunk-space \\ & modal $\mathcal{C}^0$ 
      bases (\ref{eq:genmodalC0}) up to $m = 5$ \vspace{0.12cm} \\
      Approximation space & $\mathcal{V}_h^{\rm ag}$, i.e., (strong) \ac{agfem} \vspace{0.12cm} \\
			Coef.~in Nitsche's penalty term & $\beta = 12.0 \ m^2$ \\
			\bottomrule
		\end{tabular}
	\end{small}
	\caption{Summary of the parameters and computational strategies in	the numerical examples.}
	\label{tab:params}
\end{table}

\begin{figure}[ht!]
  \centering
  \begin{subfigure}{0.20\textwidth}
    \centering
    \includegraphics[width=0.8\textwidth]{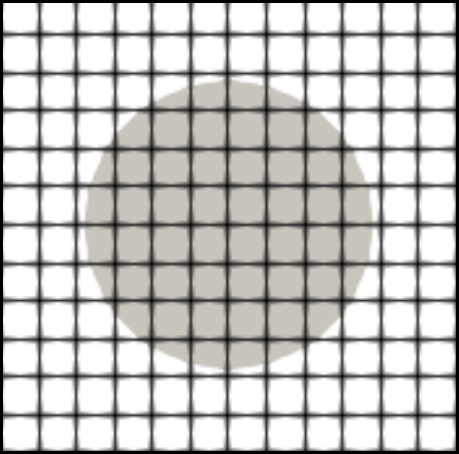}
    \caption{Disk}
  \end{subfigure}
  \begin{subfigure}{0.20\textwidth}
    \centering
    \includegraphics[width=0.8\textwidth]{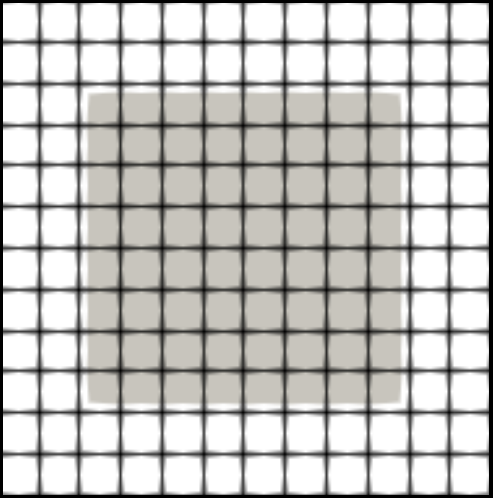}
    \caption{Square}
  \end{subfigure}
  \begin{subfigure}{0.20\textwidth}
    \centering
    \includegraphics[width=0.8\textwidth]{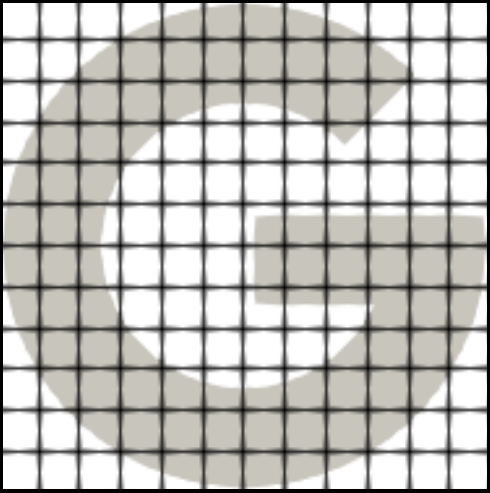}
    \caption{Letter G}
  \end{subfigure} \\
  \begin{subfigure}{0.20\textwidth}
    \centering
    \includegraphics[width=0.8\textwidth]{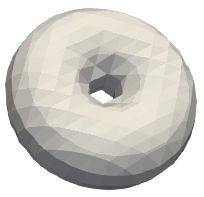}
    \caption{Torus}
  \end{subfigure}
  \begin{subfigure}{0.20\textwidth}
    \centering
    \includegraphics[width=0.8\textwidth]{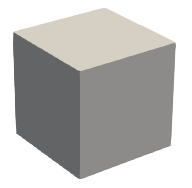}
    \caption{Cube}
  \end{subfigure}
  \begin{subfigure}{0.20\textwidth}
    \centering
    \includegraphics[width=0.8\textwidth]{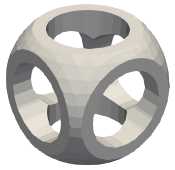}
    \caption{CSG}
  \end{subfigure}
  \caption{Geometries considered for the Poisson problem~(\ref{eq:poisson-weak}). 2D plots also represent the background uniform Cartesian grid.}
  \label{fig:geoms}
\end{figure}

All discretisations are defined on uniform Cartesian background meshes and the embedded geometry is represented with first order cuts. Thus, we neglect geometrical errors in the analysis of the results; we only report functional approximation errors. We study high-order approximations up to order $m = 5$ in $\mathcal{V}_h^{\rm ag}$. We considered both tensor-product and trunk-space Ag\ac{fe} bases, but we only report the best variants: tensor-product Lagrangian and (generalised) trunk-space modal $\mathcal{C}^0$ as given in~(\ref{eq:genmodalC0}); we argue this choice at the end of Section~\ref{sec:stat_cond}. Apart from computing cell aggregates, with the cell aggregation algorithm described in~\cite{Badia2018}, we also compute the $\mathcal{B}^{\rm mod}$ \acp{aabb} for each $k$-face $C \in \mathcal{C}_h^{\rm in}$ and each interior cell $T \in \mathcal{T}_h^{\rm in}$, in order to generate the Ag\ac{fe} spaces with modal $\mathcal{C}^0$ bases.

\subsection{Integration of high order polynomials on polyhedra}\label{sec:integration}

In order to carry out the numerical experiments devised in Table~\ref{tab:params}, we need 2D and 3D high-order numerical quadratures to integrate the high-order (Lagrangian and modal $\mathcal{C}^0$) shape functions \emph{on the cut cells}. We recall that, on each cut cell, we linearly approximate the level-set functions representing the embedded geometry. Hence, we need to come up with an appropriate strategy to integrate high-order polynomials on polygons and polyhedra.

Accurate and efficient high-order 2D and 3D numerical quadratures on implicitly defined domains is still an open topic in the literature, especially, for smooth implicit domain approximations. We refer the reader to \new{the} recent state-of-the-art \new{overviews} in~\cite{DIVI20202481,saye2022high}. In this work, we do not aim to innovate in this area but we cannot choose the numerical quadratures carelessly. We need an approach that offers good compromise between accuracy and efficiency, such that it is suitable for the numerical comparison of the high-order Ag\ac{fe} methods.

We have considered two different approaches for numerical integration on cut cells. Without loss of generality, we describe them for integrals on the cut region $T \cap \Omega$, $T \in \mathcal{T}_h^{\rm cut}$. The first method is analogous to the one described in~\cite{burman_cutfem_2015} for tetrahedral meshes. We leverage a marching cubes algorithm~\cite{Lorensen_1987} to generate subtriangulations of the cut cells. Then, it suffices to generate standard quadrature rules on each simplex of the subtriangulation as, e.g., the by-product of mapping to each simplex a reference Gaussian quadrature on the unit simplex.

We discuss next the efficiency of this method. We recall that $n$-point 1D Gaussian quadrature rules are exact for polynomials of degree up to $2n - 1$. Thus, we need $\lceil (q+1) / 2 \rceil$ points for exact integration of 1D polynomials up to order $q$, where $\lceil \cdot \rceil$ denotes the ceiling function. Let us now look at the amount of points required to integrate exactly mass matrices of $d$-dimensional standard tensor product polynomial bases of degree $m$ in a cut cell $T \in \mathcal{T}_h^{\rm cut}$. In contrast with hexahedral cells, we cannot define product measures on the simplices of the $T$-subtriangulation. Thus, we cannot apply the Fubini theorem to transform double or triple integrals into iterated 1D integrals. \new{As a result, we cannot consider tensor product quadratures on the simplices. Since we have multi-variable monomials of total order $2md$}, we deduce that we need (at least) $\left( \lceil (2md+1) / 2 \rceil^d \right) n_{\rm \Delta,T}$ points in each $T \in \mathcal{T}_h^{\rm cut}$, where $n_{\rm \Delta,T}$ is the number of $\Omega$-interior simplices of the $T$-subtriangulation.

Upon realising that the previous approach is rather inefficient for high-order approximations, especially, in 3D, we resort to moment-fitting methods~\cite{M_ller_2013,Sudhakar_2013,Mousavi_2009}. In this method, given $T \in \mathcal{T}_h^{\rm cut}$, we compute the pair of positions and weights $(\boldsymbol{x}_i,w_i)$ of an $n$-point quadrature rule, by solving the moment equations
\begin{equation}\label{eq:moment-fitted}
  \sum_{i=1}^n \varphi_j (\boldsymbol{x}_i) w_i = \int_{T \cap \Omega} \varphi_j (\boldsymbol{x}) \ \mathrm{d}T, \quad j = 1, \ldots, o,
\end{equation}
where $\{\varphi_j (\boldsymbol{x})\}_{1<j<o}$ are $o$ linearly independent basis functions. In general,~(\ref{eq:moment-fitted}) defines a rectangular nonlinear system of equations. However, as done in~\cite{Hubrich_2019}, we fix $\boldsymbol{x}_i$ as the (tensor product) Gauss-Legendre nodes and $\varphi_j (\boldsymbol{x})$ as their associated Lagrange polynomial functions. Using this combination, we can reduce~(\ref{eq:moment-fitted}) to a diagonal linear system, due to the Kronecker delta property of Lagrange polynomials. In this case, a $d$-dimensional exact quadrature for the mass matrices on cut cells requires $(2m + 1)^d$ points, only. Thus, we obtain a significantly more efficient quadrature, which can be directly defined on the cut cell, instead of defining a different one for each simplex of its subtriangulation. The major drawback of this methodology is that quadrature weights are not generally strictly positive, because the Lagrange polynomial can be mostly negative on the interior cut region. This means that numerical integration with these quadratures can be ill-conditioned and incur in accuracy losses due to, e.g., cancelling errors, especially at high order. 

It remains to see how to compute the right-hand side of~(\ref{eq:moment-fitted}), i.e., the Lagrangian moments corresponding to the quadrature weights $w_i$. Here, we leverage an extension of Lasserre's method~\cite{Chin_2015}. It establishes a way to reduce volumetric integrals of monomials on convex and nonconvex polytopes down to applying a cubature rule, where the points are the vertices of the polytope. Using this approach, we compute first the integrals on the cut region of a tensor product monomial basis. From here, it would suffice to use a change of basis to transform the monomial moments into the sought-for Lagrangian moments in~(\ref{eq:moment-fitted}). However, this potentially leads to inaccurate results, due to inverting the (severely ill-conditioned) monomial-to-Lagrange Vandermonde matrix. One way to bypass this issue is to carry out an intermediate transformation into (tensor-product) Legendre moments, because the monomial-to-Legendre and Legendre-to-Lagrangian transformation matrices are much better conditioned than the Vandermonde one. 

Fig.~\ref{fig:1qc} compares the quadrature sizes and accuracy of the marching cubes and moment fitted methods with increasing order of approximation $m$. For the plot, we solve the Poisson problem~(\ref{eq:poisson-weak}) on the disk and the cube, embedded in a uniform grid of $12^d$ cells. We consider the \emph{in-\ac{fe}-space} solution and we gather the total number of quadrature points in the cut mesh $\mathcal{T}^{\mathrm{cut}}$, as well as the $L^2$ and $H^1$ error of the solution in $\Omega$ approximated with trunk-space modal $\mathcal{C}^0$ bases. The results clearly show that moment-fitted quadratures are much more efficient than marching cube ones, while offering the same level of accuracy up to $m = 5$. Hence, they are selected over the latter for the numerical experiments that follow.

\begin{figure}[ht!]
  \centering
  \begin{subfigure}{0.46\textwidth}
    \centering
    \includegraphics[width=\textwidth]{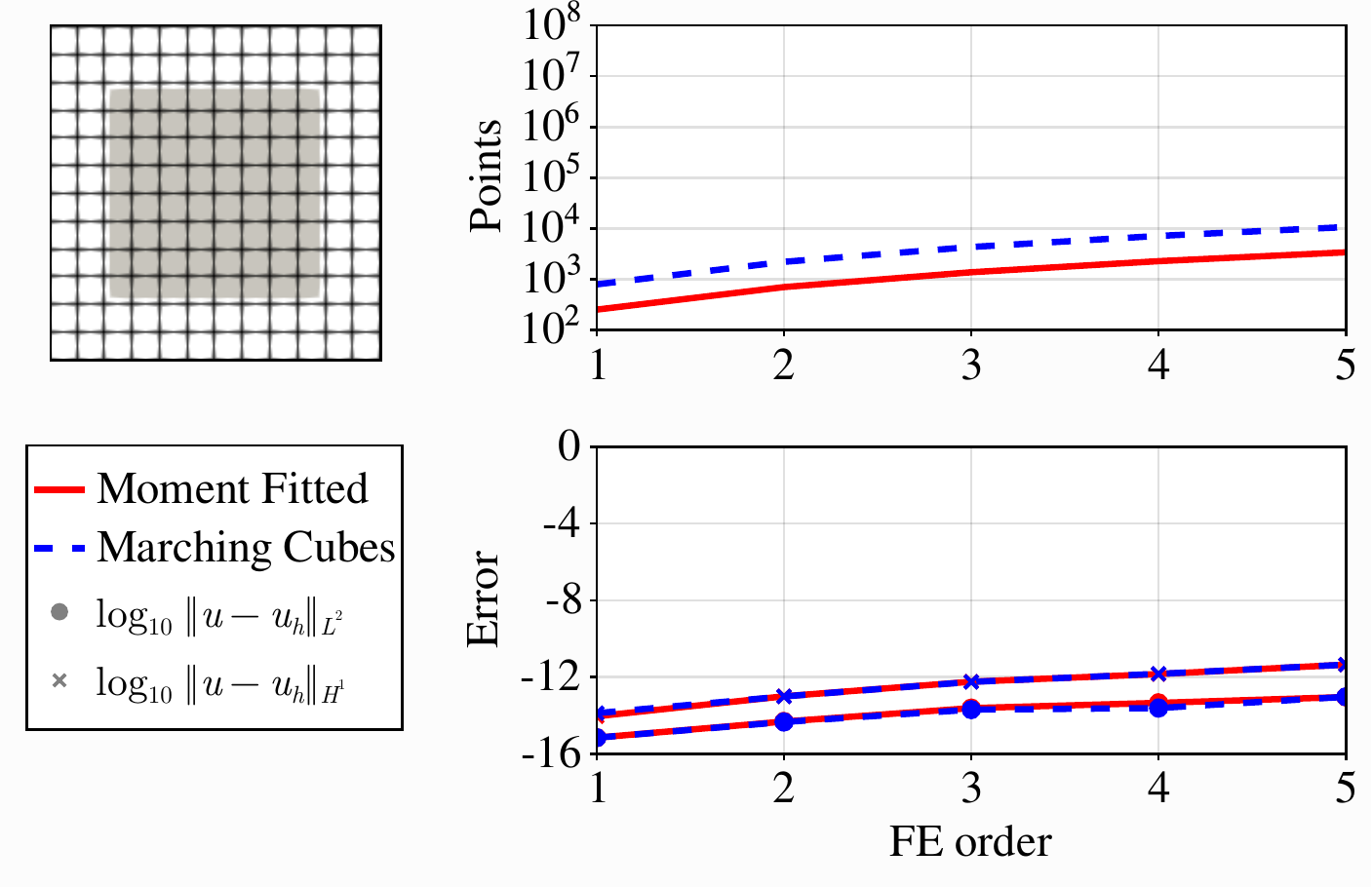}
    \caption{2D case}
  \end{subfigure} \quad
  \begin{subfigure}{0.46\textwidth}
    \centering
    \includegraphics[width=\textwidth]{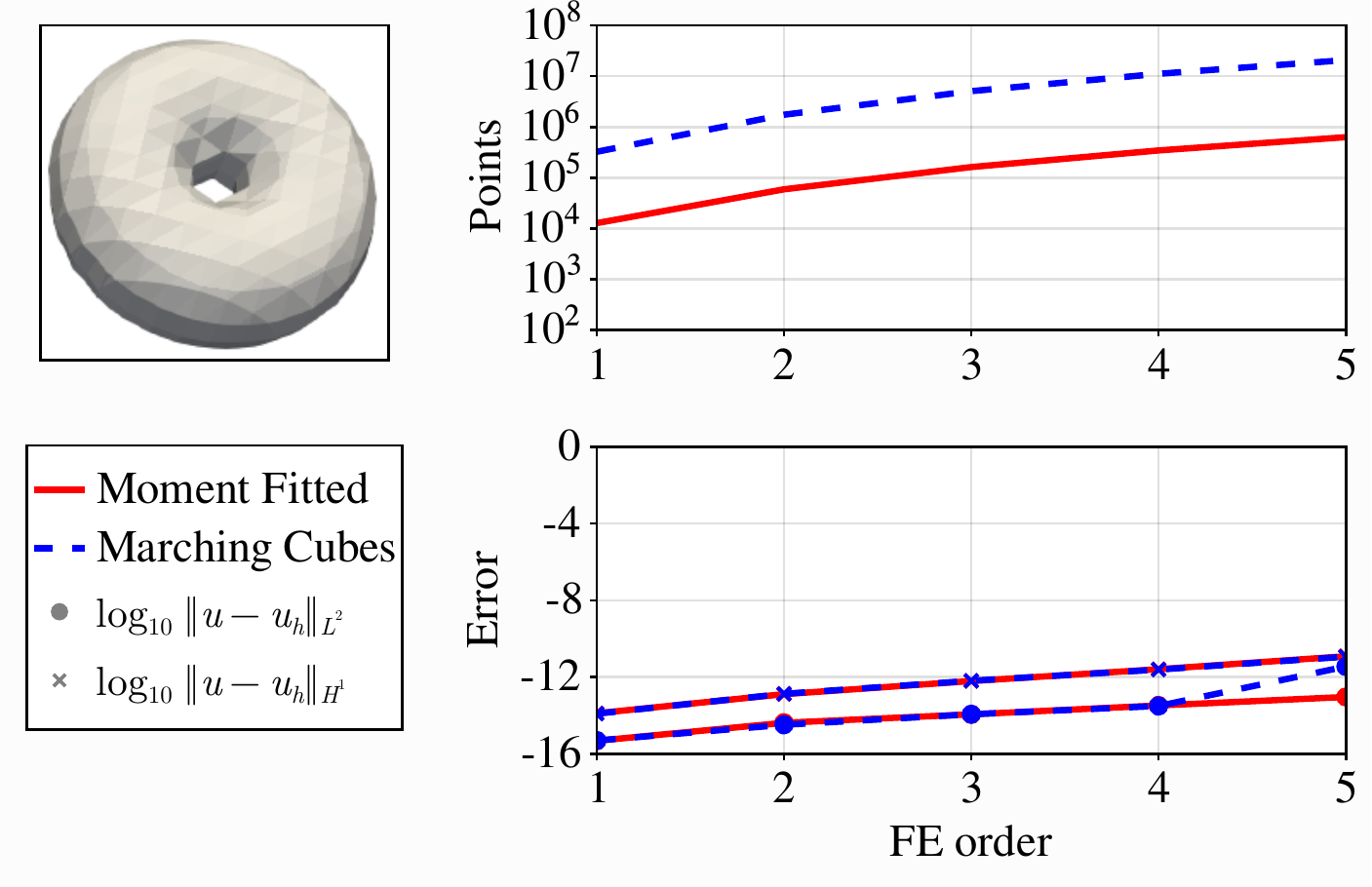}
    \caption{3D case}
  \end{subfigure}
  \caption{Comparison of moment-fitted and marching-cubes cut quadratures}
  \label{fig:1qc}
\end{figure}

\subsection{Static condensation of Ag\ac{fe} linear systems}\label{sec:stat_cond}

As usual in high-order \ac{fem}, we solve for the Schur complement system associated to the discrete problem, instead of the full linear system. We recall that the Schur complement is a reduced system, obtained by global assembly of (element-wise) statically condensed local FE matrices and vectors. In the static condensation, we eliminate internal/bubble \acp{dof} from the local system, since they are only supported in the interior of the element. The reduced system requires significantly less memory storage, at the expenses of the computational cost to perform the static condensation and a worse sparsity pattern than the full system. But the price is generally worth paying, both for direct and iterative linear solvers, especially at orders above cubic~\cite{pardo2015impact}. Additionally, in modal expansions, decoupling interior and boundary modes is key to improve condition numbers and how they scale~\cite{babuvska1991efficient,karniadakis2013spectral}. For instance, in 2D modal $\mathcal{C}^0$ expansions, the condition number of the full stiffness matrix scales as $m^4$, whereas the Schur complement system scales (at most) as $m \log(m)$~\cite{casarin1996schwarz}.

However, the restriction of active \ac{fe} spaces $\mathcal{V}_h^{\mathrm{act}}$ into aggregated ones $\mathcal{V}_h^{\mathrm{ag}}$ destroys the natural decoupling into interior and boundary \acp{dof} at the (nontrivial) root interior cells. In particular, interior \acp{dof} belonging to different root cells become coupled, if they are constraining ill-posed \acp{dof} from the same cut cell. Hence, they cannot be locally removed from the linear system; they must be removed at the global level. For this reason, we advocate for computing the boundary-reduced Ag\ac{fem} matrices in two-stages. The first stage is local and eliminates all bubble well-posed \acp{dof} that do not have support in cut cells, i.e., they do not constrain any ill-posed \ac{dof}. This can be seen as "an incomplete" standard static condensation procedure. Next, we assemble the global (first-stage) Schur complement system. In the second stage, we can now eliminate from the global matrix the rest of well-posed bubble \acp{dof}, i.e., those that are constraining ill-posed \acp{dof}. We note that the performance overhead related to the global static condensation step becomes insignificant with decreasing mesh size. Indeed, \new{on} fine meshes, we expect the number of well-posed bubble \acp{dof} coupled via cut cells to be a lot smaller than the rest of (uncoupled) well-posed bubble \acp{dof}. Thus, in modal $\mathcal{C}^0$ expansions, the impact in efficiency is clearly compensated by the improved behaviour of the condition number.

Apart from the two-stage static condensation, we can also reduce boundary-to-interior couplings by resorting to \emph{trunk} or \emph{serendipity} spaces~\cite{arnold2011serendipity}. It is well-known that they have the same approximability properties as tensor-product ones, although they are slightly less accurate, because they span a smaller multivariate polynomial space. In our context, since (modal $\mathcal{C}^0$ and Lagrangian) trunk-spaces have no bubbles up to order $\leq 3$ in 2D and $\leq 5$ in 3D, we can circumvent the global stage of the static condensation at low-medium order of approximation. We run the numerical experiments with both trunk-space and tensor-product modal $\mathcal{C}^0$ and Lagrangian expansions. For modal $\mathcal{C}^0$, trunk-space expansions lead to better conditioned matrices, in agreement with~\cite{babuvska1989problem}. Conversely, the best results for Lagrangian are obtained with tensor-product expansions because, in contrast with trunk-space ones, we can easily locate them at the set of nodes which minimise their Lebesgue constant, i.e., the Fekete nodes. We recall that minimising the Lebesgue constant mitigates the Runge phenomenon and ill-conditioning affecting high-order Lagrangian polynomials~\cite{ern2021finite}.

We conclude the section with a remark concerning 3D modal $\mathcal{C}^0$ bases. In contrast with 2D bases, eliminating internal \acp{dof} alone is not enough to recover good condition number estimates~\cite{karniadakis2013spectral}. In this case, better results are obtained with low-energy preconditioners, as in~\cite{babuvska1991efficient,sherwin2001low}. However, they are not easy to implement in Ag\ac{fe} spaces, thus we do not cover them in this work.

\subsection{Experimental environment} 
\label{sub:exp_env}

All the algorithms have been implemented in the \texttt{Gridap} open-source scientific software project \cite{Badia2020Aug}. \texttt{Gridap} is a novel framework for the implementation of grid-based algorithms for the discretisation of \acp{pde} written in the \texttt{Julia} programming language. \texttt{Gridap} has a user interface that resembles the whiteboard mathematical statement of the problem. The framework leverages the \texttt{Julia} just-in-time (JIT) compiler to generate high-performant code \cite{verdugo2022software}. \texttt{Gridap} is extensible and modular and has many available plugins. In particular, we have extensively used and extended the \texttt{GridapEmbedded} plugin \cite{GridapEmbedded-jl}, which provides all the mesh queries required in the implementation of the embedded methods under consideration, level set surface descriptions and \ac{csg}. We use the \texttt{cond()} method provided by \texttt{Julia} to estimate condition numbers. Condition numbers have been estimated in the $1$-norm for efficiency reasons. It was not possible to compute the $2$-norm condition number for all cases with the available computational resources. On the other hand, eigenextrema in Section~\ref{sub:exp_eigen} are computed with \texttt{Arpack.jl}, a Julia wrapper for \texttt{Arpack}~\cite{lehoucq1998arpack}. Concerning the linear solver, we use a sparse direct solver from the \texttt{MKL PARDISO} package~\cite*{_intel_????}.

The numerical experiments have been carried out at NCI-Gadi, 
hosted by the Australian National Computational Infrastructure Agency (NCI), and the Marenostrum-IV (MN-IV) supercomputer,  
hosted by the Barcelona Supercomputing Centre. NCI-Gadi is a petascale machine with 3,024 nodes, each containing 2x 24-core Intel Xeon Scalable \textit{Cascade Lake} processors and 192 GB of RAM. All nodes are interconnected via Mellanox Technologies' latest generation HDR InfiniBand technology. MN-IV is a petascale machine equipped with 3,456 compute nodes interconnected with the Intel OPA HPC network. Each node has 2x Intel Xeon Platinum 8160 multi-core CPUs, with 24 cores each (i.e., 48 cores per node) and from 96 to 384 GB of RAM.

\subsection{Eigenextrema convergence tests}
\label{sub:exp_eigen}

The goal of our first experiment is to evaluate the 2D spectral behaviour of the two-stage Schur complement of Ag\ac{fe} matrices corresponding to the Laplacian operator. To this end, we numerically assess the scaling of the maximum and minimum eigenvalues with both mesh size $h$ and order of approximation $m$. We centre upon the sensitivity of the eigenspectrum to the discrete extension and the comparison against the Schur complement of body-fitted \ac{fe} system matrices.

According to this, we solve the Poisson problem~(\ref{eq:poisson-strong}) with the discrete approximation~(\ref{eq:poisson-weak}) in $\mathcal{V}_h^{\rm ag}$ derived from tensor-product Lagrangian and trunk-space modal $\mathcal{C}^0$ Ag\ac{fem} bases. We consider the perturbed square $\Omega = [0,1+\varepsilon] \times [0,1]$, $0 < \varepsilon < 1$, in the artifical domain $\Omega_h^\mathrm{art} = [0,2] \times [0,1]$. The background mesh is a uniform cartesian grid of $\Omega_h^\mathrm{art}$; thus, the face given by $[0,1] \times \{1+\varepsilon\}$ is embedded in the mesh. We represent this geometry setup in Fig.~\ref{fig:eigenextremasetup}. In contrast with the rest of numerical experiments with the Poisson equation, we apply Neumann boundary conditions on the unfitted facet and strong Dirichlet boundary conditions elsewhere.

\begin{figure}[ht!]
  \centering
  \begin{subfigure}{0.24\textwidth}
    \centering
    \includegraphics[width=0.95\textwidth]{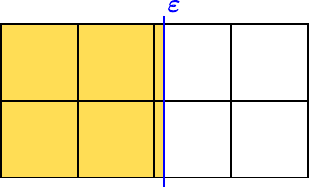}
    \caption{\emph{almost empty} case}
    \label{fig:eigenextremasetup_a}
  \end{subfigure}
  \begin{subfigure}{0.24\textwidth}
    \centering
    \includegraphics[width=0.95\textwidth]{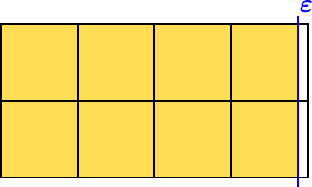}
    \caption{\emph{almost full} case}
    \label{fig:eigenextremasetup_b}
  \end{subfigure}
  \caption{Eigenextrema test: Embedded geometry setup.}
  \label{fig:eigenextremasetup}
\end{figure}

We study the convergence rates of the eigenextrema considering two different extreme cases: \emph{almost empty} or \emph{almost full} cut cells. In the \emph{almost empty} case, $\varepsilon = 10^{-8}$, i.e., the physical cuts are extremely thin slivers and it is apparent that the unfitted problem in $\Omega$ is a small perturbation of a fitted one in $[0,1]^2$. In particular, the unfitted system matrices should be almost identical to their body-fitted analogues, because the external shape functions have very small support in the physical domain. As a result, they barely change the linear system, when their degrees of freedom are removed via discrete extension. We confirm this in the numerical results. As shown in Fig.~\ref{fig:eigenextrematestmin}, we successfully recover theoretical asymptotic rates corresponding to the Schur complement of 2D body-fitted \ac{fe} stiffness matrices~\cite{casarin1996schwarz}: The maximum eigenvalue $\lambda_{\max}$ has a constant scaling with $h$ and $m$. The minimum eigenvalue $\lambda_{\min}$ scales as $h^{-2}$ and $(m \log{m})^{-1} \lesssim \lambda_{\min} \lesssim m^{-1} \log{m}$. We have also checked that the unfitted eigenvalues almost coincide with the body-fitted ones, there is just a small difference because the Schur complements are not computed in the same way. We obtain the same values, though, if we solve the full matrices. Furthermore, we stress that, in the static condensation of Ag\ac{fe} matrices, it is essential to remove all rows and columns corresponding to bubble \acp{dof} from the linear system, i.e., to carry out the second stage of the static condensation. With first-stage (local) static condensation alone we can only recover $\lambda_{\min} \sim h^{-2}m^{-4}$, which is the spectral behaviour of noncondensed body-fitted \ac{fe} matrices~\cite{hu1998bounds}. Finally, in Fig.~\ref{fig:eigenbboxes}, we also see that it is crucial to adjust the $\mathcal{B}^{\mathrm{mod}}$ bounding boxes to the physical domain $\Omega$, instead of to the active domain $\Omega_h^{\rm act}$. This circumvents very poor scaling of the minimum eigenvalue. We have plotted the basis functions with the $\mathcal{B}^{\mathrm{mod}}$ adjusted to $\Omega_h^{\rm act}$ and detected upon visual inspection that linear dependence is the cause for such ill-conditioning.

The \emph{almost full} case is a reversal of the previous one. Here, $\varepsilon = 1.0 - 10^{-8}$, i.e., cut portions differ from the full cell only by a thin sliver. This is the worst scenario to evaluate the effects of the discrete extension, because it maximises the support of external shape functions in the physical domain. Thus, their weight in the linear system, via the discrete extension, is also maximised. Fig.~\ref{fig:eigenextrematestmax} clearly exposes the superiority of modal $\mathcal{C}^0$ \acp{fe} in this scenario with respect to the Lagrangian \acp{fe}. Indeed, while the maximum eigenvalue blows up exponentially (in the log-log plane) for Lagrangian Ag\acp{fe}, modal $\mathcal{C}^0$ Ag\acp{fe} exhibit a significantly smoother (almost linear) growth. Even though modal $\mathcal{C}^0$ Ag\acp{fe} do not recover the eigenvalue asymptotic behaviour of the \emph{almost empty} case, the controlled growth of the maximum eigenvalue, in the \emph{almost full} case, clearly compensates for circumventing the small cut cell problem, in the \emph{almost empty} one. In this sense, we point out that the minimum eigenvalues barely change between the body-fitted, \emph{almost empty} and \emph{almost full} cases. We recall that, if no aggregation is carried out, the smallest eigenvalue of the system corresponds to a function which is only supported on a cut cell~\cite{DePrenter2017}. Since the results are independent of $\varepsilon$, we readily verify that aggregation is indeed robustly circumventing the small cut-cell problem.

\begin{figure}[ht!]
  \centering
  \begin{subfigure}{\textwidth}
    \centering
    \includegraphics[width=0.95\textwidth]{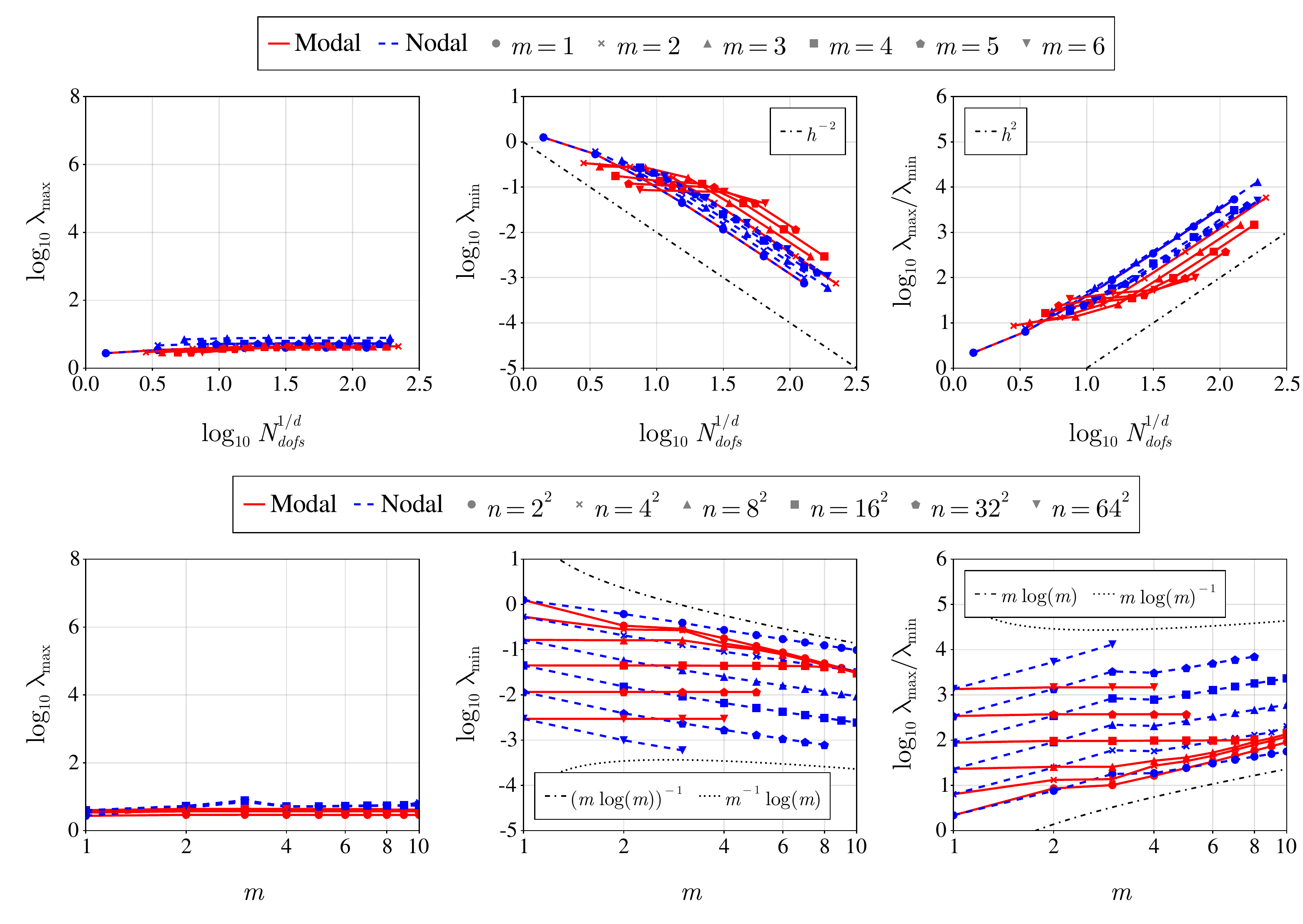}
    \caption{2D \emph{almost empty} case}
    \label{fig:eigenextrematestmin}
  \end{subfigure} \\ \vspace{0.25cm}
  \begin{subfigure}{\textwidth}
    \centering
    \includegraphics[width=0.95\textwidth]{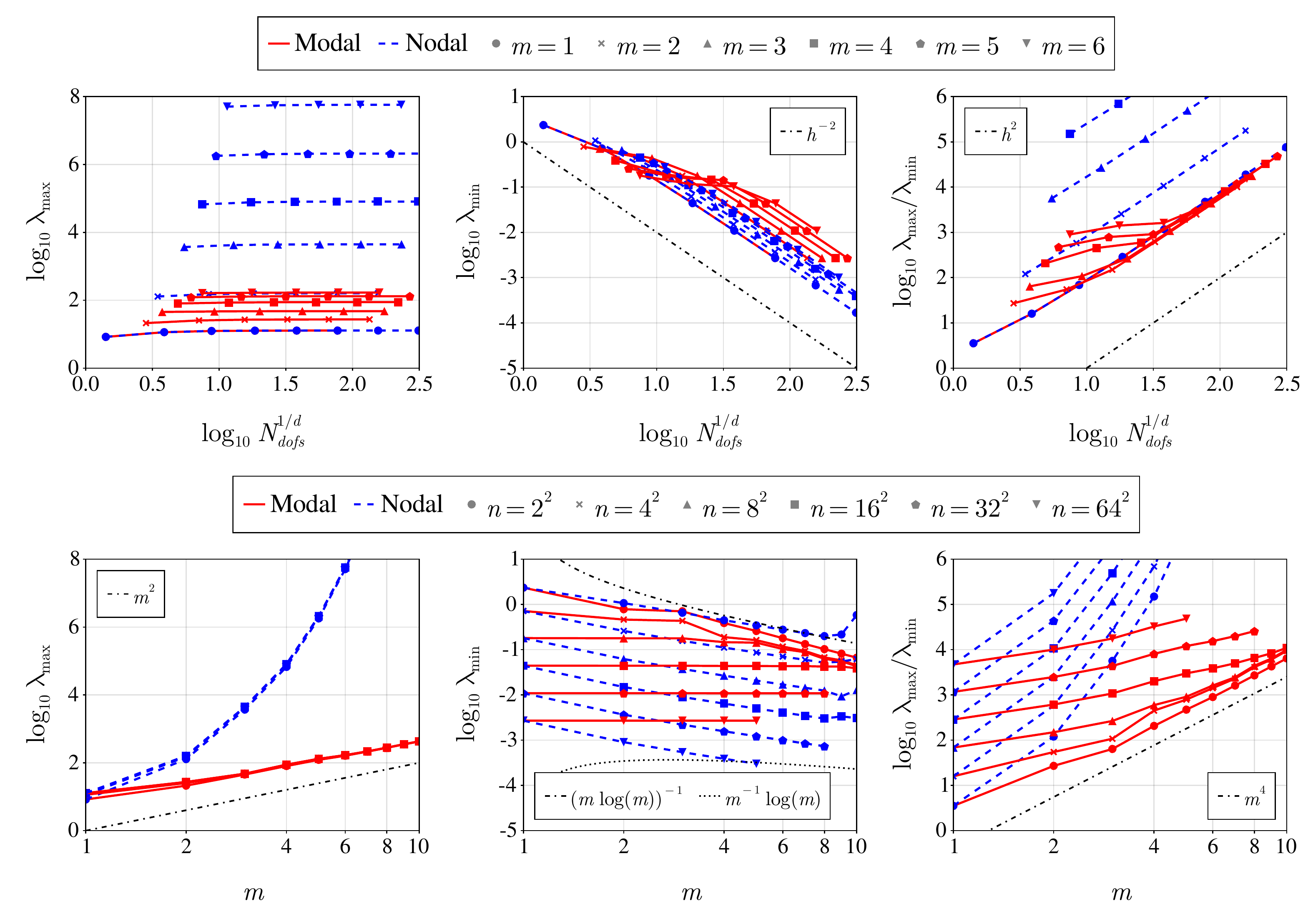}
    \caption{2D \emph{almost full} case}
    \label{fig:eigenextrematestmax}
  \end{subfigure}
  \caption{Eigenextrema test: Convergence rates with mesh size and order $m$}
  \label{fig:eigenextrematest}
\end{figure}

\begin{figure}[ht!]
  \centering
  \includegraphics[width=0.95\textwidth]{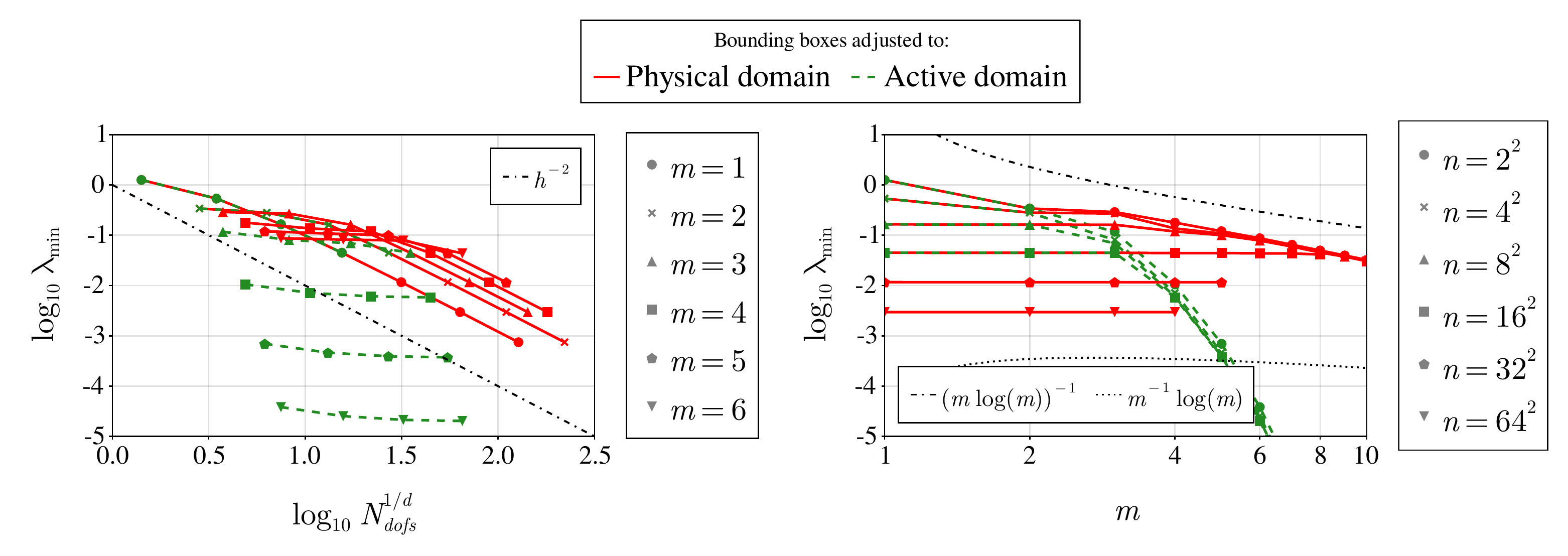}
  \caption{Eigenextrema test: 2D \emph{almost empty} case. Comparison of modal $\mathcal{C}^0$ bases with different definitions of $\mathcal{B}^{\mathrm{mod}}$: either adjusted to the physical domain $\Omega$ or to the active domain $\Omega_h^{\mathrm{act}}$. See also Remark~\ref{rem:bboxes}.}
  \label{fig:eigenbboxes}
\end{figure}

\subsection{Extrapolation distance tests}
\label{sub:exp_extra}

In this test, we continue comparing high-order Lagrangian Ag\ac{fem} against the (new) modal $\mathcal{C}^0$ counterpart. In particular, we offer deeper insight as to how the methods behave as the distance, at which we extrapolate the root shape functions, grows. To this end, we adapt the example in Fig.~\ref{fig:cn2d} with the elasticity problem~(\ref{eq:elasticity-weak}) approximated in $\mathcal{V}_h^{\mathrm{ag}}$, i.e., with Strong Ag\ac{fem}. We solve for the \ac{fe}-space case in the unit $d$-cube, $d=2$ or 3. The background mesh takes 8 cells in each direction and the embedded geometry, represented in Fig.~\ref{fig:extrapolationsetup}, is an $l$-parameterised polytope defined by the vertices $[(0,0),$ $(1/8,0),$ $(l,l),$ $(0,1/8)]$, in 2D, and $[(0,0,0),$ $(1/8,0,0),$ $(1/8,1/8,0),$ $(0,1/8,0),$ $(0,0,1/8),$ $(1/8,0,1/8),$ $(l,l,l),$ $(0,1/8,1/8)]$, in 3D. If $l = 1/8$, the polytope is exactly the background cell located at the origin and there are no cut cells. If we let $l > 1/8$, then the polytope cuts cells next to the diagonal $x = y$ or $x = y = z$. In this case, all active cells are cut, except for the one touching the origin. Therefore, cell aggregation is trivial. In particular, there is a single aggregate \emph{ag} composed by the \emph{root} cell at the origin and all the remaining cut cells. The \ac{aabb} of the aggregate is $[0,l]^d$, thus the (relative) maximum extrapolation distance is $\mathrm{diam(ag)}/\mathrm{diam(root)} = l / (1/8) = 8l$. This value is much higher than one would expect if the background mesh has enough resolution to capture the geometric features of $\Omega$, but it is used to stress the framework. The experiment consists in moving $l$ in the range $[1/8,1]$, solve the problem for each $l$, and compute the $L^2$-error, the $H^1$-error and the condition number of the Schur complement system matrix. In Fig.~\ref{fig:extrapolationtest} we plot these three quantities along the $\mathrm{diam(ag)}/\mathrm{diam(root)}$ for several orders of approximation $m \geq 2$. For brevity, we only report the 3D results, since the 2D ones are very similar. We observe results with the same pattern as the ones observed in Fig.~\ref{fig:cn2d}. Indeed, (full extrapolation) Lagrangian \acp{fe} suffer from severe loss of accuracy and ill-conditioning for $m > 2$ and long extrapolation. Plus, the higher the order of approximation $m$, the faster the degradation. \new{For $m=4$ and $m=5$ and $\mathrm{diam(ag)}/\mathrm{diam(root)}$ larger than 2}, the solution is completely wrong and the condition number estimates unreliable reliable. On the other hand, generalised modal $\mathcal{C}^0$ \acp{fe} are robust and deteriorate with growing extrapolation distance at a much lower rate, which is practically independent of the approximation order.

\begin{figure}[ht!]
  \centering
  \begin{subfigure}{0.29\textwidth}
    \centering
    \includegraphics[height=3.5cm]{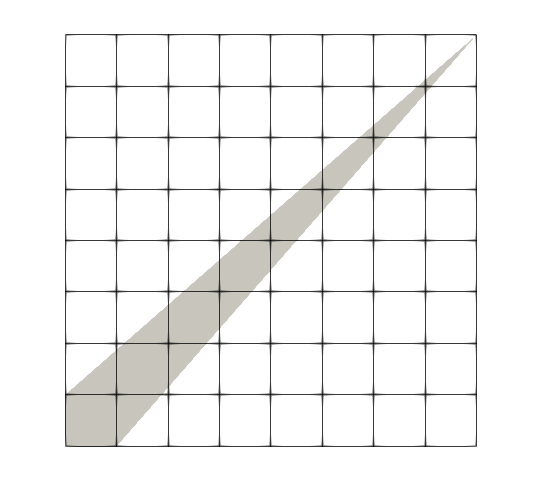}
    \caption{2D with $l=1$}
  \end{subfigure}
  \begin{subfigure}{0.29\textwidth}
    \centering
    \includegraphics[height=3.5cm]{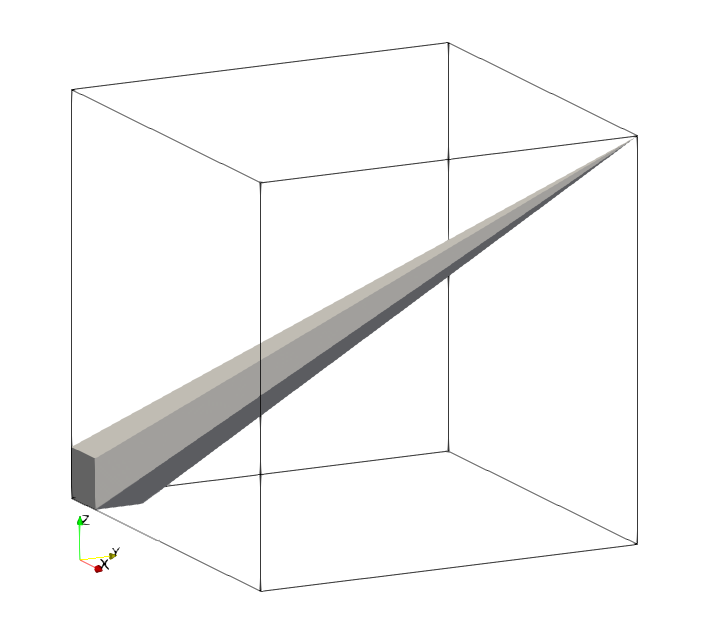}
    \caption{3D with $l=1$}
  \end{subfigure}
  \caption{Embedded geometry in the extrapolation tests of Section~\ref{sub:exp_extra}: an $l$-parameterised polytope enclosed, in 2D, by $[(0,0), (1/8,0), (l,l), (0,1/8)]$, and, in 3D, by $[(0,0,0),(1/8,0,0),(1/8,1/8,0),(0,1/8,0),(0,0,1/8),(1/8,0,1/8),(l,l,l),(0,1/8,1/8)]$.}
  \label{fig:extrapolationsetup}
\end{figure}

\begin{figure}[ht!]
  \centering
    \includegraphics[width=0.9\textwidth]{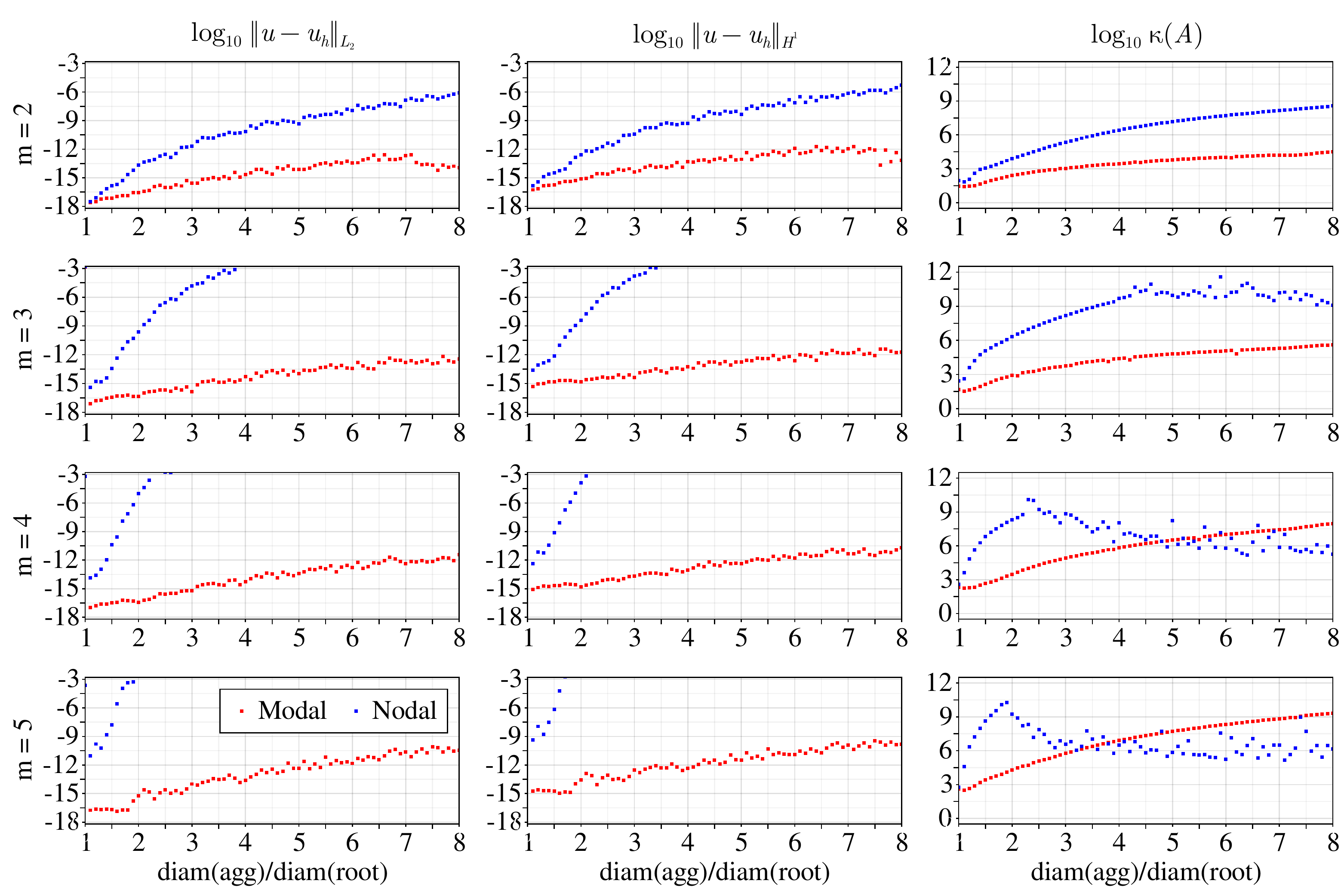}
  \caption{3D extrapolation distance test results for the elasticity problem~(\ref{eq:elasticity-weak}).}
  \label{fig:extrapolationtest}
\end{figure}

\subsection{Convergence tests: nodal (Lagrangian) and modal $\mathcal{C}^0$ methods}
\label{sub:exp_conv_1}

Here, we consider standard convergence tests on the six embedded geometries with the \emph{out-\ac{fe}-space} case. Thus, we study how the $L^2$-error, the $H^1$-error and the condition number of the Schur complement system matrix, associated to~(\ref{eq:poisson-weak}) and~(\ref{eq:elasticity-weak}), behave to uniform mesh refinements. Fig.~\ref{fig:convtest2D} and~\ref{fig:convtest3D} gathers the plots of the convergence tests. We filter results where the direct solver fails to obtain an accurate solution and we have also not computed condition numbers of 3D matrices of order $m > 3$. We observe that, while both methods are optimal, modal $\mathcal{C}^0$ Ag\ac{fem} has a clear superiority in terms of robustness and conditioning, when the order of approximation $m$ is increased, especially in 3D. Indeed, condition numbers of modal $\mathcal{C}^0$ Ag\ac{fem} discretizations fall between 1 and 3 orders of magnitude below the Lagrangian counterparts.

\begin{figure}[ht!]
  \centering
  \begin{subfigure}{\textwidth}
    \centering
    \includegraphics[width=0.9\textwidth]{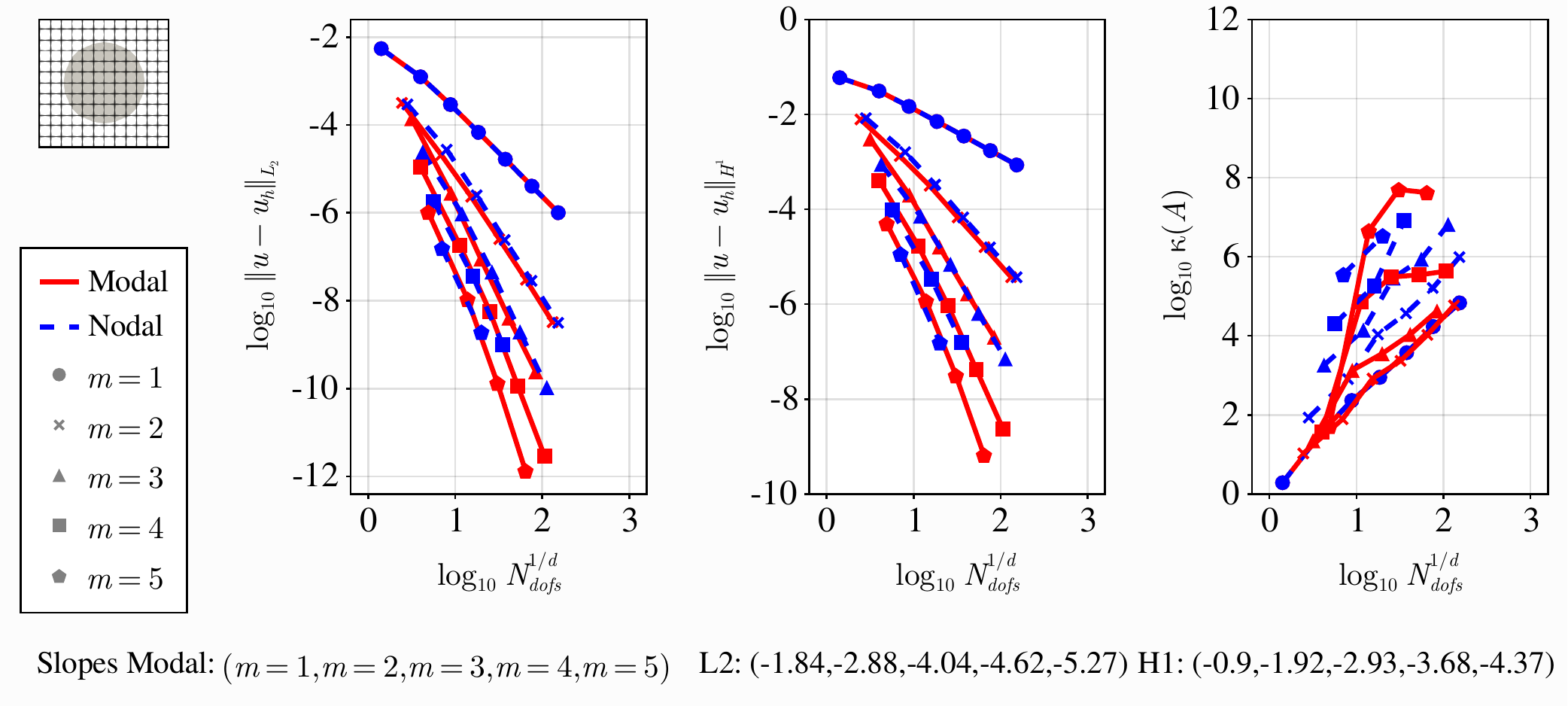}
    \caption{Linear elasticity problem~(\ref{eq:elasticity-weak}) on disk}
  \end{subfigure} \\ \vspace{0.5cm}
  \begin{subfigure}{\textwidth}
    \centering
    \includegraphics[width=0.9\textwidth]{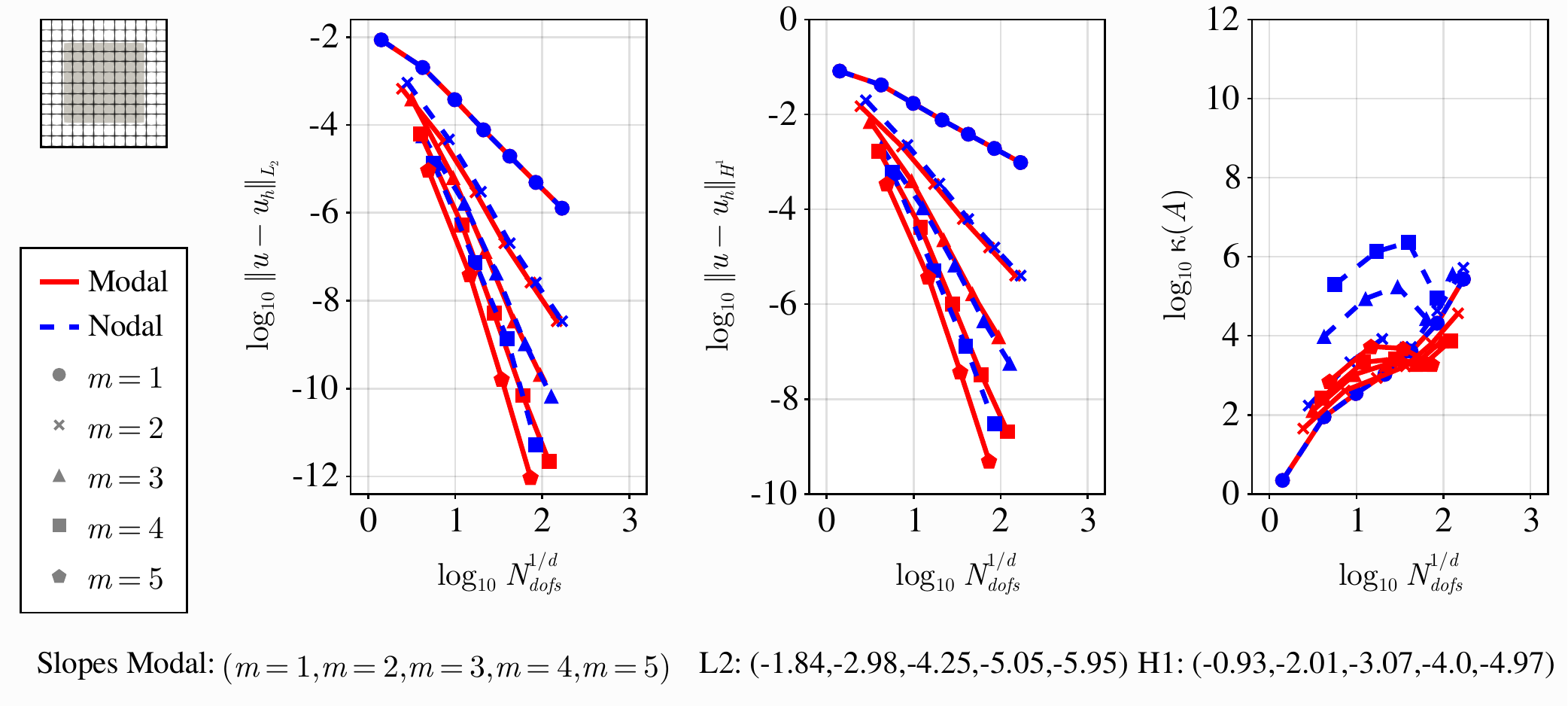}
    \caption{Linear elasticity problem~(\ref{eq:elasticity-weak}) on square}
  \end{subfigure} \\ \vspace{0.5cm}
  \begin{subfigure}{\textwidth}
    \centering
    \includegraphics[width=0.9\textwidth]{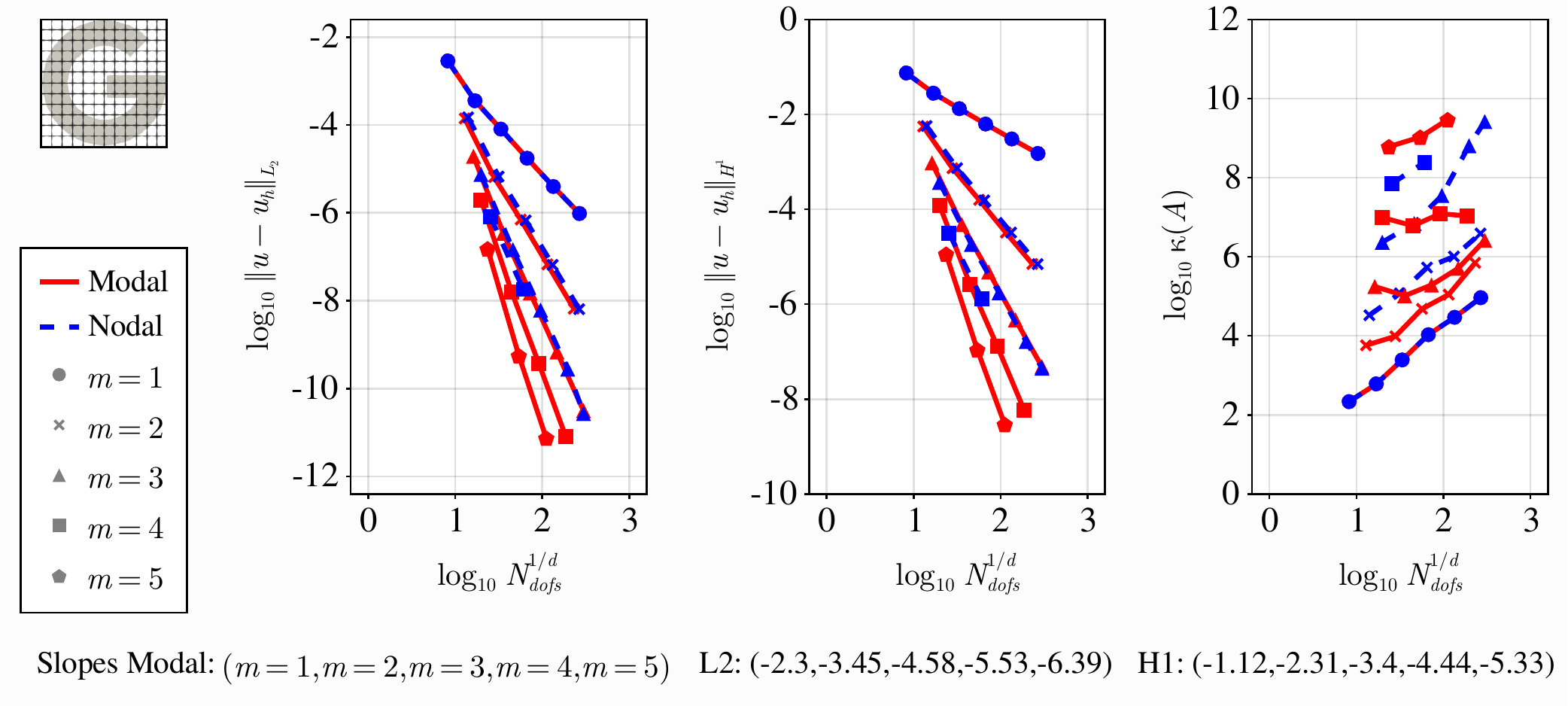}
    \caption{Poisson problem~(\ref{eq:poisson-weak}) on letter G}
  \end{subfigure}
  \caption{Selected 2D convergence tests. Lagrangian vs modal $\mathcal{C}^0$ Ag\ac{fem}.}
  \label{fig:convtest2D}
\end{figure}

\begin{figure}[ht!]
  \centering
  \begin{subfigure}{\textwidth}
    \centering
    \includegraphics[width=0.9\textwidth]{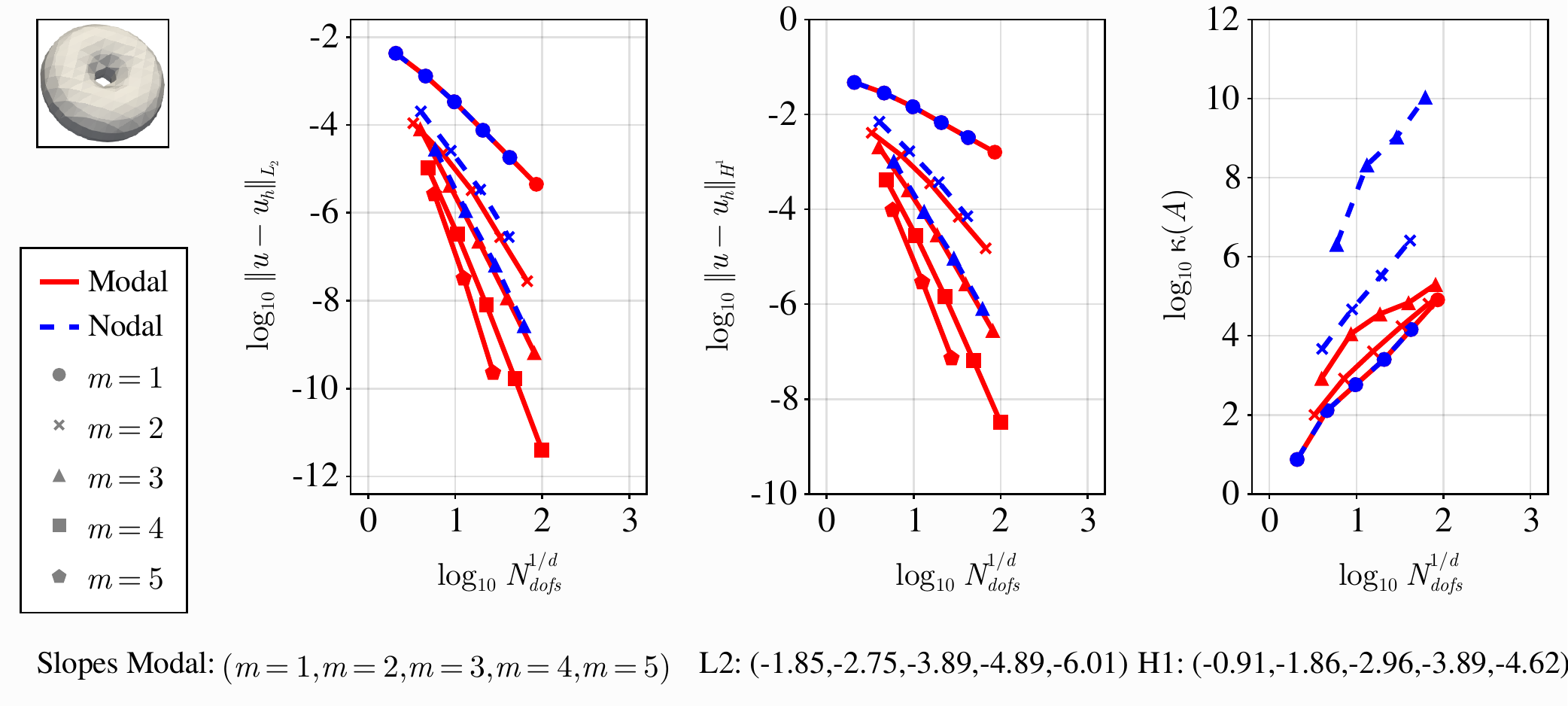}
    \caption{Linear elasticity problem~(\ref{eq:elasticity-weak}) on torus}
  \end{subfigure} \\ \vspace{0.5cm}
  \begin{subfigure}{\textwidth}
    \centering
    \includegraphics[width=0.9\textwidth]{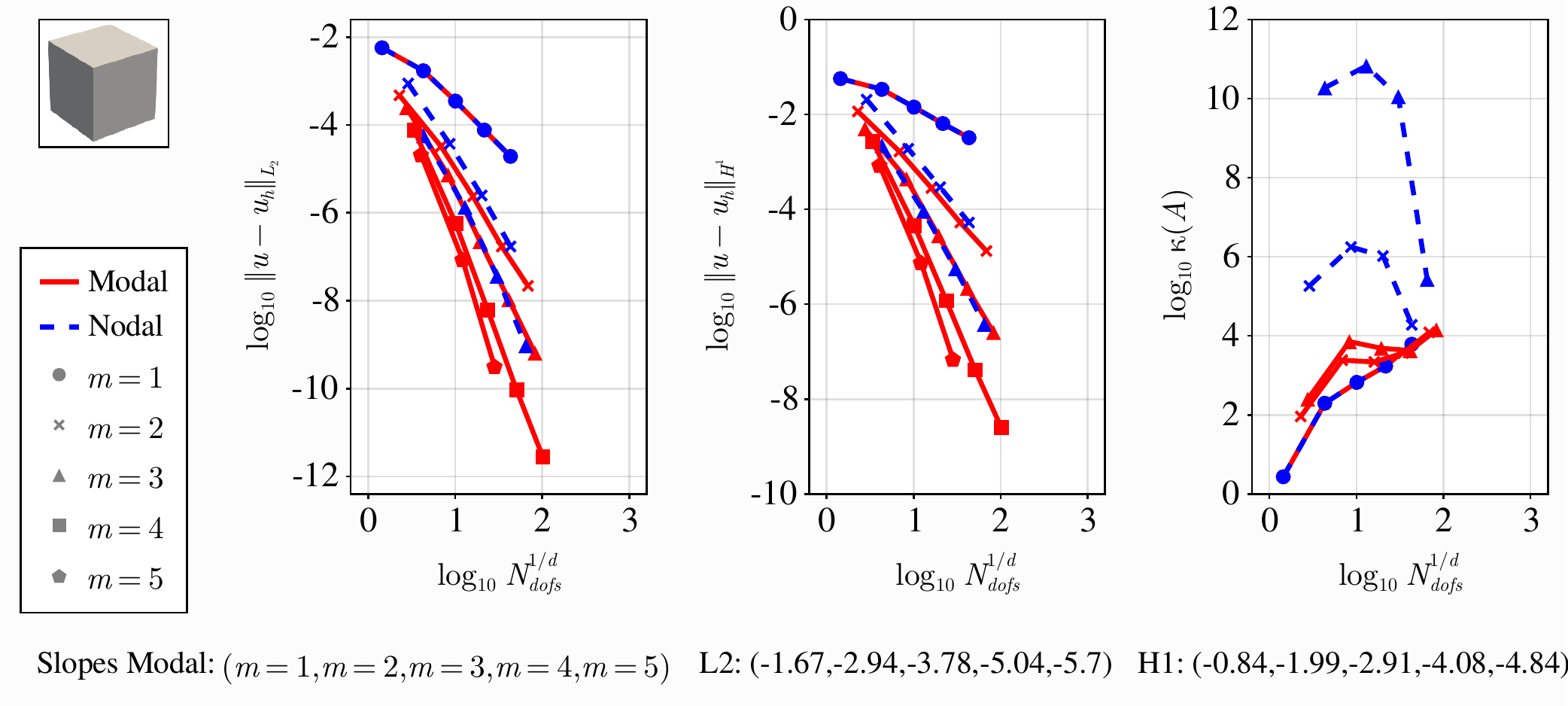}
    \caption{Linear elasticity problem~(\ref{eq:elasticity-weak}) on cube}
  \end{subfigure} \\ \vspace{0.5cm}
  \begin{subfigure}{\textwidth}
    \centering
    \includegraphics[width=0.9\textwidth]{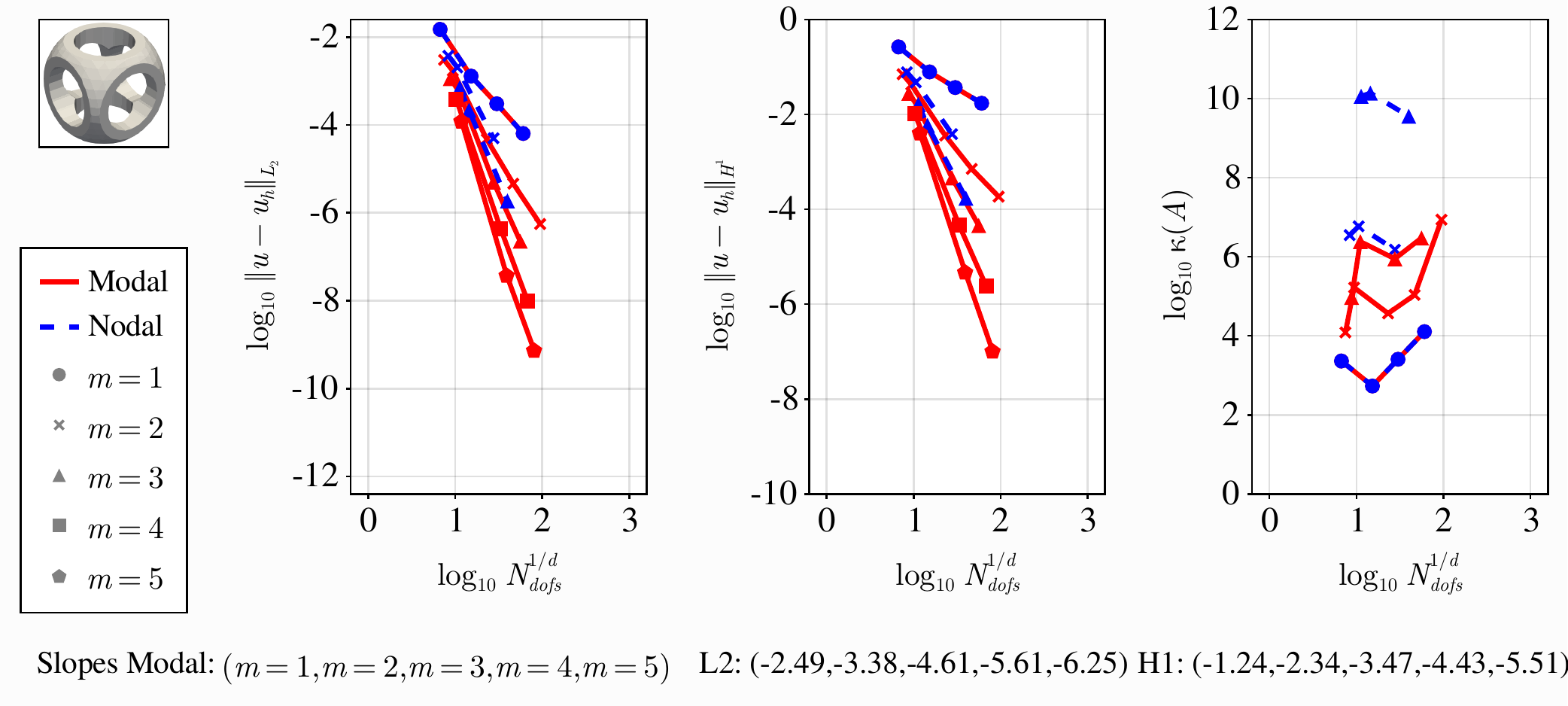}
    \caption{Poisson problem~(\ref{eq:poisson-weak}) on CSG}
  \end{subfigure}
  \caption{Selected 3D convergence tests. Lagrangian vs modal $\mathcal{C}^0$ Ag\ac{fem}.}
  \label{fig:convtest3D}
\end{figure}

\subsection{Sensitivity to cut location}
\label{sub:exp_loc}

Finally, we evaluate the robustness to cut location. We consider the \emph{in-\ac{fe}-space} solution. Instead of refining the mesh, as in the convergence tests, now we fix the cell size $h$ of the background grid and we perturb the position of the embedded geometry. In particular, we apply translations of vector $(t,t)$ (2D) and $(t,t,t)$ (3D), where $t$ is a sliding parameter in the interval $(-h,h)$. Fig.~\ref{fig:boxplots} reports problem~(\ref{eq:poisson-strong}) on the square for a 48x48 grid and problem~(\ref{eq:elasticity-strong}) on the torus for a 12x12x12 grid, other cases yield analogous results. We represent boxplots of the $L^2$-error, the $H^1$-error and the condition number of the Schur complement system matrix for the set of values obtained by sliding $t \in (-h,h)$. The numerical results complement the ones of the convergence tests. Indeed, with respect to accuracy and conditioning, we deduce that the outcomes reported in Sections~\ref{sub:exp_conv_1} are independent of the cut location.

\begin{figure}[ht!]
  \centering
  \begin{subfigure}{\textwidth}
    \centering
    \includegraphics[width=0.85\textwidth]{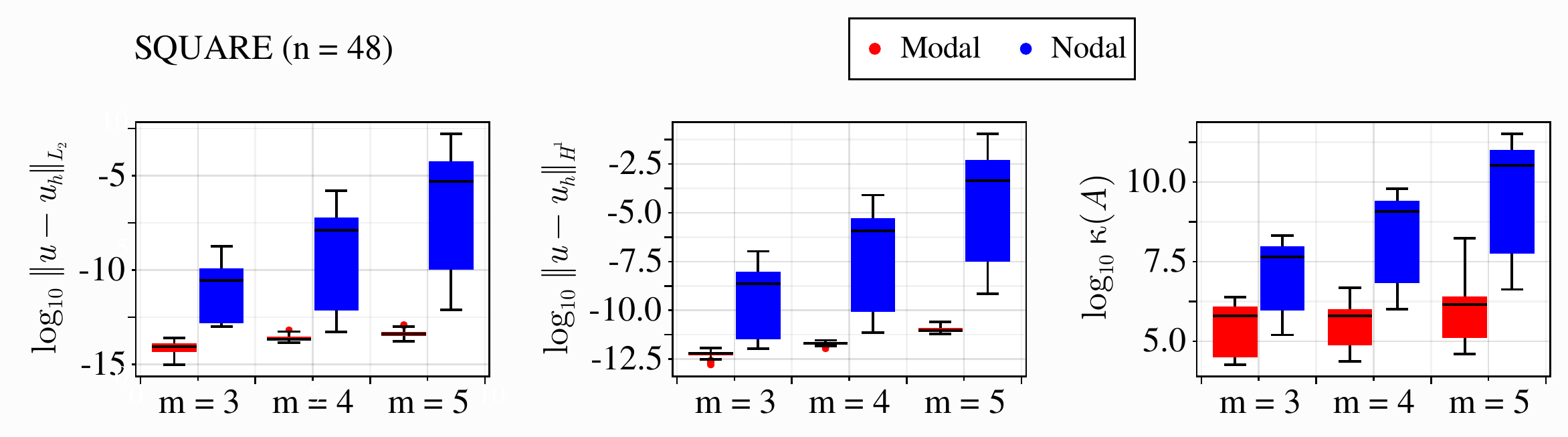}
    \caption{Poisson problem~(\ref{eq:poisson-weak}) on square and $m = 3,4,5$}
  \end{subfigure} \\ \vspace{0.25cm}
  \begin{subfigure}{\textwidth}
    \centering
    \includegraphics[width=0.85\textwidth]{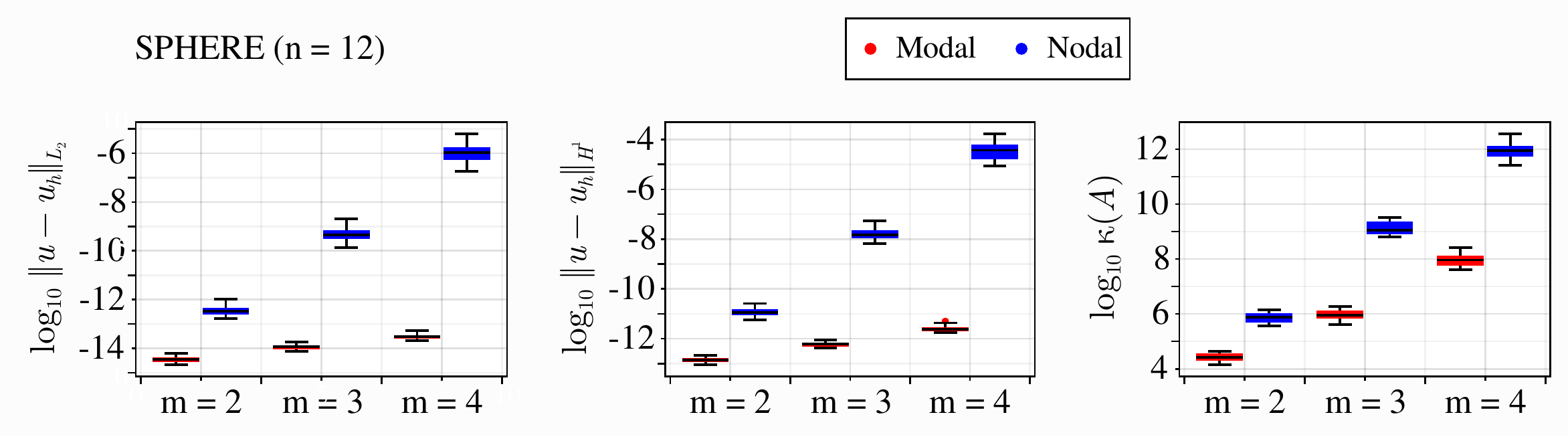}
    \caption{Poisson problem~(\ref{eq:poisson-weak}) on torus  and $m = 2,3,4$}
  \end{subfigure}
  \caption{Sensitivity to location tests.}
  \label{fig:boxplots}
\end{figure}

\section{Conclusions}
\label{sec:conclusions}

In this work, we introduce a new formulation for robust high-order unfitted finite elements by cell aggregation. The method is grounded on the discrete extension operator proposed in the \ac{agfem}, which is suitable for \ac{cg} methods. \acp{agfem} are conceptually attractive for high-order approximations. Their theoretical stability and convergence are independent of the order of approximation and they do not rely on penalising high-order jump derivatives, as ghost penalty methods. However, not all \ac{fe} bases are suitable for high-order discrete extensions. For instance, discrete extensions based on Lagrangian \acp{fe} are pure extrapolations. It follows that the aggregation constraint coefficients scale $m$-exponentially with the aggregate size, where $m$ is the order of approximation. In high order, this leads to huge constants in the continuity and stability estimates, which imply the method becomes prone to ill-conditioning.

In order to mitigate this issue, we propose a novel Ag\ac{fem} formulation, grounded on a generalisation of modal $\mathcal{C}^0$ \ac{fe} bases. We exploit the structure of modal $\mathcal{C}^0$ bubble functions, formed by the product of nodal modes and a high-order polynomial term, by stretching the latter term onto suitable (domain interior) aggregate bounding boxes. The resulting \ac{fe} basis accommodates to the structure and properties of Ag\ac{fem}, but the discrete extension operator is no longer a pure extrapolation; in particular, it is an interpolation in the physical domain for polynomials of order higher than two.

We carried out an extensive numerical experimentation up to order five and 3D, on elliptic boundary value problems, a myriad of embedded geometries and considering both tensor-product and serendipity variants of the \ac{fe} bases. Therein, we demonstrate a clear superiority of high-order modal $\mathcal{C}^0$ \ac{agfem} w.r.t.~the Lagrangian counterpart, in terms of robustness and sensitivity to cut location and aggregate size. Therefore, modal $\mathcal{C}^0$ bases are a much better alternative for high-order Ag\ac{fem} applications. We note that we have restricted the study to \emph{strong} \ac{agfem}, but we expect analogous outcomes with \emph{weak} versions of \ac{agfem} with modal $\mathcal{C}^0$ bases. Other than this, this work also showcases the potential to improve \ac{agfem}, by reusing tools that are consolidated for body-fitted \ac{fem}. Indeed, modal $\mathcal{C}^0$ \acp{fe} are long established and specialised for high-order body-fitted \ac{cg} \ac{fem} and here we show \new{that} they are also amenable for high-order unfitted aggregated methods.